\documentclass[12pt]{amsart}
\usepackage[margin=1in]{geometry}                
\usepackage{graphicx}
\usepackage{amssymb}
\usepackage{amsthm}
\usepackage{mathrsfs}
\usepackage{caption}
\usepackage{xcolor}

\newtheorem{theorem}{Theorem}[section]
\newtheorem{proposition}[theorem]{Proposition}
\newtheorem{lemma}[theorem]{Lemma}
\newtheorem{definition}[theorem]{Definition}
\newtheorem{corollary}[theorem]{Corollary}

\newtheorem{remark}{Remark}[section]

\numberwithin{equation}{section}

\newcommand{\ZZ}{\mathbb{Z}}
\newcommand{\RR}{\mathbb{R}}
\newcommand{\NN}{\mathbb{N}}

\newcommand{\bbA}{\mathbb{A}}
\newcommand{\bbB}{\mathbb{B}}

\newcommand{\bbI}{\mathbb{I}}

\newcommand{\bbJ}{\mathbb{J}}
\newcommand{\bbK}{\mathbb{K}}

\newcommand{\calx}{\mathcal{X}}

\newcommand{\calt}{\mathcal{T}}
\newcommand{\cali}{\mathcal{I}}
\newcommand{\calj}{\mathcal{J}}
\newcommand{\calk}{\mathcal{K}}

\newcommand{\calm}{\mathcal{M}}
\newcommand{\call}{\mathcal{L}}

\newcommand{\Div}{{\mathrm{Div}}}
\newcommand{\Span}{{\mathrm{Span}}}
\newcommand{\Bispan}{{\mathrm{Bispan}}}


\title{The Coven-Meyerowitz tiling conditions for 3 odd prime factors} 
\author{Izabella {\L}aba and Itay Londner}	
	
\date{Revised version, \today}

\begin{document}

\begin{abstract}
It is well known that if a finite set $A\subset\ZZ$ tiles the integers by translations, then the translation set must be periodic, so that the tiling is equivalent to a factorization $A\oplus B=\ZZ_M$ of a finite cyclic group. We are interested in characterizing all finite sets $A\subset\ZZ$ that have this property.
Coven and Meyerowitz \cite{CM} proposed conditions (T1), (T2) that are sufficient for $A$ to tile, and necessary when the cardinality of $A$ has at most two distinct prime factors. They also proved that (T1) holds for all finite tiles, regardless of size. It is not known whether (T2) must hold for all tilings with no restrictions on the number of prime factors of $|A|$.

We prove that the Coven-Meyerowitz tiling condition (T2) holds for all integer tilings of period $M=(p_ip_jp_k)^2$, where
$p_i,p_j,p_k$ are distinct odd primes. The proof also provides a classification of all such tilings.
\end{abstract}

\maketitle

\tableofcontents


\section{Introduction}


We say that a set $A\subset\ZZ$ {\em tiles the integers by translations} 
if there is a set $T\subset\ZZ$ such that every integer $n$ can be represented uniquely
as $n=a+t$ with $a\in A$ and $t\in T$. Throughout this
article, we assume that $A$ is finite.  It is well known (see
\cite{New}) that any tiling of $\ZZ$ by a finite set $A$ must be periodic, i.e.
$T=B\oplus M\ZZ$ for some finite set $B\subset \ZZ$ such that $|A|\,|B|=M$.  Equivalently,
$A\oplus B$ is a factorization of the cyclic group $\ZZ_M$.

We are interested in determining which finite sets $A\subset\ZZ$ have this property, and, in particular,
in a characterization proposed by Coven and Meyerowitz \cite{CM}. In order to state their conditions, we need to introduce some notation.
By translational invariance, we may assume that
$A,B\subset\{0,1,\dots\}$ and that $0\in A\cap B$. The {\it characteristic 
polynomials} (also known as {\it mask polynomials}) of $A$ and $B$ are
$$
A(X)=\sum_{a\in A}X^a,\ B(x)=\sum_{b\in B}X^b .
$$
Then the tiling condition $A\oplus B=\ZZ_M$ is equivalent to
\begin{equation}\label{poly-e1}
 A(X)B(X)=1+X+\dots+X^{M-1}\ \mod (X^M-1).
\end{equation}

Let $\Phi_s(X)$ be the $s$-th cyclotomic
polynomial, i.e., the unique monic, irreducible polynomial whose roots are the
primitive $s$-th roots of unity. Alternatively, $\Phi_s$ can be defined inductively via the identity
$$
X^n-1=\prod_{s|n}\Phi_s(X).
$$
In particular, (\ref{poly-e1}) is equivalent to
\begin{equation}\label{poly-e2}
  |A||B|=M\hbox{ and }\Phi_s(X)\ |\ A(X)B(X)\hbox{ for all }s|M,\ s\neq 1.
\end{equation}
Since $\Phi_s$ are irreducible, each $\Phi_s(X)$ with $s|M$ must divide at least one of $A(X)$ and $B(X)$.

The following result is due to Coven and Meyerowitz \cite{CM}. 

\begin{theorem}\label{CM-thm} \cite{CM}
Let $S_A$ be the set of prime powers
$p^\alpha$ such that $\Phi_{p^\alpha}(X)$ divides $A(X)$.  Consider the following conditions.

\smallskip
{\it (T1) $A(1)=\prod_{s\in S_A}\Phi_s(1)$,}

\smallskip
{\it (T2) if $s_1,\dots,s_k\in S_A$ are powers of different
primes, then $\Phi_{s_1\dots s_k}(X)$ divides $A(X)$.}
\medskip

Then:

\begin{itemize}

\item if $A$ satisfies (T1), (T2), then $A$ tiles $\ZZ$;

\item  if $A$ tiles $\ZZ$ then (T1) holds;

\item if $A$ tiles $\ZZ$ and $|A|$ has at most two distinct prime factors,
then (T2) holds.
\end{itemize}

\end{theorem}

While (T1) is relatively easy to prove, (T2) turns out to be much deeper and more difficult. Coven and Meyerowitz \cite{CM} proved that if $A$ satisfies (T2), then $A\oplus B^\flat=\ZZ_M$, where $M=\hbox{lcm}(S_A)$ and $B^\flat$ is an explicit, highly structured ``standard" tiling complement (defined here in Section \ref{standards}). We prove in \cite{LaLo1} that having a tiling complement of this type is in fact equivalent to (T2). In this formulation, (T2) bears some resemblance to questions on replacement of factors in theory of factorizations of abelian groups (see \cite{Szabo-book} for an overview of the latter).

The proof of Theorem \ref{CM-thm} in \cite{CM} is based on an inductive argument. Coven and Meyerowitz use a theorem of Tijdeman \cite{Tij} to prove that if $A$ tiles the integers, then it also tiles $\ZZ_M$ for some $M$ which
has the same prime factors as $|A|$. Hence, if $|A|$ has at most two distinct prime factors, we may assume that so does $M$. The authors then use Sands's theorem \cite{Sands}, which states that, in any tiling 
$A\oplus B=\ZZ_M$ with $M$ divisible by at most 2 primes, at least one of $A$ and $B$ must be contained in $p\ZZ_M$ for some prime $p|M$. Coven and Meyerowitz use this to decompose the given tiling into tilings of smaller groups while keeping track of the (T2) property.
We also note that if $|A|$ is a prime power, then the Coven-Meyerowitz characterization simplifies 
further since (T2) is vacuous; in this case, the result had been proved earlier by 
Newman \cite{New}.

The Coven-Meyerowitz proof does not extend to the general case. Sands's factor replacement
theorem is false if $M$ has three or more
prime factors, with counterexamples in \cite{Sz}, \cite{LS}. On the other hand, we prove in \cite[Corollary 6.2]{LaLo1} (using a relatively minor modification of the argument in \cite{CM}) that if $A\oplus B=\ZZ_M$, and if $|A|$ and $|B|$ share at most two distinct prime factors, then both $A$ and $B$ satisfy (T2). (See also \cite{Tao-blog}, \cite{shi}.) Thus the simplest case that is not covered by these methods is when $|A|=|B|=p_ip_jp_k$, where $p_i,p_j,p_k$ are distinct primes.

Our main result is the following theorem.

\begin{theorem}\label{T2-3primes-thm}
Let $M=p_i^2p_j^2p_k^2$, where $p_i,p_j,p_k$ are distinct odd primes. 
Assume that $A\oplus B=\ZZ_M$, with $|A|=|B|=p_ip_jp_k$. Then both $A$ and $B$ satisfy (T2).
\end{theorem}

We also obtain a classification of all tilings $A\oplus B=\ZZ_M$, where $M=p_i^2p_j^2p_k^2$. Our main results in that regard are Theorems \ref{unfibered-mainthm} and \ref{fibered-thm}. Since those theorems require some notation and definitions, we postpone their statements until Section \ref{classified}.

The proof of Theorem \ref{T2-3primes-thm} relies on the methods and concepts introduced in 
 \cite{LaLo1}. In order to keep this article reasonably self-contained modulo results that can be used as black boxes, 
 we provide a summary of the concepts and results that we will need here, specialized 
to the 3-prime setting, in Section \ref{sec-prelim}.
We then state our classification results in
Section \ref{classified1}. In Section \ref{classified2}, we discuss our strategy and the main new ideas of the proof. The rest of the paper is devoted to the proof of Theorems \ref{T2-3primes-thm}, \ref{unfibered-mainthm}, and \ref{fibered-thm}.

Since \cite{CM}, there has been essentially no progress on proving (T2), except for a few special cases that either assume particular structure of the tiling (see \cite{KL}, \cite{dutkay-kraus}) or are covered by the methods of \cite{CM} as in \cite[Corollary 6.2]{LaLo1} (see \cite{Tao-blog}, \cite{shi}). However, there has been recent work on other questions related to tiling. For instance, Bhattacharya \cite{Bh} has established the periodic tiling conjecture in $\ZZ^2$, with a quantitative version due to Greenfeld and Tao \cite{GT}. There has also been interesting work on tilings of the real line by a function (see \cite{KoLev} for a survey and some open questions).

The Coven-Meyerowitz tiling conditions have implications for the ongoing work on Fuglede's spectral set conjecture \cite{Fug}. Fuglede conjectured that a set $\Omega\subset \RR^n$ of positive $n$-dimensional Lebesgue measure tiles $\RR^n$ by translations if and only if it is {\em spectral}, in the sense that the space $L^2(\Omega)$ admits an orthogonal basis of exponential functions. 
While the conjecture has been disproved in its full generality in dimensions 3 and higher \cite{tao-fuglede}, \cite{KM}, \cite{KM2}, \cite{FMM}, \cite{matolcsi}, \cite{FR}, significant connections between tiling and spectrality do exist (see \cite{dutkay-lai} for an overview of the problem in dimension 1), and 
there is a large body of work investigating such connections from many points of view. 
In higher dimensions, the conjecture has been proved for convex sets in $\RR^n$, by Iosevich, Katz and Tao \cite{IKT} for $n=2$,   Greenfeld and Lev \cite{GL} for $n=3$, and by Lev and Matolcsi \cite{LM} for general $n$. There have been many recent results on special cases of the finite abelian group analogue of the conjecture \cite{IMP}, \cite{MK}, \cite{FMV}, \cite{KMSV}, \cite{KMSV2}, \cite{KS}, \cite{M}, \cite{shi}, \cite{shi2}, \cite{somlai}, \cite{FKS}, \cite{zhang}.

If (T2) could be proved for all finite integer tiles, this would imply the ``tiling implies spectrum" part of Fuglede's spectral set conjecture for all compact tiles in dimension 1, as well as for all cyclic groups $\ZZ_M$. This follows from the 
results of \cite{LW1}, \cite{LW2}, \cite{L}. Proving (T2) for specific tiling problems does not resolve the full conjecture, but it does imply that the conjecture holds in those settings. 
In that regard, our Theorem \ref{T2-3primes-thm} combined with \cite[Theorem 1.5]{L} and \cite[Corollary 6.2]{LaLo1} has the following immediate corollary.

\begin{corollary}\label{cor-Fuglede}
Let $M={p_i^2p_j^2p_k^2}$ be odd.

\smallskip
(i) The ``tiling implies spectrum" part of Fuglede's spectral set conjecture holds for the cyclic group $\ZZ_M$. In other words, if $A\subset\ZZ_M$ tiles $\ZZ_M$ by translations, then it is spectral.

\smallskip
(ii) Let $A\subset \ZZ$ be a finite set such that $A$ mod $M$ tiles $\ZZ_{M}$, and let ${F=\bigcup_{a\in A}[a,a+1]}$, so that $F$ 
tiles $\RR$ by translations. Then $F$ is spectral. 
\end{corollary}

Indeed, let $M$ and $A$ be as in Corollary \ref{cor-Fuglede}. 
If $|A|\neq p_ip_jp_k$, then both $A$ and $B$ satisfy (T2) by \cite[Corollary 6.2]{LaLo1}. If on the other hand 
$|A|= p_ip_jp_k$, then both $A$ and $B$ satisfy (T2) by Theorem \ref{T2-3primes-thm}. In both cases, spectrality follows from \cite[Theorem 1.5]{L}. (While Theorem 1.5 in \cite{L} is stated for unions of finite intervals as in (ii), the same argument applies in the finite group setting. See e.g. \cite{dutkay-lai}.)

After this paper was completed, we were able to extend Theorem
\ref{T2-3primes-thm} and Corollary \ref{cor-Fuglede} to the case when $M={p_i^2p_j^2p_k^2}$ is even. Thus, both results are now known to hold with no restrictions on the parity of $M$. See the follow-up article \cite{LaLo3} for details.


\section{Notation and preliminaries}\label{sec-prelim}


This section summarizes the relevant definitions and results of \cite{LaLo1}, specialized to the 3-prime case. All material due to other authors is indicated explicitly as such.

\subsection{Multisets and mask polynomials}

Throughout this paper, we will assume that $M=p_i^{n_i}p_j^{n_j}p_k^{n_k}$, where $p_i,p_j,p_k$ are distinct primes and $n_i,n_j,n_k\in\NN$. The indices $\{i,j,k\}$ can be thought of as a permutation of $\{1,2,3\}$; however, we will always use $i,j,k$ for this purpose, freeing up numerical subscripts for other uses. While the full proof of Theorem \ref{T2-3primes-thm} requires that $n_i=n_j=n_k=2$ and that $p_i,p_j,p_k\neq 2$, many of our intermediate results are valid under weaker assumptions as indicated.

We will always work in either $\ZZ_M$ or in $\ZZ_N$ for some $N|M$. 
We use $A(X)$, $B(X)$, etc. to denote polynomials modulo $X^M-1$ with integer coefficients. 
Each such polynomial $A(X)=\sum_{a\in\ZZ_M} w_A(a) X^a$ is associated with a weighted multiset in $\ZZ_M$, which we will also denote by $A$, with weights $w_A(x)$ assigned to each $x\in\ZZ_M$. (If the coefficient of $X^x$ in $A(X)$ is 0, we set $w_A(x)=0$.) In particular, if $A$ has $\{0,1\}$ coefficients, then
$w_A$ is the characteristic function of a set $A\subset \ZZ_M$. We will use $\calm(\ZZ_M)$ to denote the 
family of all weighted multisets in $\ZZ_M$, and reserve the notation $A\subset \ZZ_M$ for sets.

If $N|M$, then any $A\in \calm(\ZZ_M)$ induces a weighted multiset $A$ mod $N$ in $\ZZ_N$, with the corresponding mask polynomial $A(X)$ mod $(X^N-1)$ and induced weights 
$$
w_A^N(x)= \sum_{x'\in\ZZ_M: x'\equiv x\,{\rm mod}\, N} w_A(x'),\ \ x\in\ZZ_N.
$$
We will continue to write $A$ and $A(X)$ for $A$ mod $N$ and $A(X)$ mod $X^N-1$, respectively, while working in $\ZZ_N$.

 If $A,B\in\calm(\ZZ_M)$, we will use $A+B$ to indicate the weighted multiset corresponding to the mask polynomial
 $(A+B)(X)=A(X)+B(X)$, with the weight function $w_{A+B}(x)=w_A(x)+w_B(x)$.
 We use the convolution notation $A*B$ to denote the weighted sumset of $A$ and $B$, so that $(A*B)(X)=A(X)B(X)$
 and
 $$w_{A*B}(x)=(w_A*w_B)(x)=\sum_{y\in\ZZ_M} w_A(x-y)w_B(y).$$
 If one of the sets is a singleton, say $A=\{x\}$, we will simplify the notation and write $x*B=\{x\}*B$. 
 The direct sum notation $A\oplus B$ is reserved for tilings, i.e., $A\oplus B=\ZZ_M$ means that $A,B\subset\ZZ_M$ are both sets and $A(X)B(X)=\frac{X^M-1}{X-1}$ mod $X^M-1$. 
We will not use derivatives of polynomials in this paper, hence notation such as $A'$, $A''$, etc., will be used to denote auxiliary multisets and polynomials rather than derivatives.



\subsection{Array coordinates and geometric representation}
\label{sec-array}


For $\nu\in\{i,j,k\}$, define 
$M_\nu: = M/p_\nu^{n_\nu}= \prod_{\kappa\neq \nu} p_\kappa^{n_\kappa}$. 
Then each $x\in \ZZ_M$ can be written uniquely as
$$
x=\sum_{\nu\in\{i,j,k\}} \pi_\nu(x) M_\nu,\ \ \pi_\nu(x)\in \ZZ_{p_\nu^{n_\nu}}.
$$
This sets up an isomorphism $\ZZ_M\simeq \ZZ_{p_i^{n_i}}\oplus \ZZ_{p_j^{n_j}}\oplus \ZZ_{p_k^{n_k}}$, and identifies each element 
$x\in\ZZ_M$ with an element of a 3-dimensional lattice with coordinates $(\pi_i(x),\pi_j(x),\pi_k(x))$.
The tiling $A\oplus B=\ZZ_M$ can then be interpreted as a tiling of that lattice. 

For $D|M$, a {\em $D$-grid} in $\ZZ_M$ is a set of the form
$$
\Lambda(x,D):= x+D\ZZ_M=\{x'\in\ZZ_M:\ D|(x-x')\}
$$
for some $x\in\ZZ_M$. In other words, it is a coset of the subgroup of order $M/D$ in $\ZZ_M$. 

A few special cases have a geometric interpretation of interest. 
A {\em line} through $x\in\ZZ_M$ in the $p_\nu$ direction is the set
$\ell_\nu(x):= \Lambda(x,M_\nu)$, and a {\em plane} through $x\in\ZZ_M$ perpendicular to the $p_\nu$ direction, on the scale $M_\nu p_\nu^{\alpha_\nu}$, is the set
$\Pi(x,p_\nu^{\alpha_\nu}):=\Lambda(x,p_\nu^{\alpha_\nu}).$

An {\em $M$-fiber in the $p_\nu$ direction} is a set of the form $x*F_\nu$, where $x\in\ZZ_M$ and
$$
F_\nu=\{0,M/p_\nu,2M/p_\nu,\dots,(p_\nu-1)M/p_\nu\}.
$$
Thus $x*F_\nu=\Lambda(x,M/p_\nu)$.  A set $A\subset \ZZ_M$ is {\em $M$-fibered in the $p_\nu$ direction} if there is a subset $A'\subset A$ such that $A=A'*F_\nu$.

For $N|M$, we define
$$D(N)=\frac{N}{\hbox{rad}(N)},$$
where $\hbox{rad}(N)$ is the radical of $N$, i.e., the product of the distinct primes dividing $N$. Explicitly, 
if $N=p_i^{\alpha_i} p_j^{\alpha_j} p_k^{\alpha_k}$ with  $0\leq \alpha_\nu\leq n_\nu$, then 
$$
D(N):= p_i^{\gamma_i} p_j^{\gamma_j} p_k^{\gamma_k},\ \hbox{ where }
\gamma_\nu=\max(0,\alpha_\nu-1)\hbox{ for }\nu\in\{i,j,k\}.
$$ 
We will also write $N_\nu=M/p_\nu$ for $\nu\in\{i,j,k\}$.


\subsection{Divisor set and divisor exclusion}


For $m,n\in\ZZ$, we use $(m,n)$ to denote the greatest common divisor of $m$ and $n$. We will also write $p^\alpha\parallel m$, where $p$ is prime and $\alpha$ is a nonnegative integer, if $p^\alpha|m$ and $p^{\alpha+1}\nmid m$.

For $N|M$ and $A\subset\ZZ_M$, we define
$$
\Div_N(A):=\{(a-a',N):\ a,a'\in A\}
$$
When $N=M$, we will omit the subscript and write $\Div(A)=\Div_M(A)$. Informally, we will refer to the elements of $\Div(A)$ as the {\em divisors of $A$} or {\em differences in $A$}. 
A theorem due to Sands \cite{Sands} states that
$A\oplus B=\ZZ_M$ if and only if $A,B\subset \ZZ_M$ are sets such that $|A|\,|B|=M$ and 
$$
\Div(A) \cap \Div(B)=\{M\}.
$$
We will refer to this as {\em divisor exclusion}.

In cases when we need to indicate where a particular divisor of $A$ must occur, we will 
use the following notation for localized divisor sets. If $A,A_1,A_2\subset \ZZ_M$ and $a_0\in\ZZ_M$, we will write
\begin{equation*}
\begin{split}
\Div_N(A_1,A_2)&:=\{(a_1-a_2,N):\ a_1\in A_1,a_2\in A_2\},\\
\Div_N(A,a_0)=\Div_N(a_0,A) &:=\Div_N(A,\{a_0\})= \{(a-a_0,N):\ a\in A\}.
\end{split}
\end{equation*}
For example, if $A\oplus B=\ZZ_M$, we will often need to consider $\Div(A_1,A_2)$, where $A_1$ and $A_2$ are restrictions of $A$ to geometric structures such as planes or lines.


\subsection{Standard tiling complements}\label{standards}

Suppose that $A\oplus B=\ZZ_M$, and let
\begin{equation*}
\mathfrak{A}_\nu(A) =\left\{\alpha_\nu \in \{1,2,\dots,n_\nu\}:\ \Phi_{p_\nu^{\alpha_\nu}}(X) | A(X) \right\}  
\end{equation*}
The {\em standard tiling complement} $A^\flat$ is defined via its mask polynomial
$$
A^\flat(X) =  \prod_{\nu\in\{i,j,k\}} \prod_{\alpha_\nu\in \mathfrak{A}_\nu(A)}  
\left(1+ X^{M_\nu p_\nu^{\alpha_\nu-1}} + \dots + X^{(p_\nu-1)M_\nu p_\nu^{\alpha_\nu-1}} \right).
$$

\begin{figure}[h]
	\captionsetup{justification=centering}
	\includegraphics[scale=1.5]{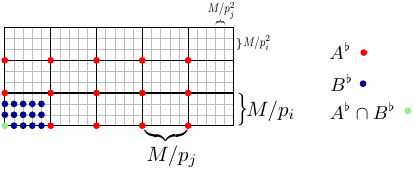}
	\caption{The standard sets $A^\flat,B^\flat\subset\ZZ_{p_i^2p_j^2}$\\
		with $p_i=3, p_j=5$ and $\Phi_{p_i^2}\Phi_{p_j^2}|A,\Phi_{p_i}\Phi_{p_j}|B$.}
\end{figure}

Then $A^\flat(X)$ satisfies (T2) and has the same prime power cyclotomic divisors as $A(X)$. 
For each prime power $s|M$, $\Phi_s$ divides exactly one of $A$ and $B$ (\cite{CM}), hence $A^\flat$ is also uniquely determined by $M$ and $B$. Coven and Meyerowitz proved in \cite{CM} that if a finite tile satisfies (T2), it has a standard tiling complement. We prove the converse in \cite{LaLo1}.

\begin{proposition}\label{replacement} Let $A\oplus B=\ZZ_M$. Then $A^\flat \oplus B= \ZZ_M$
if and only if $B$ satisfies (T2).
\end{proposition}

We say that the tilings $A\oplus B=\ZZ_M$ and $A'\oplus B=\ZZ_M$ are {\em T2-equivalent} if 
$$
A\hbox{ satisfies (T2) }\Leftrightarrow A'\hbox{ satisfies (T2).}
$$

Since $A$ and $A'$ tile the same group $\ZZ_M$ with the same tiling complement $B$, they must have the same cardinality and the same prime power cyclotomic divisors. We will sometimes say simply that $A$ is T2-equivalent to $A'$ if both $M$ and $B$ are clear from context.
Usually, $A'$ will be derived from $A$ using certain permitted manipulations such as fiber shifts (Lemma \ref{fibershift}).
In particular, if we can prove that either $A$ or $B$ in a given tiling is T2-equivalent to a standard tiling complement, this resolves the problem completely in that case.

\begin{corollary}\label{get-standard}
Suppose that the tiling $A\oplus B=\ZZ_M$ is T2-equivalent to the tiling $A^\flat \oplus B=\ZZ_M$. Then 
$A$ and $B$ satisfy (T2).
\end{corollary}


\subsection{Box product}\label{sec-box-product}

Let $A\subset \ZZ_M$ and $N|M$. 
For $x\in\ZZ_M$, define
\begin{align*}
\bbA^N_m[x] & = \# \{a\in A:\ (x-a,N)=m \}.
\end{align*}
We may think of $\bbA^N_m[x]$, with $x$ fixed and $m$ ranging over the divisors of $N$, as the entries of the 
{\em $N$-box} $\bbA^N[x] = (\bbA^N_m[x])_{m|N}$ \cite{LaLo1}. If $C\in\calm(\ZZ_M)$, we write 
$$\bbA_m^N [C] :=\sum_{x\in\ZZ_N}w_C^N(x) \bbA^N_m[x].$$
In particular, if $C\subset \ZZ_M$, we have
$\bbA_m^N [C] :=\sum_{c\in C}\bbA^N_m[c]$. Furthermore, 
if $X\subset\ZZ_M$ and $x\in\ZZ_M$, we define the restricted box entries
$$
\bbA^N_m[x|X] = \# \{a\in A\cap X: \ (x-a,N)=m\}.
$$
If $N=M$, we will usually omit the superscript and write
$\bbA^M_m[x]=\bbA_m[x]$, $\bbA^M[x]=\bbA[x]$, $\bbA^M_m[C]=\bbA_m[C]$, and so on.

If $A,B\subset\ZZ_M$, we define the {\em box product} of the associated $M$-boxes as
$$
\langle \bbA[x], \bbB[y] \rangle = \sum_{m|M} \frac{1}{\phi(M/m)} \bbA_m[x] \bbB_m[y].
$$
Here $\phi$ is the Euler totient function: if 
$n=\prod_{\iota=1}^L q_\iota^{r_\iota}$, where $q_1,\dots,q_L$ are distinct primes and $r_\iota\in\NN$, then
$
\phi(n)= \prod_{\iota=1}^L (q_\iota-1)q_\iota^{r_\iota-1}.
$

\begin{theorem} (\cite{LaLo1}; following \cite[Theorem 1]{GLW}) If
$A\oplus B=\ZZ_M$, then 
\begin{equation}\label{e-ortho2}
\langle \bbA^M[x], \bbB^M[y] \rangle =1\ \ \forall x,y\in\ZZ_M.
\end{equation}
\end{theorem}


\subsection{Cuboids}\label{cuboid-section}

\begin{definition}
(i) A {\em cuboid type} $\calt$ on $\ZZ_N$ is an ordered triple $\calt=(N,\vec{\delta}, T)$, where:

\begin{itemize}
\item $N=\prod_{\nu\in\{i,j,k\}} p_\nu^{n_\nu-\alpha_\nu}$ is a divisor of $M$, with $0\leq \alpha_\nu\leq n_\nu$ for each $\nu$,

\item $\vec{\delta}=(\delta_i,\delta_j,\delta_k)$, with $0\leq \delta_\nu\leq n_\nu-\alpha_\nu$, 

\item the {\em template} $T$ is a nonempty subset of $\ZZ_N$.

\end{itemize}

\medskip

(ii) A {\em cuboid} $\Delta$ of type $\calt$ is a weighted multiset corresponding to a mask polynomial of the form
$$
\Delta(X)= X^c\prod_{\nu\in \mathfrak{J}} (1-X^{d_\nu}),
$$
where $\mathfrak{J} =\mathfrak{J}_{\vec\delta} := \{\nu: \delta_\nu\neq 0\}$,
and $c,d_\nu$ are elements of $\ZZ_M$ such that $(d_\nu,N)=N/p_\nu^{\delta_\nu}$ for $\nu\in\{i,j,k\}$.
The {\em vertices} of $\Delta$ are the points
$$
x_{\vec\epsilon}=c+\sum_{\nu\in \mathfrak{J}} \epsilon_\nu d_\nu:\ \vec{\epsilon} = (\epsilon_\nu)_{\nu\in \mathfrak{J}} \in\{0,1\}^{|\mathfrak{J}|},
$$
with weights
$w_\Delta(x_{\vec\epsilon})=(-1)^{\sum_{\nu\in \mathfrak{J}}\epsilon_\nu}$.

\medskip  
(iii) 
Let $A\in\calm(\ZZ_N)$, and let $\Delta$ be a cuboid of type $\calt$. Define
$$
\bbA^\calt [\Delta] = \bbA^N_N[\Delta*T]= \sum_{  \vec\epsilon\in\{0,1\}^k} w_\Delta(x_{\vec\epsilon}) \bbA^N_N[x_{\vec\epsilon}*T],
$$
where we recall that $x*T=\{x+t: \ t\in T\}$, so that
$$
\bbA^N_N[x_{\vec\epsilon}*T]:= \sum_{t\in T} \bbA^N_N[x_{\vec\epsilon}+t].$$
\end{definition}

Informally, a cuboid type $\calt=(N,\vec{\delta}, T)$ is a class of cuboids $\Delta$ on the scale $N$, defined as in (ii) (so that $\vec{\delta}$ indicates the greatest common divisors of the differences in $\Delta$ with $N$), together with the evaluation rule in (iii) given by the template $T$. Thus, when we say that $\Delta$ is a cuboid of type $\calt$, this just means that $\Delta$ satisfies the condition in (ii) which is independent of $T$; however, the evaluation $\bbA^\calt [\Delta]$ depends on all of $A$, $N$, $\Delta$, and $T$.
For consistency, we will also write $
\bbA^\calt [x] = \bbA^N_N[x*T]$ for $x\in\ZZ_M.$ 

An important special case is as follows: for $N|M$, 
an {\em $N$-cuboid} is a cuboid of type $\calt=(N,\vec{\delta},T)$, where 
$N|M$, $T(X)=1$, and $\delta_\nu=1$ for all $\nu$ such that $p_\nu|N$. Thus, $N$-cuboids have the form
\begin{equation}\label{def-N-cuboids}
\Delta(X)= X^c\prod_{p_\nu|N} (1-X^{d_\nu}),
\end{equation}
 with $(d_\nu,N)=N/p_\nu$ for all $\nu $ such that $p_\nu|N$.
We reserve the term ``$N$-cuboid", without cuboid type explicitly indicated, to refer to cuboids as in (\ref{def-N-cuboids}); for cuboids of any other type, we will always specify $\calt$.

Cuboids provide useful criteria to determine cyclotomic divisibility properties of mask polynomials. 
We say that a multiset $A\in\calm(\ZZ_M)$
{\em is $\mathcal{T}$-null} if for every cuboid $\Delta$ of type $\mathcal{T}$, 
$$
\bbA^\calt [\Delta] =0.
$$
Note that this is a property of $A$ that depends on all of $N$, $\vec{\delta}$, and $T$.

For $A\in\calm(\ZZ_N)$, we have
$\Phi_N(X)|A(X)$ if and only if $\bbA^N_N[\Delta]=0$ for every $N$-cuboid $\Delta$. This 
has been known and used previously in the literature,
see e.g.  \cite[Section 3]{Steinberger}, or \cite[Section 3]{KMSV}.
In particular, for any $N|M$, $\Phi_N$ divides $A$ if and only if it divides the mask polynomial of $A\cap\Lambda(x,D(N))$ for every $x\in\ZZ_M$. If $N=M$, we will follow the convention from Section \ref{sec-box-product} and use the simplified notation $\bbA^M_M[\Delta]=\bbA_M[\Delta]$.

More generally, if for every $m|N$ the polynomial
$\Phi_m(X)$ divides at least one of $A(X)$, $T(X)$, or $\Delta(X)$ for every $\Delta$ of type $\calt=(N,\vec{\delta}, T)$,
then $A$ is $\calt$-null \cite[Lemma 5.3]{LaLo1}. We use such cuboid types to test for divisibility by combinations of cyclotomic polynomials.
In particular, the following are special cases of Examples (1)--(3) in \cite[Section 5.3]{LaLo1} with $M=p_i^{n_i}p_j^{n_j}p_k^{n_k}$.

\begin{itemize}

\item Assume that $n_i\geq 2$, and let $\calt=(M, \vec{\delta}, 1)$, with $\delta_i=2$ and $\delta_j=\delta_k=1$. Then
\begin{equation*}
\Phi_M \Phi_{M/p_i}|A  \Leftrightarrow A \hbox{ is }\calt\hbox{-null}.
\end{equation*}
If $n_i=1$, the same is true with $\delta_i=0$ (instead of $\delta_i=2$).

\item Assume that $n_i= 2$, and let $\calt=(M, \vec{\delta}, T)$,
where $\delta_i=0$, $\delta_j=\delta_k=1$, and
$$
T(X)=\frac{X^{M/p_i}-1}{X^{M/p_i^2}-1}
= 1+X^{M/p_i^2}+\dots+X^{(p_i-1)M/p_i^2}.
$$
If $\Phi_M\Phi_{M/p_i^2}|A$, then $A$ is $\calt$-null.
\end{itemize}


\subsection{Tiling reductions}
The general formulations of the tiling reductions below are provided in Theorems 6.1, 6.5, and Corollary 6.6 in 
\cite{LaLo1}. The additional assumption that $p_\nu\parallel |B|$ in Theorem \ref{subgroup-reduction} ensures that in any tiling 
$A'\oplus B'=\ZZ_{N_\nu}$ with $|A'|=|A|$ and $|B'|=|B|/p_\nu$, $|A'|$ and $|B'|$ have only two common factors. Hence the assumption (ii) of \cite[Theorem 6.1]{LaLo1} is satisfied by \cite[Corollary 6.2]{LaLo1}, and we deduce that $A$ and $B$ both satisfy (T2). The assumption that $p_\nu\parallel |A|$ in Corollary
\ref{slab-reduction} serves the same purpose.

\begin{theorem}\label{subgroup-reduction} {\bf (Subgroup reduction)} \cite[Lemma 2.5]{CM}
Let $M=p_i^{n_i}p_j^{n_j}p_k^{n_k}$.
Assume that $ A\oplus B=\ZZ_M $, and that $A\subset p_\nu \ZZ_M$ for some 
$\nu\in\{i,j,k\}$ such that $p_\nu\parallel |B|$.
Then $A$ and $B$ satisfy (T2). 
\end{theorem}

\begin{theorem}\label{subtile}
Assume that $A\oplus B=\ZZ_M$ and $\Phi_{p_\nu^{n_\nu}}|A$ for some $\nu\in\{i,j,k\}$.
Define
$$
A_{p_\nu}=\{a\in A:\ 0\leq\pi_\nu(a)\leq p_\nu^{n_\nu-1}-1\}.
$$
Then the following are equivalent:

\smallskip

(i) For any translate $A'$ of $A$ we have $A'_{p_\nu}\oplus B=\ZZ_{M/p_\nu}$.
 
\smallskip

(ii) For every $d$ such that $p_\nu^{n_\nu}|d|M$, at least one of the following holds:
$$
\Phi_d|A,
$$
$$
\Phi_{d/p_\nu}\Phi_{d/p_\nu^2}\dots \Phi_{d/p_\nu^{n_\nu}}\mid B.
$$

\end{theorem}

\begin{corollary}\label{slab-reduction} {\bf (Slab reduction)}
Let $M=p_i^{n_i}p_j^{n_j}p_k^{n_k}$.
Assume that $A\oplus B=\ZZ_M$, and that 
there exists a $\nu\in\{i,j,k\}$ such that $\Phi_{p_\nu^{n_\nu}}|A$, $p_\nu\parallel |A|$,  and $A,B$ obey the condition (ii) of Theorem \ref{subtile}. (In particular, this holds if $A$ is $M$-fibered in one of the $p_i,p_j,p_k$ directions.)
Then $A$ and $B$ satisfy (T2). 
\end{corollary}


\subsection{Saturating sets}

Let $A\oplus B= \ZZ_M$, and $x,y\in\ZZ_M$.
Define
$$
A_{x,y}:=\{a\in A:\ (x-a,M)=(y-b,M) \hbox{ for some }b\in B\},
$$
$$
A_{x}:=\{a\in A: (x-a,M)\in\Div(B)\} =\bigcup_{b\in B} A_{x,b}.
$$
We will refer to $A_x$ as the {\em saturating set} for $x$. The sets $B_{y,x}$ and $B_y$ are defined similarly, with $A$ and $B$ interchanged.

By divisor exclusion, $A_a=\{a\}$ for all $a\in A$. For $x\in\ZZ_M\setminus A$, saturating spaces are more robust, but are still subject to geometric constraints based on divisor exclusion. 

\begin{lemma}[\bf Bispan lemma]\cite[Lemma 7.7]{LaLo1}\label{bispan-lemma} 
Let $A\oplus B=\ZZ_M$. 
For $x,x'\in\ZZ_M$ such that $(x-x',M)=p_i^{\alpha_i}p_j^{\alpha_j} p_k^{\alpha_k}$, with $0\leq \alpha_\nu\leq n_\nu$, define
\begin{equation}\label{e-span}
\begin{split}
\Span(x,x')&=\bigcup_{\nu: \alpha_\nu<n_\nu} \Pi(x,p_\nu^{\alpha_\nu+1}),
\\
\Bispan(x,x')&= \Span(x,x')\cup \Span(x',x).
\end{split}
\end{equation}
Then for any $x,x',y\in\ZZ_M$, we have
$$
A_{x',y}\subset A_{x,y}\cup\Bispan(x,x'),
$$
and in particular,
\begin{equation}\label{bispan}
A_x \subset \bigcap_{a\in A} \Bispan (x,a).
\end{equation}
\end{lemma}

The following special case will be used often. 

\begin{corollary}\label{bispan-corollary}
Let $A\oplus B=\ZZ_M$. 

\begin{itemize}
\item[(i)] Suppose that $x\in \ZZ_M\setminus A$ satisfies $(x-a,M)=M/p_\nu$ for some $a\in A$ and $\nu\in\{i,j,k\}$. Then
\begin{equation}\label{two-planes}
A_x\subset \Pi(x,p_\nu^{n_\nu}) \cup\Pi(a,p_\nu^{n_\nu}).
\end{equation}

\item[(ii)] If $x\in \ZZ_M\setminus A$ satisfies $\bbA_{M/p_\nu}[x]\geq 2$ for some $\nu\in\{i,j,k\}$, then
$$
A_x\subset \Pi(x,p_i^{n_i}).
$$

\end{itemize}
\end{corollary}

\begin{proof}
Let $x,a$ be as in (i). Then
$$
\Bispan(x,a)=\Pi(x,p_i^{n_i}) \cup\Pi(a,p_i^{n_i}),
$$
so that (\ref{two-planes}) follows from (\ref{bispan}). 

Suppose now that $\bbA_{M/p_\nu}[x]\geq 2$, and let $a,a'\in A$ satisfy $a\neq a'$ and $(x-a,M)=(x-a',M)=M/p_\nu$. Then (\ref{two-planes}) holds for $x$ and $a$, as well as for $x$ and $a'$. Taking the intersection, we get
\begin{align*}
A_x& \subset \Big( \Pi(x,p_\nu^{n_\nu}) \cup\Pi(a,p_\nu^{n_\nu})\Big)
\cap \Big( \Pi(x,p_\nu^{n_\nu}) \cup\Pi(a',p_\nu^{n_\nu})\Big)
\\
&= \Pi(x,p_\nu^{n_\nu}),
\end{align*}
as claimed.
\end{proof}

In the sequel, whenever we evaluate saturating sets, we will always start with geometric restrictions based on Lemma \ref{bispan-lemma} and Corollary \ref{bispan-corollary}.

\subsection{Fibers and cofibered structures}
The following is a simplified version of the definitions and results of  \cite[Section 8]{LaLo1}. In the exposition below, we focus primarily on the case $M=p_i^2p_j^2p_k^2$, which contains all of the main ideas and will be sufficient for our purposes most of the time. The more general case is covered in Lemma \ref{fibershift-plus}. We refer the reader to \cite[Section 8]{LaLo1} for more details.

Let $N|M$, $c\in\NN$, and $\nu\in\{i,j,k\}$. Assume that $p_\nu|N$.
An {\em $N$-fiber in the $p_\nu$ direction} with multiplicity $c$ is a set $F\subset \ZZ_M$ such that
$F$ mod $N$ has the mask polynomial
$$
F(X)\equiv cX^a(1+X^{N/p_\nu}+ X^{2N/p_\nu}+\dots +X^{(p_\nu-1)N/p_\nu})\mod (X^{N}-1)
$$
for some $a\in \ZZ_M$. We will say sometimes that $F$ {\em passes through} $a$, or {\em is rooted at} $a$.
A set $A\subset \ZZ_M$ is {\em $N$-fibered in the $p_\nu$ direction} if it can be written as a union of disjoint 
$N$-fibers in the $p_\nu$ direction, all with the same multiplicity.

{\em Fiber chains} in the $p_\nu$ direction are translates of sets that tile $(M/p_\nu^\gamma)\ZZ_M$ for some $\gamma$ with $1\leq\gamma\leq n_i$. For $M=p_i^2p_j^2p_k^2$, the only fiber chains of interest that are not fibers on some scale are multisets $\tilde{F}$ with mask polynomials
$$
\tilde{F}(X)=cX^a(1+X^{M_\nu}+ X^{2M_\nu}+\dots +X^{(p_\nu^2-1)M_\nu}),\ \ \nu\in\{i,j,k\},\ a\in\ZZ_M,\ c\in\NN.
$$

If $F\subset \ZZ_M$ is an $M$-fiber in the $p_\nu$ direction, we say that an element $x\in \ZZ_M$ is at {\em distance} $m$ from $F$ if $m|M$ is the maximal divisor
such that $(z-x,M)=m$ for some $z\in F$. It is easy to see that such $m$ exists.

Let $A\oplus B=\ZZ_M$. We will often be interested 
in finding ``complementary" fibers and fibered structures in $A$ and $B$, with the following special case occurring particularly often.

\begin{definition}[\bf Cofibers and (1,2) cofibered structures]\label{cofibers} 
Let $A, B\subset \ZZ_M$,  with $M=p_i^2p_j^2p_k^2$, and let $\nu\in\{i,j,k\}$.

\smallskip
 
(i) We say that $F\subset A,G\subset B$ are {\em $(1,2)$-cofibers} in the $p_\nu$ direction
if $F$ is an $M$-fiber and $G$ is an $M/p_\nu$-fiber, both in the $p_\nu$ direction.

\smallskip
(ii) We say that the pair $(A,B)$ has a {\em (1,2)-cofibered structure} in the $p_\nu$ direction if

\begin{itemize}
\item $B$ is  $M/p_\nu$-fibered in the $p_\nu$ direction,

\item $A$ contains at least one ``complementary" $M$-fiber $F\subset A$
 in the $p_\nu$ direction, which we will call a  {\em cofiber} for this structure. 
 \end{itemize}

\end{definition}

The advantage of cofibered structure is that it permits fiber shifts as described below. In many cases, we will be able to use this to reduce the given tiling to a simpler one.

\begin{lemma}[\bf Fiber-Shifting Lemma]\label{fibershift} \cite[Lemma 8.7]{LaLo1}
Let $A\oplus B=\ZZ_M$. Assume that the pair $(A,B)$ has a $(1,2)$-cofibered structure in the $p_\nu$ direction,
with a cofiber $F\subset A$. Let $A'$ be the set obtained from $A$ by shifting $F$ to a point $x\in\ZZ_M$ at 
a distance $M/p_\nu^{2}$ from it. 
Then $A'\oplus B=\ZZ_M$, and $A$ is T2-equivalent to $A'$.
\end{lemma}

In order to identify $(1,2)$-cofibered structures in $(A,B)$, we will use
saturating sets, via the following lemma.

\begin{lemma}\label{1dim_sat-cor}\cite[Corollary 8.11]{LaLo1}
Assume that $A\oplus B =\ZZ_M$, with $M=p_i^2p_j^2p_k^2$. 
Suppose that $x\in\ZZ_M\setminus A$, $b\in B$, $M/p_\nu\in\Div(A)$, and $A_{x,b}\subset \ell_\nu(x)$ for some $\nu\in\{i,j,k\}$. Then
$$
\mathbb{A}^M_{M/p_\nu^2}[x]\mathbb{B}^M_{M/p_\nu^2}[b]=\phi(p_\nu^2).
$$
with the product saturated by a $(1,2)$-cofiber pair $(F,G)$ such that $F\subset A$ is at distance $M/p_\nu^2$ from $x$ and $G\subset B$ is rooted at $b$.
In particular, if $A_x\subset \ell_\nu(x)$,
then the pair $(A,B)$ has a $(1,2)$-cofibered structure in the $p_\nu$ direction.
\end{lemma}

We now return to the more general case when we do not assume that $n_i=n_j=n_k=2$.
In this case, the inclusion $A_x\subset \ell_\nu(x)$ implies a more complicated cofibered structure, described in \cite{LaLo1} in terms of fiber chains. We will only need the following fact, which is a consequence of Lemmas 8.7, 8.8, and 8.10 (i) of \cite{LaLo1}.

\begin{lemma}\label{fibershift-plus}
Assume that $A\oplus B =\ZZ_M$, with $M=p_i^{n_i}p_j^{n_j}p_k^{n_k}$. 
Suppose that $x\in\ZZ_M\setminus A$, $M/p_\nu\in\Div(A)$, and $A_x\subset \ell_\nu(x)$ for some $\nu\in\{i,j,k\}$. Then:
\begin{itemize}
\item[(i)]  There exists a single exponent $\gamma$ with $2\leq\gamma\leq n_i$ such that 
\begin{equation}\label{fibershift-plus-e1}
\mathbb{A}^M_{M/p_\nu^\gamma}[x]\mathbb{B}^M_{M/p_\nu^\gamma}[b]=\phi(p_\nu^\gamma) \text{ for all } b\in B.
\end{equation}
\item[(ii)]  $A_x$ is a disjoint union of $M$-fibers in the 
$p_\nu$ direction. 
\item[(iii)] Let $A'$ be the set obtained from $A$ by shifting $A_x$ to $x$.
More precisely, let $a\in A_x$, and let $A'\subset\ZZ_M$ be the set such that for all $z\in\ZZ_M$,
$$
w_{A'}(z)=\begin{cases}
w_A(z) &\hbox{ if } M/p_\nu^\gamma \nmid x-z,\\
w_A(z-x+a) &\hbox{ if } M/p_\nu^\gamma \mid x-z.
\end{cases} 
$$
Then $A'\oplus B=\ZZ_M$, and $A$ is T2-equivalent to $A'$. Moreover, $x*F_\nu\subset A'$.
\end{itemize}

\end{lemma}

\begin{proof}
Part (i) follows from \cite[Lemma 8.8 (ii)]{LaLo1}, with $\gamma\geq 2$ since $x\not\in A$ and $M/p_\nu\in\Div(A)$. 

Part (ii) follows from \cite[Lemma 8.10 (ii)]{LaLo1}. The lemma asserts, in particular, that either the cofiber $A_x$ must be $M$-fibered in the $p_\nu$ direction, or else every $b\in B$ must belong to a fiber chain in $B$ that is $M$-fibered in the $p_\nu$ direction (see \cite[Definition 8.2]{LaLo1}). However, the second alternative is not possible since 
$M/p_\nu\in\Div(A)$.

Finally, (iii) follows from \cite[Lemma 8.7]{LaLo1}, with $A_x$ as a cofiber as provided by \cite[Lemma 8.10 (ii)]{LaLo1}.
\end{proof}


\section{Classification results}\label{classified}


\subsection{Classification results}\label{classified1}

We are now ready to state our classification results and provide a more detailed outline of the proof. 
Let $A\oplus B=\ZZ_M$, where $M=p_i^{n_i}p_j^{n_j}p_k^{n_k}$. By (\ref{poly-e2}), we have $\Phi_s(X)\ |\ A(X)B(X)$ for all $s|M$ such that $s\neq 1$. In particular, $\Phi_M$ divides at least one of $A(X)$ and $B(X)$. Without loss of generality, we may assume that $\Phi_M|A$.

We have $\Phi_M|A$ if and only if $\Phi_M$ divides $A\cap\Lambda(x,D(M))$ for every $x\in\ZZ_M$
(see Section \ref{cuboid-section}). This implies structure results for restrictions of $A$ to such grids. Let $\Lambda:=\Lambda(a,D(M))$ for some $a\in A$, so that $A\cap\Lambda$ is nonempty. It is easy to see that $\Phi_M|F_\nu$ for each $\nu\in\{i,j,k\}$. 
By the classic results on vanishing sums of roots of unity \cite{deB}, \cite{Re1}, \cite{Re2}, \cite{schoen}, \cite{Mann}, \cite{LL}, $\Phi_M$ divides $A\cap\Lambda$ if and only if 
$$
(A\cap\Lambda)(X)=\sum_{\nu\in\{i,j,k\}} Q_\nu(X) F_\nu(X),
$$
where $Q_i,Q_j,Q_k$ are polynomials with integer coefficients depending on both $A$ and $\Lambda$.

A particularly simple case occurs when $A(X)=Q_\nu(X)F_\nu(X)$ for a single $\nu\in \{i,j,k\}$, so that $A$ is fibered on all $D(M)$-grids in the same direction. However, much more complicated structures are also possible. 
For instance, $A\cap\Lambda$ may be $M$-fibered in different directions on different $D(M)$-grids $\Lambda$, or there may exist a $D(M)$-grid $\Lambda$ such that $A\cap\Lambda$ contains nonintersecting $M$-fibers in two or three different directions.
An additional issue is that the polynomials $Q_i,Q_j,Q_k$ are not required to have nonnegative coefficients. In such cases, there may be points $a\in A$ such that $a*F_\nu\not\subset A$ for any $\nu$, due to cancellations between $M$-fibers in different directions.

Our classification results, and our proof of (T2), split into cases according to the fibering properties of $A$. 
Theorems \ref{unfibered-mainthm} and \ref{fibered-thm} summarize our main findings in the
unfibered and fibered case, respectively.

\begin{theorem}\label{unfibered-mainthm}
Let $A\oplus B=\ZZ_M$, where $M=p_i^{2}p_j^{2}p_k^{2}$ is odd. Assume that
$|A|=|B|=p_ip_jp_k$, $\Phi_M|A$, and there exists a $D(M)$-grid $\Lambda$ such that $A\cap\Lambda$ is nonempty and is not $M$-fibered in any direction. Assume further, without loss of generality, that $0\in\Lambda$.
Then $A^\flat=\Lambda$, and the tiling $A\oplus B=\ZZ_M$ is T2-equivalent to $\Lambda \oplus B=\ZZ_M$ via fiber shifts. By Corollary \ref{get-standard}, both $A$ and $B$ satisfy (T2).
\end{theorem}

\begin{theorem}\label{fibered-thm}
Let $A\oplus B=\ZZ_M$, where $M=p_i^{2}p_j^{2}p_k^{2}$ is odd. Assume that
$|A|=|B|=p_ip_jp_k$, $\Phi_M|A$, and that for every $a\in A$, the set $A\cap\Lambda(a,D(M))$ is $M$-fibered in at least one direction (possibly depending on $a$). 

\medskip\noindent
{\rm (I)} Suppose that there exists an element $a_0\in A$ such that 
\begin{equation}\label{intro-inter}
a_0*F_\nu\subset A\ \ \forall \nu\in\{i,j,k\}.
\end{equation}
Then the tiling $A\oplus B=\ZZ_M$ is T2-equivalent to $\Lambda\oplus B=\ZZ_M$ via fiber shifts, where $\Lambda:=\Lambda(a_0,D(M))$.
By Corollary \ref{get-standard}, both $A$ and $B$ satisfy (T2).

\medskip\noindent
{\rm (II)} Assume that (\ref{intro-inter}) does not hold for any $a_0\in A$. Then at least one of the following holds.

\begin{itemize}
\item $A\subset \Pi(a,p_\nu)$ for some $a\in A$ and $\nu\in\{i,j,k\}$. By Theorem \ref{subgroup-reduction}, both $A$ and $B$ satisfy (T2).

\item There exists a $\nu\in\{i,j,k\}$ such that (possibly after interchanging $A$ and $B$) the conditions of Theorem \ref{subtile} are satisfied in the $p_\nu$ direction. By Corollary \ref{slab-reduction}, both $A$ and $B$ satisfy (T2).

\end{itemize}

\end{theorem}

A more detailed breakdown of the case (II) of  Theorem \ref{fibered-thm} is provided in
Theorem \ref{fibered-mainthm}.


\subsection{Outline of the proof}\label{classified2}

In the rest of this section, we provide an outline of the proof of Theorems \ref{unfibered-mainthm} and \ref{fibered-thm}. 
We assume that $A\oplus B=\ZZ_M$, where $M=p_i^{2}p_j^{2}p_k^{2}$ is odd, 
$|A|=|B|=p_ip_jp_k$, and $\Phi_M|A$. Some of our arguments apply to tilings with more general $M$. In order to be able to sketch the main ideas without interruptions, we postpone the discussion of such extensions until the end of this section.

We begin with general arguments that are needed in both fibered and unfibered cases. 
In Section \ref{sec-toolbox}, we develop technical tools we will use throughout the article. Lemma \ref{triangles} is from \cite{LaLo1}; several of the other results in that section are specific to the 3-prime setting.

Assume first that $\Phi_M|A$ and that there exists a $D(M)$-grid $\Lambda$ such that $A\cap\Lambda$ is not $M$-fibered in any direction.
In Sections \ref{unfibered-structure-section}
and \ref{unfibered-missing}, we prove that $A\cap\Lambda$ must then contain at least one of two special structures, either {\em diagonal boxes} (Proposition \ref{db-prop}) or an {\em extended corner} (Proposition \ref{prop-ecorner}). Large parts of the argument are combinatorial and apply to all $A\subset\ZZ_M$ such that $\Phi_M|A$; however, to get the full strength of our results, we need to use saturating set techniques, hence the tiling assumption is necessary.

Some of our technical tools work only when all the ``top differences" are divisors of $A$, i.e.,
\begin{equation}\label{topdivisorsoncemore}
\{m:\ D(M)|m|M\}\subset\Div(A).
\end{equation}
We therefore must pay special attention to the cases where (\ref{topdivisorsoncemore}) fails. A classification of such structures is provided in Section \ref{unfibered-missing}. This analysis is also needed  in the fibered case (Theorem \ref{fibered-thm}) when fibering, or lack thereof, on lower scales must be considered.

We resolve diagonal boxes and extended corner structures in Sections \ref{res-boxes} and \ref{corner-section}, respectively. In Theorem \ref{db-theorem}, we prove that if $A\cap\Lambda$ contains diagonal boxes, then $A$ is T2-equivalent to either $\Lambda$ (in which case we are done) or to another tile $A'$ containing an extended corner. We then prove in Theorem \ref{cornerthm} that if $A\cap\Lambda$ contains an extended corner structure, then $A$ is T2-equivalent to $\Lambda$. Theorem \ref{unfibered-mainthm} follows by combining Theorem \ref{db-theorem} and Theorem \ref{cornerthm}. 

The main idea of that part of the proof is that all such tilings 
can be obtained via fiber shifts (Lemma \ref{fibershift}) from the tiling 
\begin{equation}\label{happyending}
A^\flat\oplus B=\ZZ_M,
\end{equation}
where $A^\flat=\Lambda(0,D(M))$. 
In the case when $B=B^\flat$ is the standard tiling complement with $\Phi_{p_i}\Phi_{p_j}\Phi_{p_k}|B^\flat$, tilings of this type were constructed by 
Szab\'o (\cite{Sz}; see also \cite{LS}). We prove that all tilings satisfying the assumptions of Theorems \ref{unfibered-mainthm} must in fact come from constructions of this type..
Starting with an unfibered grid in the given tiling, and using saturating set methods, we are able to locate the shifted fibers and shift them back into place, returning to (\ref{happyending}). This proves (T2) and provides full information about the structure of the tiling.

In Section \ref{fibered-sec}, we consider the fibered case. In the simple case when the entire set $A$ is $M$-fibered in the same direction, we can apply Corollary \ref{slab-reduction} and be done; however, it is possible for $A$ to be fibered in different directions on different $D(M)$-grids. The proof breaks down into cases, according to how the fibers in different directions interact. 

Suppose first that there exists an element $a_0\in A$ such that (\ref{intro-inter}) holds. This case turns out to be similar to that of unfibered grids and is resolved by similar methods, ending in T2-equivalence to (\ref{happyending}). 

Assume now that no such element exists. Our main intermediate result in this case is that, in fact, only two fibering directions are allowed (see Theorem \ref{fibered-mainthm} for more details). This breaks down further into cases according to fibering properties and cyclotomic divisibility, with each case terminating in either the subgroup reduction (Theorem \ref{subgroup-reduction}) or slab reduction (Theorem \ref{subtile} and Corollary \ref{slab-reduction}).

While our final result is restricted to the case when $M=p_i^{2}p_j^{2}p_k^{2}$ is odd, many of our methods and intermediate results apply under weaker assumptions. Whenever a significant part of the argument can be run in a more general case with little or no additional effort, we do so, assuming that $M=p_i^{n_i}p_j^{n_j}p_k^{n_k}$ for more general $n_i,n_j,n_k\geq 2$ and $p_i,p_j,p_k\geq 2$. 
For example, the classification of unfibered grids in Sections \ref{unfibered-structure-section}
and \ref{unfibered-missing} allows all $n_\nu$ to be arbitrary and $M$ to be either odd or even. The resolution of the $p_i$ extended corner case in Section \ref{corner-section} works for both odd and even $M$, with $n_i\geq 2$ arbitrary and with only a few additional lines needed to accommodate the even case. 
On the other hand, the arguments in Section \ref{fibered-sec} are limited to the odd $M=p_i^{2}p_j^{2}p_k^{2}$ case from the beginning.

In the follow-up article \cite{LaLo3}, we prove that our main conclusions continue to hold in the even case. However, many of our technical tools work differently when one of the primes is equal to 2. 
We would like to draw the reader's attention to the basic fibering argument in Lemma \ref{gen_top_div_mis}. This argument does not work when $p_i=2$, and indeed, in Section \ref{unfibered-missing} we provide examples of unfibered grids in the even case where the fibering conclusions of the lemma fail. 
In particular, the unfibered structures in Lemma \ref{even_struct} do not have a counterpart in the odd case.
Additionally, with fewer geometric restrictions coming from (\ref{bispan}), saturating set arguments can be more difficult to run. In \cite{LaLo3}, we compensate for this by introducing additional new methods.

The constraint $n_i=n_j=n_k=2$ is often needed in arguments based on divisor exclusion. For example, while (\ref{bispan}) provides geometric restrictions on saturating sets, we often need additional constraints based on availability of divisors, and with $n_i=n_j=n_k=2$ there are fewer divisors available to begin with. In the fibered case, several of our proofs terminate in an essentially 2-dimensional (therefore easier) problem after we have run out of scales in one direction. In order to allow arbitrary $n_i,n_j,n_k$ throughout the argument, we expect that a systematic way to induct on scales may be necessary.



\section{Toolbox}\label{sec-toolbox}


\subsection{Divisors} The first lemma is Lemma 8.9 of \cite{LaLo1}, specialized to the 3-prime case.

\begin{lemma}[\bf Enhanced divisor exclusion] \label{triangles} 
Let $A\oplus B=\ZZ_M$, with $M=\prod_{\iota\in\{i,j,k\}} p_\iota^{n_\iota}$. Let $m=\prod_{\iota\in\{i,j,k\}} p_\iota^{\alpha_\iota}$
and $m'=\prod_{\iota\in\{i,j,k\}} p_\iota^{\alpha'_\iota}$, with $0\leq \alpha_\iota,\alpha'_\iota\leq n_\iota$. 
Assume that at least
one of $m,m'$ is different from $M$, and that for every $\iota\in\{i,j,k\}$ we have 
\begin{equation}\label{triangles-e1}
\hbox{ either }\alpha_\iota\neq \alpha'_\iota \hbox{ or }
\alpha_\iota =\alpha'_\iota=n_\iota.
\end{equation}
 Then for all $x,y\in\ZZ_M$ we have
$$
\bbA^M_m[x] \,\bbA^M_{m'}[x] \, \bbB^M_m[y] \, \bbB^M_{m'}[y] =0.
$$
In other words, there are no configurations $(a,a',b,b')\in A\times A\times B\times B$ such that
\begin{equation}\label{step}
(a-x,M)=(b-y,M)=m,\ \ (a'-x,M)=(b'-y,M)=m'.
\end{equation}
\end{lemma}

\begin{proof}
If we did have a configuration as in (\ref{step}), then, under the assumption (\ref{triangles-e1}) for all $\iota$,
we would have
$$
(a-a',M)=(b-b',M)=\prod_{\iota\in\{i,j,k\}}p_\iota^{\min(\alpha_\iota,\alpha'_\iota)},
$$
with the right side different from $M$. But that is prohibited by divisor exclusion.
\end{proof}


\subsection{Cyclotomic divisibility}

\begin{lemma}\label{cyclo-grid}
Let $A\in\calm(\ZZ_M)$, and let $m,s|M$ with $s\neq 1$. Suppose that for every $a\in A$, $\Phi_s$ divides $A\cap\Lambda(a,m)$. Then $\Phi_s|A$.
\end{lemma}

\begin{proof}
Write $\ZZ_M=\bigcup_\nu \Lambda_\nu$, where $\Lambda_\nu$ are pairwise disjoint $m$-grids. Accordingly, $A(X)=\sum_\nu A_\nu(X)$, where $A_\nu=A\cap\Lambda_\nu$. If $A$ is disjoint from $\Lambda_\nu$, we have $A_\nu(X)\equiv 0$. If on the other hand $ A\cap\Lambda_\nu\neq \emptyset$, then $\Phi_s|A_\nu$. Summing up in $\nu$, we get $\Phi_s|A$.
\end{proof}

The next two lemmas are based on a combinatorial interpretation of divisibility by prime power cyclotomics. 
For $A\subset\ZZ_M$ and $1\leq \alpha\leq n_i$, we have $\Phi_{p_i^\alpha}(X)|A(X)$ if and only if 
\begin{equation}\label{cyclo-uniform}
|A\cap \Pi(x,p_i^\alpha)|=\frac{1}{p_i} |A\cap \Pi(x,p_i^{\alpha-1})| \ \ \forall x\in\ZZ_M,
\end{equation} 
so that the elements of $A$ are uniformly distributed mod $p_i^\alpha$ within each residue class mod $p_i^{\alpha-1}$.
This in particular limits the number of elements that a tiling set may have in a plane on some scale.

\begin{lemma}[\bf Plane bound]\label{planebound}
Let $A\oplus B=\ZZ_M$, where $M=p_i^{n_i}p_j^{n_j}p_k^{n_k}$ and $|A|=p_i^{\beta_i}p_j^{\beta_j}p_k^{\beta_k}$.
Then for every $x\in\ZZ_M$ and $0\leq\alpha_i\leq n_i$ we have 
$$
|A\cap\Pi(x,p_i^{n_i-\alpha_i})|\leq p_i^{\alpha_i}p_j^{\beta_j} p_k^{\beta_k}.
$$
\end{lemma}

\begin{corollary}\label{planegrid}
Let $A\oplus B=\ZZ_M$, where $M=p_i^{n_i}p_j^{n_j}p_k^{n_k}$ and $|A|=p_i^{\beta_i}p_j^{\beta_j}p_k^{\beta_k}$ 
with $\beta_i>0$.
Suppose that for some $x\in \ZZ_M$ and $1\leq\alpha_0\leq n_i$
\begin{equation}\label{Aplanebound}
|A\cap\Pi(x,p_i^{n_i-\alpha_0})|>p_i^{\beta_i-1}p_j^{\beta_j}p_k^{\beta_k},
\end{equation}
then $\Phi_{p_i^{n_i-\alpha}}|A$ for at least one $\alpha\in \{0,\ldots,\alpha_0-1\}$. 
\end{corollary}
\begin{proof}
Suppose that $\Phi_{p_i^{n_i-\alpha}}\nmid A$ for all $\alpha\in \{0,\ldots,\alpha_0-1\}$. It follows that there must exist a $\gamma$ with $\alpha_0\leq\gamma\leq n_i$ such that $\Phi_{p_i^{n_i-\gamma}}|A$. The latter implies, by (\ref{cyclo-uniform}) and (\ref{Aplanebound}), that $|A|>\prod_\nu p_\nu^{\beta_\nu}$, which is a contradiction. 
\end{proof}


\subsection{Saturating sets}

\begin{lemma}[\bf No missing joints]\label{smallcube}
Let $A\oplus B=\ZZ_M$, where $M=p_i^{n_i}p_j^{n_j}p_k^{n_k}$. Suppose that
\begin{equation}\label{notopdivB}
\{D(M)|m|M\}\cap \Div(B)=\{M\},
\end{equation}
and that for some $x\in\ZZ_M$ there exist $a_i,a_j,a_k\in A$ such that 
\begin{equation}\label{joint-e}
(x-a_i,M)=M/p_i,\ (x-a_j,M)=M/p_j,\ (x-a_k,M)=M/p_k.
\end{equation}
Then $x\in A$.
\end{lemma}

\begin{proof}
Suppose that $x\not\in A$, and let $\Delta$ be the $M$-cuboid with vertices $x,a_i,a_j,a_k$. By (\ref{joint-e}) and (\ref{two-planes}), we have the saturating set inclusions
$$
A_x\subset\Pi(x,p_\nu^{n_\nu})\cup \Pi(a_\nu,p_\nu^{n_\nu})\hbox{ for each }\nu\in\{i,j,k\}.
$$
Taking the intersection, we see that
$A_x$ is contained in the vertex set of $\Delta$. But that is impossible by (\ref{notopdivB}).
\end{proof}

\begin{lemma}[\bf Flat corner]\label{flatcorner}
Let $A\oplus B=\ZZ_M$, where $M=p_i^2p_j^{n_j}p_k^{n_k}$ and $|A|=p_ip_jp_k$. Suppose that (\ref{notopdivB}) holds, 
and that $A$ contains the following $3$-point configuration: for some $x\in \ZZ_M\setminus A$ there exist $a,a_j,a_k\in A$ such that 
\begin{equation}\label{e-flat-corner}
(a-a_j,M)=(a_k-x,M)=M/p_j,\,(a-a_k,M)=(a_j-x,M)=M/p_k
\end{equation}
Then $A_{x}\subset \ell_i(x)$, and the pair $(A,B)$ has a $(1,2)$-cofibered structure in the $p_i$ direction, with an $M$-cofiber in $A$ at distance $M/p_i^2$ from $x$.
\end{lemma}
\begin{proof}
Fix $b\in B$. 
By (\ref{e-flat-corner}) and (\ref{two-planes}), we have 
$$
A_{x,b}\subset\Pi(x,p_\nu^{n_\nu})\cup \Pi(a_\nu,p_\nu^{n_\nu})\hbox{ for each }\nu\in\{j,k\}.
$$
Taking the intersection, we see that
$$
A_{x,b}\subset \ell_i(x)\cup\ell_i(a)\cup\ell_i(a_j)\cup\ell_i(a_k).
$$
By (\ref{notopdivB}), we have the following.
\begin{itemize}
\item $A_{x,b}\cap\ell_i(x)\neq\emptyset$ implies that $\bbA_{M/p_i^2}[x]\bbB_{M/p_i^2}[b]>0$, hence 
\begin{equation}\label{flatcorner-e3}
M/p_i^2p_j, M/p_i^2p_k, M/p_i^2p_jp_k\in \Div(A) \hbox{ and }M/p_i^2\in \Div(B).
\end{equation}
\item $A_{x,b}\cap\ell_i(a)\neq\emptyset$ implies that $\bbA_{M/p_i^2p_jp_k}[x|\ell_i(a)]\bbB_{M/p_i^2p_jp_k}[b]>0$, hence 
$$M/p_i^2, M/p_i^2p_j, M/p_i^2p_k\in \Div(A) \hbox{ and }M/p_i^2p_jp_k\in \Div(B).
$$
\item $A_{x,b}\cap\ell_i(a_j)\neq\emptyset$ implies that 
$\bbA_{M/p_i^2p_k}[x|\ell_i(a_j)]\bbB_{M/p_i^2p_k}[b]>0$, hence 
$$M/p_i^2, M/p_i^2p_j, M/p_i^2p_jp_k\in \Div(A)\hbox{ and }M/p_i^2p_k\in \Div(B).
$$
\item $A_{x,b}\cap\ell_i(a_k)\neq\emptyset$ implies that $\bbA_{M/p_i^2p_j}[x|\ell_i(a_k)]\bbB_{M/p_i^2p_j}[b]>0$, hence 
$$M/p_i^2, M/p_i^2p_k, M/p_i^2p_jp_k\in \Div(A)\hbox{ and }M/p_i^2p_j\in \Div(B).
$$
\end{itemize}
It follows from divisor exclusion that $A_{x,b}$ cannot intersect more than one of the above lines. We now show it cannot intersect any line other than $\ell_i(x)$. To this end, it suffices to prove that neither one of the following can hold:
\begin{equation}\label{flatcorner-e1}
A_{x,b}\subset \ell_i(a),
\end{equation}
\begin{equation}\label{flatcorner-e2}
A_{x,b}\subset \ell_i(a_\nu) \hbox{ for some }\nu\in\{j,k\}.
\end{equation}

Assume for contradiction that (\ref{flatcorner-e1}) holds. It follows from (\ref{e-ortho2}) that
$$
1=\sum_{m|M}\frac{1}{\phi(M/m)} \bbA_m[x|\ell_i(a)]\bbB_m[b],
$$
and by (\ref{notopdivB}), the only contributing divisor can be $m={M/p_i^2p_jp_k}$. Hence
\begin{align*}	\phi(p_i^2p_jp_k)&=\bbA_{M/p_i^2p_jp_k}[x|\ell_i(a)]\bbB_{M/p_i^2p_jp_k}[b]\\
&=\bbA_{M/p_i^2}[a]\sum_{y:(y-b,M)=M/p_jp_k}\bbB_{M/p_i^2}[y].
\end{align*}
Observe that $\bbA_{M/p_i^2}[a]\leq\phi(p_i^2)$. Since $M/p_i,M/p_i^2\notin \Div(B)$, we have $\bbB_{M/p_i^2}[y]\leq 1$ for all $y\in \ZZ_M\setminus B$ with $(y-b,M)=M/p_jp_k$. It follows that both must hold with equality. Now, if $p_i>p_j$ then $\bbA_{M/p_i^2}[a]=\phi(p_i^2)$ implies that 
\begin{align*}
|A\cap \Pi(a,p_k^{n_k})|&\geq |\{a\}|+\bbA_{M/p_i^2}[a]\\
& = 1+\phi(p_i^2)>p_ip_j,
\end{align*}
which contradicts Lemma \ref{planebound}. The same argument works if $p_i>p_k$, with the $j$ and $k$ indices interchanged. 
It remains to consider the case when $p_i<\min(p_j,p_k)$. In this case we have $\phi(p_i^2)<\phi(p_jp_k)$, so that there are $y_1\neq y_2$ with $y_1-b,M)=(y_2-b,M)=M/p_jp_k$, and $b_1,b_2\in B$ with $(y_1-b_1,M)=(y_2-b_2,M)=M/p_i^2$, such that $y_1-b_1=y_2-b_2$. But then $M/p_jp_k$ divides $b_1-b_2$, contradicting (\ref{notopdivB}). 
Hence (\ref{flatcorner-e1}) cannot be true.

Next, assume that (\ref{flatcorner-e2}) holds with $\nu=k$, so that $A_{x,b}\subset \ell_i(a_k)$. In this case, (\ref{e-ortho2}) implies that
$$
1=\sum_{m|M}\frac{1}{\phi(M/m)} \bbA_m[x|\ell_i(a_k)]\bbB_m[b],
$$
and by (\ref{notopdivB}), the only contributing divisor can be $m={M/p_i^2p_j}$. Hence
\begin{align*}
\phi(p_jp_i^2)&=\bbA_{M/p_i^2p_j}[x|\ell_i(a_k)]\bbB_{M/p_i^2p_j}[b]\\
&=\bbA_{M/p_i^2}[a_k]\sum_{y:(y-b,M)=M/p_j}\bbB_{M/p_i^2}[y]
\end{align*}
By the same argument as above, we deduce that $\bbA_{M/p_i^2}[a_k]=\phi(p_i^2)$ and $\bbB_{M/p_i^2}[y]=1$ for all $y\in \ZZ_M\setminus B$ with $(y-b,M)=M/p_j$. When $p_i>p_j$, we therefore get $|A\cap \Pi(a_k,p_k^{n_k})|>p_ip_j$, which contradicts Lemma \ref{planebound}. When $p_i<p_j$, we have $\phi(p_i)<\phi(p_j)$, so that 
there are $y_1\neq y_2$ with $(y_1-b,M)=(y_2-b,M)=M/p_j$, and $b_1,b_2\in B$ with $(y_1-b_1,M)=(y_2-b_2,M)=M/p_i^2$, such that $p_i|b_1-b_2$. Hence $M/p_ip_j$ divides $b_1-b_2$, contradicting (\ref{notopdivB}).
This proves that $A_{x,b}\cap\ell_i(a_k)=\emptyset$. By symmetry, $A_{x,b}\cap\ell_i(a_j)=\emptyset$.

We have proved that $A_{x}\subset \ell_i(x)$, as claimed. 
If we know that $M/p_i\in\Div(A)$, the cofibered structure statement now follows immediately from Lemma \ref{1dim_sat-cor} with $\nu=i$. If we only assume that (\ref{notopdivB}) holds, we still have $M/p_i\not\in\Div(B)$, so that (\ref{e-ortho2}) implies
\begin{equation}\label{flatcorner-e4}
\bbA_{M/p_i^2}[x]\bbB_{M/p_i^2}[b]=\phi(p_i^2)\hbox{ for all }b\in B.
\end{equation}
Moreover, since $M/p_i\not\in\Div(B)$ and (by (\ref{flatcorner-e3})) $M/p_i^2\not\in\Div(A)$,  
we must have $\bbA_{M/p_i^2}[x]\leq p_i$ and $\bbB_{M/p_i^2}[b]\leq \phi(p_i)$ for each $b\in B$. By (\ref{flatcorner-e4}), both must hold with equality. This implies the desired cofibered structure, with the cofiber in $A$ equal to $A\cap\ell_i(x)$.
\end{proof}


\subsection{Fibering lemmas}


Lemma \ref{2d-cyclo} below is a simple version of the de Bruijn-R\'edei-Schoenberg theorem for cyclic groups $\ZZ_N$, where $N$ has at most two distinct prime factors. This was essentially proved in \cite{deB}; see also \cite[Theorem 3.3]{LL}.

\begin{lemma}\label{2d-cyclo} {\bf (Cyclotomic divisibility for 2 prime factors)}
Let $ A\in \mathcal{M}(\ZZ_N)$ for some $N\mid M, M=p_i^{n_i}p_j^{n_j}p_k^{n_k}$ such that $p_i\nmid N$. Then:

\smallskip
(i) $\Phi_N|A$ if and only if $A$ is a linear combination of $N$-fibers in the $p_j$ and $p_k$ direction with non-negative integer coefficients.

\smallskip
(ii) Let $\Lambda$ be a $D(N)$-grid. Assume that $\Phi_N|A$, and that there
exists $c_0\in\NN$ such that $\bbA^N_N[x]\in\{0,c_0\}$ for all $x\in\Lambda$. 
Then $A\cap\Lambda$ is $N$-fibered in either the $p_j$ or the $p_k$ direction.
\end{lemma}

Lemma \ref{flat-cuboids} is a localized version of the above.

\begin{lemma}\label{flat-cuboids} {\bf (Flat cuboids)}
Let $A\subset \ZZ_M$. Assume that $\Phi_M|A$, and that there is a plane $\Pi:=\Pi(z,p_i^{n_i})$ such that
$A\cap\Pi$ is a disjoint union of $M$-fibers in the $p_j$ and $p_k$ directions. 
 Then for every parallel plane $\Pi':=\Pi(z',p_i^{n_i})$, where
 $(z-z',M)=M/p_i$, the set $A\cap\Pi'$ is a disjoint union of $M$-fibers in the $p_j$ and $p_k$ directions.
\end{lemma}

\begin{proof}
Consider a 2-dimensional cuboid $\Delta'$ with vertices $x',x'+d_j,x'+d_k,x'+d_j+d_k$, where 
 $x'\in\Pi'$, $(d_j,M)=M/p_j$, $(d_k,M)=M/p_k$. Let $x\in\Pi$ be the point such that $(x-x',M)=M/p_i$,
and let $\Delta$ be the 2-dimensional cuboid with vertices $x,x+d_j,x+d_k,x+d_j+d_k$.
By the fibering property of $A\cap\Pi$, we have
$\bbA_M[\Delta]=0$. Since $\Phi_M|A$, we also have $\bbA_M[\Delta-\Delta']=0$, hence
$\bbA_M[\Delta']=0$. If we consider $A\cap\Pi'$ (after translation) as a subset of $\ZZ_{M/p_i^{n_i}}$,
it follows that $\Phi_{M/p_i^{n_i}}\mid (A\cap \Pi')(X)$. By Lemma \ref{2d-cyclo}, $A\cap\Pi'$ is a union of fibers as claimed.
\end{proof}

\begin{lemma} \label{gen_top_div_mis} {\bf (Missing top difference implies fibering)} 
Let $N|M$ with $p_ip_jp_k|N$.  Let $\Lambda:=\Lambda(x_0,D(N))$ for some $x_0\in \ZZ_N$.
Assume that $ A\subset \ZZ_M$ satisfies $\Phi_N|A$ and $\Lambda\cap A\neq\emptyset$.
Assume further that there exists a constant $c_0\in\NN$ such that
$$
\bbA^N_N[x]\in\{0,c_0\}\hbox{ for all }x\in\ZZ_N.
$$

\smallskip
(i) Suppose that $p_i\neq 2 $, and that 
\begin{equation}\label{topdiv_restr}
N/p_i\not\in\Div_{N}(A\cap\Lambda).
\end{equation}
Then $A\cap\Lambda$ is $N$-fibered in one of the $p_j$ and $p_k$ directions. In particular, if $N/p_i\not\in\Div_N(A)$, then
$A\cap\Lambda$ is $N$-fibered in one of the $p_j$ and $p_k$ directions for every $D(N)$-grid $\Lambda$.

\smallskip

(ii) Suppose that $N/p_i,N/p_j\not\in\Div_N(A\cap\Lambda)$. Then
$A\cap\Lambda $ is $N$-fibered in the $p_k$ direction.
\end{lemma}

\begin{proof}
(i) We will assume that $c_0=1$ and identify $A$ with the set $A$ mod $N$ in $\ZZ_N$. (The general case is identical, except that every element of $A$ mod $N$ has multiplicity $c_0$ instead of 1.)

We first prove that each $a\in A\cap\Lambda$ belongs to an $N$-fiber in the $p_{\nu}$ direction for at least one $\nu\in\{j,k\}$. Suppose, for contradiction, that there exists an $a\in A\cap\Lambda$ that does not have this property. Then there are $x_j, x_k\in \ZZ_N\setminus A$ such that 
$$
 (x_j-a,N)=N/p_j,\,(x_k-a,N)=N/p_k.
$$
Since $ p_i>2 $ and $\bbA^N_{N/p_i}[a]=0$, there exist at least two distinct elements $ x_i,x_i'\in \ZZ_N\setminus A$ satisfying 
$$
(a-x_i,N)=(a-x_i',N)=N/p_i.
$$
Consider two $N$-cuboids in $\ZZ_N$, each with vertices at $ a,x_j, x_k$, and with another vertex at $x_i$ and $x_i'$ respectively. 
By the cyclotomic divisibility assumption, each of those cuboids must be balanced. This can only happen if there are two elements $a_{ijk},a_{ijk}'\in A$ at the opposite vertex of each cuboid from $a$ (that is, with $(a-a_{ijk},N)=(a-a_{ijk}',N)=D(N)$).
However, this leads to a contradiction, since $ (a_{ijk}-a'_{ijk},N)=N/p_i$. We therefore conclude that 
each $a\in A\cap\Lambda$ belongs to an $N$-fiber in at least one direction as indicated.

Next, suppose that $a_j,a_k\in A\cap\Lambda$ belong to $N$-fibers in, respectively, the $p_{j}$ and $p_k$ direction.
If $a_j*F_j$ and $a_k*F_k$ do not intersect, then $N/p_i\in\Div_N (a_j*F_j,a_k*F_k)$, contradicting the assumption (\ref{topdiv_restr}). We may therefore assume that $a_j=a_k=a$ and that $a*F_j,a*F_k\subset A$.

Consider any $N$-cuboid with one vertex at ${a}$. Then the vertices at distance $N/p_j$ and $N/p_k$ from $a$ belong to $A$, and, by (\ref{topdiv_restr}), the vertices at distance $N/p_ip_j$ and $N/p_ip_k$ from $a$ cannot be in $A$. 
The only way to balance the cuboid is for the vertex at distance $N/p_jp_k$ from $a$ to be in $A$. 
Allowing such cuboids to vary, we see that $a*F_j*F_k\subset A$. This also implies that $A$ cannot have any other elements in $\Lambda$, since that would contradict  (\ref{topdiv_restr}). Hence $A\cap\Lambda=a*F_j*F_k$ is $N$-fibered in both of the $p_j$ and $p_k$ directions.  This proves part (i) of the lemma.

\smallskip
(ii) 
Assume that $N/p_i,N/p_j\not\in\Div_N(A\cap\Lambda)$. At least one of $p_i,p_j$ must be odd; without loss of generality, we may assume that $p_i\neq 2$. By part (i) of the lemma, $A\cap\Lambda$ must be $N$-fibered in at least one of the $p_j$ and $p_k$ directions on $\Lambda$. However, it cannot be $N$-fibered in the $p_j$ direction, since $N/p_j\not\in\Div_N(A\cap\Lambda)$. Part (ii) follows.
\end{proof}


\section{Structure on unfibered grids}\label{unfibered-structure-section}



\subsection{Diagonal boxes}\label{diag-box-section}


Throughout this section, we will use the following notation. Let $M=p_i^{n_i}p_j^{n_j}p_k^{n_k}$, and let
$\Lambda$ be a fixed $D(M)$-grid such that $A\cap\Lambda\neq\emptyset$. 
We identify $\Lambda$ with $\ZZ_{p_i}\oplus \ZZ_{p_j} \oplus \ZZ_{p_k}$, and represent
each point $x\in \Lambda$ as $(\lambda_ix,\lambda_j x, \lambda_k x)$ in the implied coordinate system.

\begin{definition}\label{def-db}
Let $A\subset\ZZ_M$. We say that $A\cap\Lambda$ {\em contains diagonal boxes} if
there are nonempty sets 
$I\subset \ZZ_{p_i}$, $J\subset \ZZ_{p_j}$, $K\subset \ZZ_{p_k}$, such that  
$$
I^c:=\ZZ_{p_i}\setminus I,\  J^c:=\ZZ_{p_j}\setminus J,\  K^c:=\ZZ_{p_k}\setminus K
$$
are also nonempty, and 
$$
(I\times J\times K) \cup (I^c\times J^c\times K^c)\subset A\cap\Lambda.
$$
\end{definition}

\begin{figure}[h]
	\includegraphics[scale=0.65]{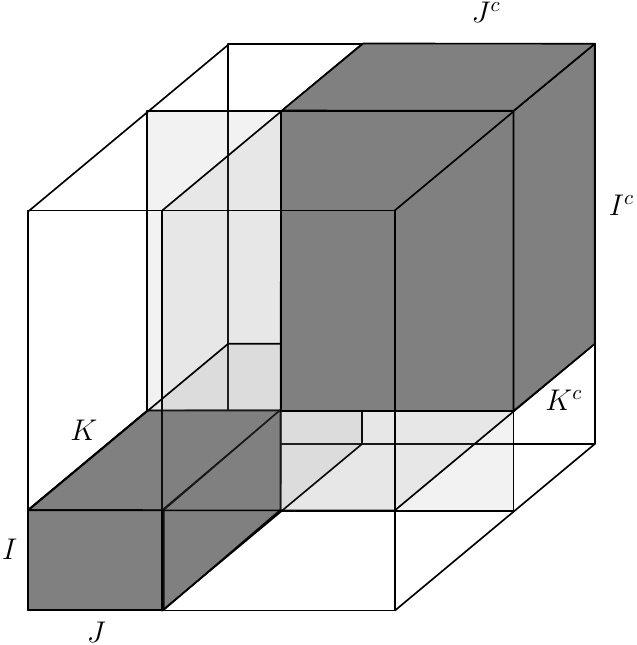}
	\caption{A pair of diagonal boxes.}
\end{figure}

\begin{proposition}\label{db-prop}
Let $A\oplus B=\ZZ_M$, where $M=p_i^{n_i}p_j^{n_j}p_k^{n_k}$, and assume that $\Phi_M\mid A$. Suppose that there is a $D(M)$-grid $\Lambda$ such that
$A\cap\Lambda$ is not a union of disjoint $M$-fibers (possibly in different directions).
Then $A\cap\Lambda$ contains diagonal boxes.
\end{proposition}

\begin{proof}
Write $D=D(M)$ for short.
We first construct a set $A_0\subset A\cap\Lambda$ as follows. If $A\cap\Lambda$ contains no $M$-fibers, let $A_0:=A\cap\Lambda$. If $A\cap \Lambda$ does contain an $M$-fiber $F$, consider the set $(A\cap\Lambda)\setminus F$ (if there is more than one such fiber, just choose one arbitrarily and remove it). If this set contains no $M$-fibers, we let $A_0$ be that set; otherwise continue by induction. The procedure terminates when no more $M$-fibers can be found. The remaining set $A_0$ is nonempty by our assumption on $A$, contains no $M$-fibers,
and $\Phi_M | A_0$; however, $A_0$ need not be a tiling complement. 

We remark that $A_0$ is not necessarily uniquely determined by $A$, as the fiber removal procedure may lead to different outcomes depending on the order in which fibers are removed. In that event, we fix one such set $A_0$ and keep it fixed throughout the proof. 

For future reference, we record a lemma.

\begin{lemma}\label{takesomegetsome}
Suppose that $x\in A\setminus A_0$. Then there exists a $\nu\in\{i,j,k\}$ such that $x*F_\nu\subset A\setminus A_0$.
\end{lemma}

\begin{proof}
 This follows directly from the construction, since any point $x\in A\setminus A_0$ would have been removed from $A$ together with an $M$-fiber (in some direction) containing $x$.
\end{proof}

For each $x\in\Lambda$, we define
\begin{align*}
I(x)&=\{l\in\ZZ_{p_i}:\ \ (l,\lambda_jx,\lambda_kx)\in A_0\}, & I^c(x)=\ZZ_{p_i}\setminus I(x),
\\
J(x)&=\{l\in\ZZ_{p_j}:\ \ (\lambda_i x,l,\lambda_k x)\in A_0\}, &J^c(x)=\ZZ_{p_j}\setminus J(x),
\\
K(x)& =\{l\in\ZZ_{p_k}:\ \ (\lambda_i x, \lambda_j x,l)\in A_0\}, & K^c(x)=\ZZ_{p_k}\setminus K(x).
\end{align*}
Let $a\in A_0$ be an element such that $|K(a)|$ is maximal, in the sense that
\begin{equation}\label{Kmax}
|K(a')|\leq |K(a)|\ \ \forall a'\in A_0.
\end{equation}
By translational invariance, we may assume that $a=(0,0,0)$. 
Observe that $I(a),J(a), K(a)$ are all nonempty since $a\in A_0$, and $I^c(a),J^c(a), K^c(a)$ are all nonempty since 
$a$ does not belong to an $M$-fiber in $A_0$ in any direction.

\medskip
\noindent
{\bf Claim 1.} {\it For all $a'=(0,0,l)\in A_0$, with $l\in K(a)$, we have $I(a')\times J(a')\times \{l\}\subset A_0$.}


\begin{proof}
 Since $K(a')=K(a)$, it suffices to prove this with $a'=a$. If $|I(a)|=1$ or $|J(a)|=1$, there is nothing to prove. Assume now that $\min(|I(a)|,|J(a)|\geq 2$. Let $a_i=(l_i,0,0)$ and $a_j=(0,l_j,0)$ for some $l_i\in I(a)\setminus\{0\}$,
$l_j\in J(a)\setminus\{0\}$, and let $z=(l_i,l_j,0)\in I(a)\times J(a)\times \{0\}$. We need to prove that $z\in A_0$. Suppose, for contradiction, that $z\not\in A_0$, and consider the cuboid with vertices at $a,a_i,a_j$, and
$(0,0,l_k)$. 
For $l_k\in K(a)$, the cuboid can only be balanced if both $(l_i,0,l_k)$ and $(0,l_j,l_k)$ are elements of $A_0$. Furthermore, let $l_k\in K^c(a)$, then for the cuboid to be balanced we still need at least one of $(l_i,0,l_k)$ and $(0,l_j,l_k)$ to be in $A_0$. This implies that $\max(|K(a_i)|,|K(a_j)|)>|K(a)|$,
contradicting the maximality assumption (\ref{Kmax}). Hence Claim 1 follows.
\end{proof}

\noindent
{\bf Claim 2.} {\it For all $a'=(0,0,l)\in A_0$ with $l\in K(a)$,
we have $I^c(a')\times J^c(a')\times K^c(a)\subset A_0$.}


\begin{proof}
Let $x\in I^c(a')\times J^c(a')\times K^c(a)$.
Considering the $M$-cuboid with opposite vertices at $a'$ and $x$, we see that the three vertices at distance $M/p_i,M/p_j,M/p_k$ from $a'$ are not in $A_0$. Hence, in order for the cuboid to
be balanced, we must have $x\in A_0$ as claimed. 
\end{proof}

\noindent
{\bf Claim 3.} {\it Suppose that for all $l',l''\in K(a)$, we have either
$(I(a')\times J(a'))\subseteq (I(a'')\times J(a'')) $ or $(I(a'')\times J(a''))\subseteq (I(a')\times J(a')) $,
where $a'=(0,0,l')$, $a''=(0,0,l'')$.
Then $A_0$ contains diagonal boxes.
}


\begin{proof} Under the assumptions of the claim, the sets $I(a')\times J(a')$ with $a'\in A_0$ and $M/p_k\mid a-a'$ have a minimal element. Without loss of generality, we may assume that
$$
(I(a)\times J(a))\subseteq (I(a')\times J(a'))\ \ \forall a'\in A_0 \hbox{ such that }M/p_k\mid a-a'.
$$
Then Claims 1 and 2 imply that
$A_0$ (and therefore $A$) contains the diagonal boxes 
$$(I(a)\times J(a)\times K(a)) \cup (I^c(a)\times J^c(a)\times K^c(a)),
$$
which proves the claim. \end{proof}

\medskip

It remains to consider the case when there exist $l',l''\in K(a)$ such that for $a'=(0,0,l')$, $a''=(0,0,l'')$, we have
\begin{equation}\label{db-e8}
I(a')\not\subset I(a'')\hbox{ and }J(a'')\not\subset J(a').
\end{equation}
For this to be possible, $K(a)$ must have at least two distinct elements, and each of $I(a')$ and $J(a'')$ must have at least one element different from $0$. 
Since none of $a,a',a''$ belongs to an $M$-fiber in $A_0$, we must have
$$
p_i,p_j,p_k>2.
$$
Furthermore, this configuration implies that
\begin{equation}\label{db-e10}
\{m:\ D|m| M\}\subset\Div(A_0)\subset\Div(A\cap\Lambda).
\end{equation}
Indeed, we have 
$$
\{m:\ D|m| M\}\setminus\{M/p_ip_j\}\subset 
\Div ((I(a')\times J(a')\times \{l'\})\cup  (I(a'')\times J(a'')\times \{l''\})),
$$
$$
M/p_ip_j\in\Div(I^c(a')\times J^c(a')\times \{l_k\}, I^c(a'')\times J^c(a'')\times \{l_k\}) \hbox{ for any }l_k\in K^c(a).
$$

\medskip

\noindent
{\bf Claim 4.} {\it Suppose that  (\ref{db-e8}) holds for some $l',l''\in K(a)$, with $a'=(0,0,l')$, $a''=(0,0,l'')$. Then 
$$I^c(a'')\times J^c(a')\times K^c(a)\subset A_0.$$
}


\begin{proof} 
We may assume that $a''=a$ and $l''=0$. We will also write $l_k=l'$, so that $a'=(0,0,l_k)$
Let $l_i\in I(a')\setminus I(a), l_j\in J(a)\setminus J(a'), l_i^c\in I^c(a')$ and $l_j^c\in J^c(a)$. We also fix $l^c_k\in K^c(a)$. For the purpose of this proof, we will need to consider points with coordinates $(\beta_i,\beta_j,\beta_k)$ such that 
$$
\beta_\nu\in \{0,l_\nu,l_\nu^c\} \text{ for } \nu\in\{i,j,k\}.
$$
Let $x_i:=(l_i,0,0)$, $x_j:=(0,l_j,l_k)$, and $a_i:=(l_i,0,l_k),a_j:=(0,l_j,0)$.  Then $x_i,x_j  \not\in  A_0$, and $a_i,a_j\in A_0$. By Claim 2, we have
\begin{equation}\label{zlocation}
\begin{split}
z_i &:=(l_i^c ,l_j,l_k^c)\in A_0 \hbox{ since } l_j\in J^c(a'),\\
z_j &:=(l_i,l_j^c,l_k^c)\in A_0 \hbox{ since } l_i\in I^c(a),\\
\end{split}
\end{equation}
Let $z=(l_i,l_j,l_k^c)\in I^c(a)\times J^c(a')\times K^c(a)$ be the point such that
\begin{equation}\label{ztriple}
(z-z_i,M)=M/p_i,\ (z-z_j,M)=M/p_j.
\end{equation}
We need to prove that 
\begin{equation}\label{claim4-ez}
z\in A_0. 
\end{equation}

It is tempting at this point to try to apply the flat corner argument to $z$, with the flat corner configuration given by $z_i$, $z_j$, and $z_{ij}:=(l_i^c,l_j^c,l_k^c)$. Unfortunately, we do not know that $A_0$ is a tiling complement. This means that an additional argument is required even in the basic case $n_i=n_j=n_k=2$.  With more work, we are also able to extend that argument to more general $n_i,n_j,n_k$.

We begin with the following reduction.

\medskip

\noindent
{\bf Claim 4'.} {\it Let $x_{ij}=(l_i,l_j,0)$ and $y_{ij}=(l_i,l_j,l_k)$. Suppose that at least one of the following holds: either
\begin{equation}\label{fiberacja3}
\{x_{ij},y_{ij}\}\cap A\neq\emptyset,
\end{equation}
or else there exists a set $A'\subset\ZZ_M$, identical to $A$ except possibly along the line $\ell_j(x_i)$, such that $A'\oplus B=\ZZ_M$ and 
\begin{equation}\label{fiberacja2}
	x_i*F_j\subset A', \text{ hence } x_{ij}\in A'.
\end{equation}
Then $z\in A_0$. 
}

\begin{proof}
For notational consistency, if (\ref{fiberacja3}) holds, we let $A'=A$. We have
\begin{equation}\label{zdoubt}
(x_{ij}-z,M)=(y_{ij}-z,M)=M/p_k.
\end{equation}
Suppose that at least one of $x_{ij}$, $y_{ij}$ belongs to $A'$. By (\ref{ztriple}) and 
(\ref{db-e10}), the assumptions of Lemma \ref{smallcube} hold for $z$ and $A'$. Hence $z\in A'$.
Since $A'$ and $A$ may differ only along $\ell_j(x_i)$, we must in fact have $z\in A$. 

Next, we claim that $z\in A_0$. Indeed, assume for contradiction that $z\in A\setminus A_0$. By Lemma \ref{takesomegetsome}, there is a $\nu\in\{i,j,k\}$ such that $z*F_\nu\subset (A\setminus A_0)$. But by (\ref{zlocation}) and (\ref{ztriple}), $z_i\in (z*F_i)\cap A_0$ and  $z_j\in (z*F_j)\cap A_0$. It follows that $\nu=k$, and $z*F_k\subset (A\setminus A_0)$. By (\ref{zdoubt}), we have $x_{ij},y_{ij}\in z*F_k$, and in particular $x_{ij},y_{ij}\in A$. 

Consider now the point $x_{i}$. We have 
\begin{equation}\label{pkpair}
(x_i-a,M)=M/p_i,\ \ (x_i-a_i,M)=M/p_k,\ \ (x_i-x_{ij},M)=M/p_j, 
\end{equation}
with $a,a'\in A_0,x_{ij}\in A$.
Taking also (\ref{db-e10}) into account, we see that $x_i$ satisfies the assumptions of Lemma \ref{smallcube} applied to $A$. Therefore $x_i\in A$.

Since $x_i\not\in A_0$, we must have  $x_i\in A\setminus A_0$. By Lemma \ref{takesomegetsome} and (\ref{pkpair}), it follows that $x_i*F_j\subset A\setminus A_0$. However, consider the point $x_{ij}\in x_i*F_j$. We have already seen that $z$ must have been removed from $A$ together with the fiber $z*F_k$, which also contains $x_{ij}$. This is a contradiction, since we are not allowed to remove the same point twice. 
\end{proof}

To prove Claim 4, it remains to prove that at least one of (\ref{claim4-ez}), 
(\ref{fiberacja3}), (\ref{fiberacja2}) must hold. Assume first that $x_i\in A\setminus A_0$. By Lemma \ref{takesomegetsome}, $A\setminus A_0$ must contain an $M$-fiber $x_i*F_\nu$ for some $\nu\in\{i,j,k\}$. Since $a\in (x_i*F_i)\cap A$ and $a_i\in (x_i*F_k)\cap A$,  we must have $\nu=j$. This clearly implies $x_{ij}\in A $. The same argument applies with $i$ and $j$ interchanged.

We are left with the case when (\ref{claim4-ez}) fails and
\begin{equation}\label{claim4-e4}
x_i,x_j,x_{ij},y_{ij}\notin A.
\end{equation}
Consider the $M$-cuboid with vertices at $x_{ij},z, z_i$ and $z_j$, and note that $z_i,z_j\in A_0$. In order to balance this cuboid in $A_0$, at least one of the points $y_i:=(l_i^c ,l_j,0)$ and $y_j:=(l_i,l_j^c,0)$ must be in $A_0$. However, if $y_j\in A_0$, then we can apply Lemma \ref{smallcube} again, this time to $A$ and $x_i$, using (\ref{db-e10}), (\ref{pkpair}), and 
$(x_i-y_j,M)=M/p_j$. Hence $x_i\in A$. This, however, contradicts the assumption (\ref{claim4-e4}) of this case.


We therefore have $y_i\in A_0$. Now, considering the saturating set $A_{x_i}$, we claim that 
\begin{equation}\label{claim4-e0}
A_{x_i}\subset \bigcap_{\hat{a}\in \{a,a_i,z_j,y_i \}} \Bispan(x_i,\hat{a}) = \ell_j(x_i) \cup \{a,a',a_i, a_j, (l_i,l_j^c,l_k)\}.
\end{equation}
To prove this, we argue as follows. The first inclusion follows from Lemma \ref{bispan-lemma} (ii), since $a,a_i,z_j,y_i\in A$. We now prove the second part of the formula. By (\ref{two-planes}) applied to $x_i$, $a$ and $a_i$ we have 
\begin{equation}\label{claim4-e1}
\begin{split}
A_{x_i}&\subset \Bispan(x_i,a)\cap\Bispan(x_i,a_i)
\\
&= \Big(\Pi(x_i,p_i^{n_i})\cup \Pi(a,p_i^{n_i})\Big)
\cap \Big(\Pi(x_i,p_k^{n_k})\cup \Pi(a_i,p_k^{n_k})\Big)
\\
&= \ell_j(a)\cup\ell_j(a')\cup\ell_j(a_i)\cup\ell_j(x_i).
\end{split}
\end{equation}
By (\ref{e-span}), we have
$$
\Bispan(x_i,y_i)=\Pi(x_i, p_i^{n_i})\cup\Pi(x_i, p_j^{n_j})\cup \Pi(y_i, p_i^{n_i})\cup\Pi(y_i, p_j^{n_j}).
$$
Taking the intersection with (\ref{claim4-e1}), we get 
$$
A_{x_i}\subset \ell_j(a_i)\cup\ell_j(x_i)\cup\{a,a',a_j,x_j\}.
$$
Finally, 
$$
\Bispan(x_i,z_j)=\Pi(x_i, p_j^{n_j})\cup\Pi(x_i, p_k^{n_k})\cup \Pi(z_j, p_j^{n_j})\cup\Pi(z_j, p_k^{n_k}).
$$
Taking the intersection again, we get (\ref{claim4-e0}).

By (\ref{db-e10}), we must in fact have $A_{x_i}\subset \ell_j(x_i)$. Hence $A_{x_i}$ satisfies the conditions of Lemma \ref{fibershift-plus}. In particular, (\ref{fibershift-plus-e1}) holds with some $\gamma\geq 2$, and $A_{x_i}$ is a disjoint union of $M$-fibers in the $p_j$ direction. Applying Lemma \ref{fibershift-plus} (iii) to shift $A_{x_i}$ to $x_i$, we obtain a set $A'$ with the desired properties, thus concluding the proof of Claim 4.
\end{proof}

We can now complete the proof of the proposition. We claim that 
$$
(I\times J\times K) \cup (I^c\times J^c\times K^c)\subset A_0,
$$
where 
$$
I=\bigcap_{l\in K(a)} I(a_l),\ \ J=\bigcap_{l\in K(a)} J(a_l),\ \ K=K(a),
$$
$$
I^c=\bigcup_{l\in K(a)} I^c(a_l),\ \ J^c=\bigcup_{l\in K(a)} J^c(a_l),\ \ K^c=K^c(a),
$$
and $a_l=(0,0,l)$.
 It is clear that $I\times J\times K\subset A_0$. Suppose now that
$(l_i,l_j,l_k)\in I^c\times J^c\times K^c$. Then there are $l,l'\in K(a)$ such that
$l_i\in I^c(a_l)$ and $l_j\in J^c(a_{l'})$. If $l=l'$, or if $l_i\in I^c(a_{l'})$, then 
$(l_i,l_j,l_k)\in A_0$ by Claim 2 applied to $a_{l'}$. The case $l_j\in J^c(a_l)$ is similar. 
Assume therefore that
$l_i\in I(a_{l'})\setminus I(a_{l})$ and $l_j\in J(a_l)\setminus J(a_{l'})$.
But then (\ref{db-e8}) holds with $a',a''$ replaced by $a_l,a_{l'}$. Applying Claim 4, we see that
$(l_i,l_j,l_k)\in A_0$ in that case as well. This ends the proof of the proposition.
\end{proof}


\subsection{Extended corners}\label{ec-section}


We continue to write $M=p_i^{n_i}p_j^{n_j}p_k^{n_k}$. Assume that $A\oplus B=\ZZ_M$, $\Phi_M\mid A$,
and there exists a $D(M)$-grid $\Lambda$ such that $A\cap\Lambda\neq\emptyset$ and $A\cap \Lambda$ is not $M$-fibered in any direction. 
By Proposition \ref{db-prop}, if 
$A\cap\Lambda$ is not a union of disjoint $M$-fibers, then $A\cap\Lambda$ contains diagonal boxes.
It remains to consider the case when $A\cap\Lambda$ is a union of disjoint $M$-fibers. In that case, we claim that $A\cap\Lambda$ contains the following structure.

\begin{definition}\label{corner}
Suppose that $ A\subset \ZZ_M$, and let $\Lambda$ be a $D(M)$-grid.

\smallskip

(i) We say that $A\cap\Lambda$ contains a {\em $p_i$ corner} if there exist $a,a_i\in A\cap\Lambda$
with $(a-a_i,M)=M/p_i$ satisfying  
$$
A\cap(a*F_j*F_k)= a*F_j,\ \ A\cap(a_i*F_j*F_k)= a_i*F_k.
$$
	
\smallskip
(ii) We say that $A\cap\Lambda$ contains a
{\em $p_i$ extended corner} if there exist $a,a_i\in A\cap\Lambda$
such that $(a-a_i,M)=M/p_i$ and 
\begin{itemize}
\item $A\cap(a*F_j*F_k)$ is $M$-fibered in the $p_j$ direction but not in the $p_k$ direction,
\item $A\cap(a_i*F_j*F_k)$ is $M$-fibered in the $p_k$ direction but not in the $p_j$ direction.
\end{itemize}
	
\end{definition}

We now prove that unfibered grids as described above must contain extended corners.

\begin{proposition}\label{prop-ecorner}
	Let $D=D(M)$, and let $\Lambda$ be a $D$-grid. Assume that $A\cap\Lambda$ is a union of disjoint $M$-fibers, but is not fibered in any direction. Then $A\cap\Lambda$ contains a $p_\nu$ extended corner for some $\nu\in\{i,j,k\}$.
\end{proposition}

\begin{proof}
Fix $A$ and $\Lambda$ as in the statement of the proposition.
	We will say that $\kappa:A\cap\Lambda\to\{i,j,k\}$ is an {\em assignment function} if $A\cap\Lambda$ can be written as
	$$
	A\cap\Lambda=\bigcup_{a\in A\cap\Lambda} (a*F_{\kappa(a)}),
	$$
	where for any $a,a'\in A\cap\Lambda$, the fibers $a*F_{\kappa(a)}$ and $a'*F_{\kappa(a')}$ are either identical or disjoint.
	Thus, if $a'\in a*F_{\kappa(a)}$, then $\kappa(a')=\kappa(a)$.
	
We recall \cite[Proposition 7.10]{LaLo1}:

\begin{proposition}\label{prop-twodirections}
Let $M=p_i^{n_i}p_j^{n_j}p_k^{n_k}$, and let $D=D(M)$. Assume that $A\oplus B=\ZZ_M$ 
	and that there exists a $D$-grid $\Lambda$ such that $A\cap \Lambda$ is a nonempty union of disjoint $M$-fibers.
	Then there is a subset $\{\nu_1,\nu_2\}\subset\{i,j,k\}$ of cardinality 2 such that $A\cap\Lambda$ is a union of disjoint $M$-fibers in the $p_{\nu_1}$ and $p_{\nu_2}$ directions.
\end{proposition}

Hence there exists an assignment function $\kappa$ that takes at most two distinct values. (In fact, the proof of \cite[Proposition 7.10]{LaLo1} shows that this is true for any assignment function.) 
	Without loss of generality, we may assume that $\kappa(a)\in\{j,k\}$ for all $a\in A\cap\Lambda$. We claim that this implies that $A\cap\Lambda$ contains a $p_i$ extended corner.
	
	Split $\Lambda$ into 2-dimensional grids 
	$\Pi_\iota:={x_\iota}*F_j*F_k$, $\iota=0,1,\dots,p_i-1$. Then for each $\iota$, the set $A\cap\Pi_\iota$ is a union of disjoint fibers in at least one of the $p_j$ or $p_k$ direction. Moreover, we are assuming that $A\cap\Lambda$ is not fibered in either the $p_j$ or $p_k$ direction. Therefore for each $\nu\in\{j,k\}$, there must be at least one $\iota(\nu)$ such that 
	$A\cap\Pi_{\iota(\nu)}$ is fibered only in the $p_\nu$ direction. Choosing $a\in A\cap \Pi_{\iota(j)}$
	and $a_i\in A\cap \Pi_{\iota(k)}$ with $(a-a_i)=M/p_i$, we see that the condition (ii) of Definition \ref{corner} is satisfied.
\end{proof}

A $p_i$ corner is a special case of a $p_i$ extended corner, with only one fiber in each of the planes through $a$ and $a_i$ in $\Lambda$. This is one of the special structures that occur when $\Phi_M\mid A$,
$A\cap \Lambda$ is not $M$-fibered in any direction, and $\{D(M)|m|M\}\not\subset\Div(A)$ 
(see Section \ref{unfibered-missing}).

In addition to the present purpose of classification of unfibered grids on scale $M$, 
we will also refer to Definition \ref{corner} (ii) in the fibered case, in the proofs of Proposition \ref{initialstatement} (Claim 1) and Lemma \ref{calknointer}.


\section{Unfibered grids with missing top differences}\label{unfibered-missing}


Let $A\subset\ZZ_M$, and let $\Lambda$ be a $D(M)$-grid such that $A\cap\Lambda\neq\emptyset$.
The purpose of this section is to classify all possible unfibered grids $A\cap\Lambda$ under the assumption that $\Phi_M|A$ and that $\Div(A)$ does not contain all $m$ such that $D(M)|m|M$. We do {\em not} assume in this section that $A$ is a tiling complement.

\subsection{A structure result}

\begin{proposition}\label{notopdiff}
Let $M=p_i^{n_i}p_j^{n_j}p_k^{n_k}$.
Assume that $A\subset\ZZ_M$ satisfies $\Phi_M|A$. Let 
$\Lambda$ be a $D(M)$-grid such that $A\cap\Lambda\neq\emptyset$.
Suppose that $A\cap\Lambda$ is not $M$-fibered in any direction, 
and that (\ref{db-e10}) fails, i.e.
\begin{equation}\label{db-e18}
\{m:\ D(M)|m| M\}\not\subset\Div(A\cap\Lambda).
\end{equation}
Then $A\cap\Lambda$ is a union of at most one set of diagonal boxes
\begin{equation}\label{db-e19}
A_1=(I_1\times J_1\times K_1) \cup (I_1^c\times J_1^c\times K_1^c),
\end{equation}
where $I_1\subset \ZZ_{p_i}$, $J_1\subset \ZZ_{p_j}$, $K_1\subset \ZZ_{p_k}$ are non-empty sets such that  
$
I_1^c:=\ZZ_{p_i}\setminus I_1,\  J_1^c:=\ZZ_{p_j}\setminus J_1,\  K_1^c:=\ZZ_{p_k}\setminus K_1
$
are also non-empty, and possibly additional $M$-fibers in one or more directions, disjoint from $A_1$ and from each other. Furthermore, if $A\cap\Lambda$ does contain a set of diagonal boxes (\ref{db-e19}), then at least one of the sets
$I_1,J_1,K_1$ has cardinality 1, and at least one of the sets
$I^c_1,J^c_1,K^c_1$ has cardinality 1.
\end{proposition}

\begin{remark}
For simplicity, we only state and prove
Proposition \ref{notopdiff} for sets $A\subset \ZZ_M$. However, if we assume instead that $A\in\calm(\ZZ_N)$ for some $N|M$, and that (\ref{bin_cond}) holds (i.e., $A\cap\Lambda$ is a multiset of
constant multiplicity $c_0$), the same argument applies except that the diagonal boxes and fibers in the conclusion also have multiplicity $c_0$.
\end{remark}

\begin{proof}
We begin as in the proof of Proposition \ref{db-prop}. 
Let $A_0\subset A\cap\Lambda$ be a set constructed by removing $M$-fibers from $A\cap\Lambda$ until none are left, so that $A\cap\Lambda$ is the union of $A_0$ and some number of $M$-fibers in one or more directions, disjoint from $A_0$ and from each other. If $A_0=\emptyset$, we are done. Otherwise, we proceed with Claims 1, 2, and 3 from the proof of Proposition \ref{db-prop}, noting that this part does not require the use of saturating sets (hence $A$ need not be a tile). At that point, the only remaining case in the proof of Proposition \ref{db-prop} is when there exist $l',l''\in K(a)$ such that (\ref{db-e8}) holds. However, in that case we have (\ref{db-e10}), which contradicts (\ref{db-e18}). Therefore, under the assumptions of Proposition \ref{notopdiff},
$A_0$ (if nonempty) contains diagonal boxes $A_1$ as in (\ref{db-e19}). 
Moreover, the cardinality statement must hold, since otherwise we would not have (\ref{db-e18}).

As in Proposition \ref{db-prop}, $A_0$ need not be unique and may depend on the order in which the fibers are removed, and $A_1$ may then depend on the choice of $A_0$. We fix one such choice of $A_0$ and $A_1$, and keep it fixed for the remainder of the proof.

We claim that $A_0=A_1$. To prove this, assume for contradiction that $A_0\setminus A_1$ is nonempty. We clearly have $\Phi_M|A_0$ and $\Phi_M|A_1$, therefore $\Phi_M|A_0(X)-A_1(X)$. Since the set $A_0\setminus A_1$ is nonempty and contains no fibers, it must contain another set of diagonal boxes
$$
A_2=(I_2\times J_2\times K_2) \cup (I_2^c\times J_2^c\times K_2^c),
$$
with obvious notation. Furthermore, since $A_2\subset A_0\setminus A_1$, we must have 
\begin{equation}\label{db-e21}
A_1\cap A_2=\emptyset.
\end{equation}

We first claim that at least one of
$$
I_1=I_2,\ I_1=I_2^c,\ J_1=J_2,\ J_1=J_2^c,\ K_1=K_2, \ K_1=K_2^c,
$$
must hold.
Indeed, by (\ref{db-e21}), we must have
$$
(I_1\times J_1\times K_1)\cap (I_2\times J_2\times K_2)=\emptyset,
$$
so that at least one of $I_1\cap I_2$, $J_1\cap J_2$, $K_1\cap K_2$ is empty. Without loss of generality, we may assume that $K_1\cap K_2=\emptyset$, so that
\begin{equation}\label{db-e22}
K_1\subset K_2^c,\ \ K_2\subset K_1^c.
\end{equation}
We also have
$$
(I_1\times J_1\times K_1)\cap (I_2^c\times J_2^c\times K_2^c)=\emptyset,
$$
and since $K_1\cap K_2^c\neq\emptyset$, one of $I_1\cap I_2^c$ and $J_1\cap J_2^c$ is empty. Without loss of generality, we may assume that $J_1\cap J^c_2=\emptyset$, so that
\begin{equation}\label{db-e23}
J_1\subset J_2,\ \ J_2^c\subset J_1^c.
\end{equation}
Next,
$$
(I_1^c \times J_1^c \times K_1^c)\cap (I_2^c\times J_2^c\times K_2^c)=\emptyset.
$$
By (\ref{db-e23}), we have $J^c_1\cap J_2^c\neq\emptyset$, therefore one of $I_1^c\cap I_2^c$ and $K^c_1\cap K_2^c$ is empty. If $K^c_1\cap K_2^c=\emptyset$, then $K_2^c\subset K_1$, which together with (\ref{db-e22}) shows that $K_1=K_2^c$ and proves the claim. If $I_1^c\cap I_2^c=\emptyset$, we get that
\begin{equation}\label{db-e24}
I_1^c \subset I_2,\ \ I_2^c\subset I_1.
\end{equation}
Using that
$$
(I_1^c \times J_1^c \times K_1^c)\cap (I_2\times J_2\times K_2)=\emptyset.
$$
and taking (\ref{db-e22}) and (\ref{db-e24}) into account, we see that $J_1^c\cap J_2=\emptyset$. But then $J_2\subset J_1$, which together with (\ref{db-e23}) proves the claim. 

We may assume without loss of generality that 
$$
K_1= K_2^c.
$$
This implies that 
\begin{equation}\label{db-e26}
(I_1\times J_1)\cap (I_2\times J_2)=\emptyset \hbox{ and }(I_1^c \times J_1^c)\cap (I_2^c\times J_2^c)=\emptyset,
\end{equation}
since otherwise we would have an $M$-fiber in the $p_k$ direction in $A_0$. But also, considering the box layers with 
third coordinate $l\in K_1$ and $l'\in K_2$, we have
\begin{equation}\label{db-e27}
(I_1\times J_1)\cap (I_2^c\times J_2^c)=\emptyset \hbox{ and }(I_1^c \times J_1^c)\cap (I_2\times J_2)=\emptyset,
\end{equation}
The first parts of (\ref{db-e26}) and (\ref{db-e27}) imply that either
\begin{equation}\label{db-e28}
I_1\cap I_2=J_1\cap J_2^c=\emptyset,
\end{equation}
or else the same holds with $I,J$ interchanged. Assume that (\ref{db-e28}) holds. 
Then 
$I_2\subset I_1^c$ and $J_2^c\subset J_1^c$, so that in order for the second parts of (\ref{db-e26}) and (\ref{db-e27})
to hold, we must have
$$
I_1^c \cap I_2^c =J_1^c\cap J_2=\emptyset.
$$
But this implies that $I_1^c\subset I_2$ and $J_2^c\subset J_1^c$. Hence $I_1=I_2^c$ and $J_1=J_2$. But then
$A_0$ contains $M$-fibers in the $p_i$ direction, a contradiction.
\end{proof}


\subsection{Special unfibered structures: odd $M$}\label{special-odd}

\begin{lemma} \label{misscorner_fullplane}
Let $M=p_i^{n_i}p_j^{n_j}p_k^{n_k}$ and $N=M/p_i^{\alpha_i}p_j^{\alpha_j}p_k^{\alpha_k}$ with 
$\alpha_\iota<n_\iota$ for all $\iota\in\{i,j,k\}$. 
Assume that $ 2\nmid M $, and that $ A\in \mathcal{M}(\ZZ_N)$ satisfies $\Phi_N|A$. 
Let $\Lambda$ be a $D(N)$-grid such that $A\cap\Lambda\neq\emptyset$.
Assume further that 
\begin{itemize}
\item there exists a $c_0\in\NN$ such that
\begin{equation} \label{bin_cond}
\bbA^N_N[x]\in\{0,c_0\}\,\text{ for all } x\in\Lambda,
\end{equation}
\item $A\cap\Lambda$ is not $N$-fibered in any direction.
\end{itemize}
Then there is a permutation of $\{i,j,k\}$ such that
\begin{equation}\label{divrest}
\{D(N)|m|N\}\setminus \Div_N(A\cap\Lambda)\subseteq\{N/p_{i}p_j,N/p_{i}p_k\}. 
\end{equation}
Moreover, (\ref{divrest}) holds with equality if and only if there exists $x\in \ZZ_M\setminus A$ such that 
$$
 \bbA^{N}_{N/p_{i}}[x]=c_0\phi(p_{i}),\, \bbA^{N}_{N/p_jp_k}[x]=c_0\phi(p_jp_k),$$
$$
\bbA^{N}_m[x]=0\hbox{ for all }m\in \{D(N)|m|N\}\setminus \{N/p_{i},N/p_jp_k\}.
$$
We will refer to this structure as a {\bf $p_i$-full plane}.
\end{lemma}

\begin{figure}[h]
	\includegraphics[scale=0.7]{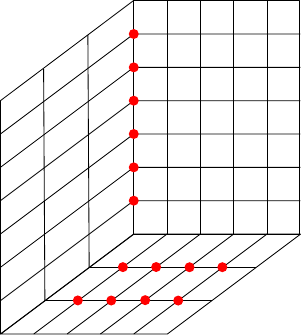}
	\caption{Full plane structure on a $D(M)$-grid.}
\end{figure}

\begin{proof}
We may assume that $M=N$ and $c_0=1$. We will also write $D=D(M)$ for short.
By Proposition \ref{notopdiff} and the assumption that $A\cap\Lambda$ is not fibered, $A\cap\Lambda$ must contain at least one of the following:

\medskip
(a) two nonintersecting $M$-fibers in different directions, say $p_i$ and $p_j$,

\smallskip
(b) diagonal boxes as in Definition \ref{def-db}, and possibly additional $M$-fibers in one or more directions, disjoint from the diagonal boxes and from each other.

\medskip

In the first case, we have $\{D|m|M\}\setminus \Div(A\cap\Lambda)\subseteq\{M/p_{i}p_j\}$, just based on these two fibers. It remains to consider the second case.
Suppose that $A\cap\Lambda$ contains diagonal boxes 
$$(I\times J\times K) \cup (I^c\times J^c\times K^c)$$
with $I,J,K,I^c,J^c,K^c$ all nonempty. In order for $\{D|m|M\}\setminus \Div(A\cap\Lambda)$ to be nonempty, we must have
\begin{equation}\label{db-e60}
\min(|I|,|J|,|K|)=\min(|I^c|,|J^c|,|K^c|)=1.
\end{equation}
We may assume without loss of generality that $|I^c|=|J|=1$.
Since $p_\iota\geq 3$ for all $\iota$, it follows that $|I|,|J^c|$, and at least one of $|K|,|K^c|$ are greater than 1. Assume that $|K^c| > 1$. Then 
\begin{equation}\label{db-e61}
\begin{split}
M/p_i &\in\Div(I\times J\times K),\\
M/p_j,M/p_k,M/p_jp_k  &\in\Div(I^c\times J^c\times K^c),\\
M/p_ip_jp_k &\in\Div(I\times J\times K, I^c\times J^c\times K^c).
\end{split}
\end{equation}
This implies (\ref{divrest}).
Furthermore, if (\ref{divrest}) holds with equality, then $|K|=1$, since otherwise we would also have 
$M/p_ip_k \in\Div(I\times J\times K)$. This proves the second conclusion of the lemma, with $x$ equal to the unique element of $I^c\times J\times K$. Note that if we add an $M$-fiber in any direction to this structure, then 
equality in (\ref{divrest}) can no longer hold.
\end{proof}

\begin{lemma}\label{oddcornerplus}
Let $M=p_i^{n_i}p_j^{n_j}p_k^{n_k}$ and $N=M/p_i^{\alpha_i}p_j^{\alpha_j}p_k^{\alpha_k}$ with 
$\alpha_\iota<n_\iota$ for all $\iota\in\{i,j,k\}$. 
Assume that $ 2\nmid M $, and that $ A\in \mathcal{M}(\ZZ_N)$ satisfies $\Phi_N|A$.
Let $\Lambda$ be a $D(N)$-grid such that $A\cap\Lambda\neq\emptyset$. Assume further that
\begin{itemize}
\item (\ref{bin_cond}) holds for all $ x\in\Lambda$, 
\item $A\cap\Lambda$ is not fibered in any direction,
\end{itemize}
and that
\begin{equation}\label{div_rest}
\{D|m|N\}\setminus \Div_N(A\cap\Lambda)=\{N/p_ip_j\}.
\end{equation}
Then $A\cap\Lambda$ has one of the following, mutually exclusive, possible structures:	

\begin{enumerate}
\item[(\emph{i})] ({\bf $p_k$-corner}, cf. Definition \ref{corner} (i).) For each $x\in\ZZ_N$, the set $A\cap\Lambda\cap\Pi(x,p_k^{n_k-\alpha_k})$ is either empty or 
consists of a single $N$-fiber in one of the $p_i$ or $p_j$ directions. Since $A\cap\Lambda$ is not fibered, there has to be at least one of each.

\begin{figure}[h]
	\includegraphics[scale=0.7]{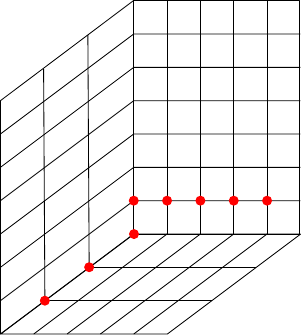}
	\caption{Corner structure on a $D(M)$-grid.}
\end{figure}

\item[(\emph{ii})] {\bf ($p_k$-almost corner)}
There exist $ x_0,x_1 ,\ldots,x_{\phi(p_k)}\in \ZZ_N $ with $(x_l-x_{l'},N)=N/p_k$ for $l\neq l'$, and 
two disjoint sets $ \mathcal{L}_i,\mathcal{L}_j \subset\ZZ_{p_k}$ satisfying $|\call_i|,|\call_j|>1$ and $ \call_i\cup \call_j=\{0,1,\ldots,\phi(p_k)\} $,
such that for all $z\in\Lambda$ we have
$$
\bbA^N_N[z]=\begin{cases}
c_0 &\hbox{ if } (z-x_l,N)=N/p_i \hbox{ for some } l\in \call_i\\
&\ \hbox{ or } (z-x_l,N)=N/p_j \hbox{ for some } l\in \call_j\\
0 &\hbox{ otherwise. }
\end{cases}
$$
In particular, $\bbA^{N}_N[x_l]=0$ and $\bbA^{N}_{N/p_i}[x_l]=c_0\phi(p_i)$ for all $l\in \call_i$, and similarly with $i$ and $j$ interchanged.

\begin{figure}[h]
	\includegraphics[scale=0.7]{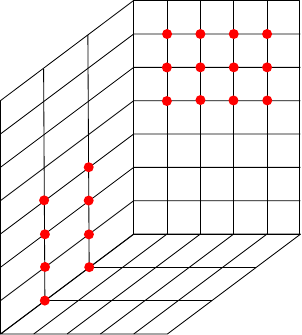}
	\caption{An almost corner structure on a $D(M)$-grid.}
\end{figure}

\end{enumerate}
\end{lemma}

\begin{proof}
We may assume that $M=N$ and $c_0=1$, and proceed as in the proof of Lemma \ref{misscorner_fullplane}. Suppose first that (a) holds (i.e., $A\cap\Lambda$ contains two non-overlapping $M$-fibers in two different directions), but $A$ does not contain diagonal boxes. Then the $M$-fibers must be in the $p_i$ and $p_j$ directions, or else (\ref{div_rest}) would be violated. Moreover, having any two such fibers in the same plane
$\Pi(x,p_k^{n_k-\alpha_k})$ would also violate (\ref{div_rest}), hence (i) holds in this case.

In case (b), $A\cap\Lambda$ contains diagonal boxes. We proceed as in the proof of Lemma \ref{misscorner_fullplane} to get (\ref{db-e60}). Since $M/p_ip_j\not\in\Div(A)$, we must have
$$
\min(|I|,|J|)=\min(|I^c|,|J^c|)=1. 
$$
Therefore we may again assume without loss of generality that $|I^c|=|J|=1$ and that $|K^c| > 1$, and get  
(\ref{db-e61}). We now consider two cases.
\begin{itemize}
\item If $|K|=1$, the diagonal boxes are as in the conclusion of Lemma \ref{misscorner_fullplane}, and instead of 
(\ref{div_rest}) we have (\ref{divrest}) with equality. 

\item If $|K|>1$, the diagonal boxes present the structure described in (ii), with $\call_i=K$ and $\call_j=K^c$.

\end{itemize}

The only way we can add an $M$-fiber to either of these structures without adding $M/p_ip_j$ to $\Div(A)$ is to add an $M$-fiber in the $p_k$ direction rooted at $x$ (in the first case, with $x$ defined as in Lemma \ref{misscorner_fullplane}), or at $x_0$ (in the second case, with $x_0$ specified in (ii))
That, however, puts us in the case (i) of the lemma. 
\end{proof}


\subsection{Special unfibered structures: even $M$}\label{special-even}

\begin{lemma}\label{evencorner}
Let $M=p_i^{n_i}p_j^{n_j}p_k^{n_k}$ with $2|M$, and $N=M/p_i^{\alpha_i}p_j^{\alpha_j}p_k^{\alpha_k}$ with 
$\alpha_\iota<n_\iota$ for all $\iota\in\{i,j,k\}$. 
Assume that $ A\in \mathcal{M}(\ZZ_N)$ satisfies $\Phi_N|A$. 
Let $\Lambda$ be a $D(N)$-grid such that $A\cap\Lambda\neq\emptyset$.
Assume further that
\begin{itemize}
\item (\ref{bin_cond}) holds for all $ x\in\Lambda$, 
\item $A\cap\Lambda$ is not fibered in any direction.
\end{itemize}
Assume that $N/p_\iota \in \Div_N(A\cap\Lambda)$ for all $\iota\in\{i,j,k\}$, but $\{D(N)|m|N\}\not\subset \Div_N(A\cap\Lambda)$. Then for some permutation of $\{i,j,k\}$ we have
$$\{D(N)|m|N\}\setminus \Div_N(A\cap\Lambda)=\{N/p_ip_j\},$$
and $A\cap\Lambda$ has the $p_k$ corner structure in the sense of Definition \ref{corner} (i):
for each $x\in\Lambda$, the set $A\cap\Lambda\cap\Pi(x,p_k^{n_k-\alpha_k})$ is either empty or 
consists of a single $N$-fiber in one of the $p_i$ or $p_j$ directions. Since $A\cap\Lambda$ is not fibered, there has to be at least one of each. 
\end{lemma}

\begin{proof}
We may assume that $M=N$ and $c_0=1$. 
Assume without loss of generality that $p_k=2$.
In this case, in order
for $N/p_k\in\Div(A\cap\Lambda)$, we must have at least one $N$-fiber $F$ in the $p_k$ direction in $A\cap\Lambda$.
Let $A_1$ be the set 
obtained from $A\cap\Lambda$ by removing all $N$-fibers in the $p_k$ direction. Since $A\cap\Lambda$ is not fibered, $A_1$ is nonempty and satisfies (\ref{db-e18}). By Proposition \ref{notopdiff},
it must contain either at least one fiber in another direction or a set of diagonal boxes, each disjoint from $F$. 

\begin{itemize}
\item Suppose that $A_1$ contains diagonal boxes as in Definition \ref{def-db}. Without loss of generality, we may assume that $F$ is rooted at a point $a\in I^c\times J \times K$. Then 
\begin{align*}
M/p_i,M/p_ip_k &\in \Div(F,I\times J\times K), \\
M/p_j,M/p_jp_k & \in \Div(F,I^c\times J^c\times K^c),\\
M/p_ip_jp_k& \in \Div(I\times J\times K,I^c\times J^c\times K^c).
\end{align*}
It follows that the only missing divisor can be $M/p_ip_j$. In order to avoid that divisor within each box, we must have 
$$
\min(|I|,|J|)=\min(|I^c|,|J^c|)=1. 
$$
Taking into account the differences $(a'-a'',M)$, where $a'\in F$ and $a''$ belongs to one of the boxes, we see that the only possible case is $|I^c|=|J|=1$. Then $A$ contains the $M$-fiber in the $p_i$ direction rooted at the
unique point $x\in I^c\times J\times K$, and the $M$-fiber in the $p_j$ direction rooted at the
unique point $x'\in I^c\times J\times K^c$. Any other points in $A\cap\Lambda$ would add $M/p_ip_j$ to $\Div(A\cap\Lambda)$. Thus the conclusion of the lemma holds.

\item If $A_1$ contains no diagonal boxes, then it must contain a fiber in at least one other direction. 
This case is identical to the corresponding case of Lemma \ref{oddcornerplus} (i), for some permutation of $\{i,j,k\}$.

\end{itemize}
\end{proof}

\begin{lemma}\label{even_struct}
Let $M=p_i^{n_i}p_j^{n_j}p_k^{n_k}$ with $2|M$, and $N=M/p_i^{\alpha_i}p_j^{\alpha_j}p_k^{\alpha_k}$ with 
$\alpha_\iota<n_\iota$ for all $\iota\in\{i,j,k\}$. 
Assume that $ A\in \mathcal{M}(\ZZ_N)$ satisfies $\Phi_N|A$. 
Let $\Lambda$ be a $D(N)$-grid such that $A\cap\Lambda\neq\emptyset$.
Assume further that $p_k=2$, and that
\begin{itemize}
\item (\ref{bin_cond}) holds for all $ x\in\Lambda$, 
\item $A\cap\Lambda$ is not fibered in any direction,
\item $\{N/p_i,N/p_j,N/p_k\}\not\subset \Div_N(A\cap\Lambda)$.
\end{itemize}
Then 
\begin{equation}\label{db-e70}
N/p_k\not\in \Div_N(A\cap\Lambda),
\end{equation} and
there is a pair of diagonal boxes 
$$A_0=
(I\times J\times K) \cup (I^c\times J^c\times K^c)\subset \Lambda,$$
as in Definition \ref{def-db}, such that for all $z\in A\cap\Lambda$ we have 
$$
\bbA^N_N[z]=\begin{cases}
c_0 &\hbox{ if } z\in A_0,\\
0 &\hbox{ otherwise. }
\end{cases}
$$
\end{lemma}

\begin{figure}[h]
	\captionsetup{justification=centering}
	\includegraphics[scale=0.7]{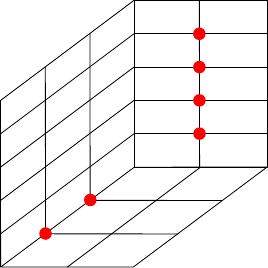}
	\caption{An even almost corner structure on a $D(M)$-grid\\
		with $M/p_k\notin \Div(A)$, $p_k=2$. }
\end{figure}

\begin{proof}
We may assume that $M=N$ and $c_0=1$. If $M/p_i$ or $M/p_j$ is not in $\Div(A\cap\Lambda)$, then 
$A\cap\Lambda$ is fibered by Lemma \ref{gen_top_div_mis}, contradicting the assumptions of the lemma. Therefore (\ref{db-e70}) holds.
Invoking Proposition \ref{notopdiff} again, we see that $A\cap\Lambda$ must contain either diagonal boxes or
at least two non-overlapping $M$-fibers in different directions. The second case cannot be reconciled with (\ref{db-e70}). Therefore $A\cap\Lambda$ must contain a set $A_0$ of diagonal boxes. Notice that adding an $M$-fiber in any direction to $A_0$ would introduce $M/p_k$ as a divisor of $A\cap\Lambda$. Therefore $A\cap\Lambda=A_0$.
\end{proof}


\section{Resolving diagonal boxes}\label{res-boxes}


\begin{theorem}\label{db-theorem}
Let $A\oplus B=\ZZ_M$, where $M=p_i^2p_j^2p_k^2$ with $p_i,p_j,p_k\geq 3$, $|A|=|B|=p_ip_jp_k$, and assume that $\Phi_M\mid A$. 
Let $D=D(M)$, and let $\Lambda$ be a $D(M)$-grid such that $A\cap\Lambda\neq\emptyset$.
Assume further that $A\cap\Lambda$ is not fibered in any direction, and 
that one of the following holds: either
\begin{equation}\label{db-e1000}
\{m:\ D(M)|m|M\}\subset \Div(A\cap\Lambda)
\end{equation}
and $A\cap\Lambda$ contains diagonal boxes as in Definition \ref{def-db},
or else $A\cap\Lambda$ has one of the structures described in Lemma \ref{misscorner_fullplane} (full plane) or Lemma \ref{oddcornerplus} (ii) (almost corner). Then at least one of the following is true:
\begin{itemize}
\item The tiling $A\oplus B=\ZZ_M$ is T2-equivalent to $\Lambda \oplus B=\ZZ_M$ via fiber shifts. Thus $\Lambda$ is a translate of $A^\flat$, and by Corollary \ref{get-standard}, both $A$ and $B$ satisfy (T2).

\item The tiling $A\oplus B=\ZZ_M$ is T2-equivalent to a tiling $A'\oplus B$, where $A'\cap\Lambda$ contains a $p_\nu$ corner structure as in Definition \ref{corner} (i) for some $\nu\in\{i,j,k\}$.
\end{itemize}
\end{theorem}

We remark that, in the case when (\ref{db-e1000}) holds, $A\cap\Lambda$ might be larger than just a pair of diagonal boxes. For example, it could contain diagonal boxes and some number of $M$-fibers in various directions disjoint from the boxes. However, any such additional structures can only make our task easier.

We split the proof into cases. Since $p_i,p_j,p_k\geq 3$, at least one of $I$ and $I^c$ must have cardinality greater than 1, and similarly for each of the pairs $J,J^c$ and $K,K^c$. We claim that it suffices to consider the following two cases.

\medskip\noindent
{\bf Case (DB1):} The tiling $A\oplus B$ satisfies the assumptions of Theorem \ref{db-theorem}, and additionally
$\min( |I|, |J^c|, |K^c|)\geq 2$.

\medskip\noindent
{\bf Case (DB2):} The tiling $A\oplus B$ satisfies the assumptions of Theorem \ref{db-theorem}, and additionally
$|I^c|=|J^c|=|K^c|=1$.

\medskip

Indeed, if either $|I|=|J|=|K|=1$ or $|I^c|=|J^c|=|K^c|=1$, then we are in the case (DB2), possibly after relabelling $I,J,K$ as $I^c,J^c, K^c$ and vice versa. Suppose now that neither of these holds, say $|I|\geq 2$ and $|J^c|\geq 2$. Since $p_k\geq 3$, at least one of $K$ and $K^c$ must have cardinality at least 2. If $|K^c|\geq 2$, we are in the case (DB1). 
If $|K|\geq 2$, we are in the case (DB1) again, 
with $p_i$ and $p_j$ interchanged, and with the sets $I,J,K$ relabelled as $I^c,J^c, K^c$ and vice versa.
All other cases are identical up to a permutation of the indices $i,j,k$.

We will proceed to resolve the cases (DB1) and (DB2) in Sections \ref{subsec-db1} and \ref{subsec-db2}, respectively.
Throughout this section, we continue to use the notation of Section \ref{diag-box-section}.


\subsection{Preliminary results}\label{db-subsecprel}

\begin{lemma}\label{db-nontrivial}
Let $A\oplus B=\ZZ_M$, where $M=p_i^2p_j^2p_k^2$, $|A|=|B|=p_ip_jp_k$, and assume that $\Phi_M\mid A$. 
Let $D=D(M)$, and let $\Lambda$ be a $D$-grid such that $A\cap\Lambda\neq\emptyset$.
Assume that (\ref{db-e1000}) holds.
Suppose that $A\cap\Lambda$ is not fibered, and that it contains diagonal boxes
\begin{equation}\label{db-eeA0}
A_0:=(I\times J\times K) \cup (I^c\times J^c\times K^c)\subset A\cap\Lambda,
\end{equation}
with $I,J,K,I^c,J^c,K^c$ as in Definition \ref{def-db}. If one of the ``complementary boxes" is contained in $A$, say
\begin{equation}\label{db-ee1}
(I^c\times J\times K)\subset A,
\end{equation}
then the tiling $A\oplus B=\ZZ_M$ is T2-equivalent to $\Lambda\oplus B=\ZZ_M$. Consequently, $A$ and $B$ both satisfy T2.
\end{lemma}

\begin{proof} Throughout the proof, we assume that $A\oplus B=\ZZ_M$ is a tiling satisfying the assumptions of the lemma.

\medskip\noindent
{\bf Claim 1.} {\it Assume that (\ref{db-ee1}) holds. Then either $A\oplus B=\ZZ_M$ is T2-equivalent to $\Lambda\oplus B=\ZZ_M$, or }
\begin{equation}\label{db-ee2}
(I^c\times J\times K^c) \cap A=\emptyset.
\end{equation}

\begin{proof}
Suppose that there is an $a\in( I^c\times J\times K^c )\cap A$, and let 
$$Z=\Pi(a,p_i^2)\cap ( I^c\times J^c\times K).$$
 For any $z\in Z$, let $a_j,a_k$ be points such that $(z-a_j,M)=(a_k-a,M)=M/p_j$ and $(z-a_k,M)=(a_j-a,M)=M/p_k$.
Then $a_j,a_k\in A$, since $a_j\in I^c\times J\times K$ and $a_k\in I^c\times J^c\times K^c$. By Lemma \ref{flatcorner},
either $z\in A$, or else $A_z\subset\ell_i(z)$, with a $(1,2)$ cofibered structure for $(A,B)$ in the $p_i$ direction and an $M$-fiber in $A$ at distance $M/p_i^2$ from $z$ as a cofiber. We can use Lemma \ref{fibershift} to shift that fiber to $z$. Repeating this procedure for all $z\in Z\setminus A$, we get a new set $A_1\subset\ZZ_M$ such that $A_1\oplus B=\ZZ_M$ and $A_1$ is T2-equivalent to $A$. Moreover,
$Z\subset A_1$, and, since $A\cap\Lambda\subset A_1\cap\Lambda$, (\ref{db-e1000}) holds with $A$ replaced by $A_1$.

Now, any $x\in (I^c\times J^c\times K)\setminus Z$ satisfies the assumptions of Lemma \ref{smallcube} applied to $A_1$, with 
$(x-z,M)=M/p_i$ for some $z\in Z$, $(x-a'_j,M)=M/p_j$ for some $a'_j\in I^c\times J\times K$, and 
$(x-a'_k,M)=M/p_k$ for some $a'_k\in I^c\times J^c\times K^c$. It follows that
\begin{equation}\label{db-ee4}
I^c\times J^c\times K \subset A_1.
\end{equation}

We can then apply Lemma \ref{flatcorner} and the fiber shifting argument again, first to all points in $I^c\times J\times K^c$ with the flat corner configurations in planes perpendicular to the $p_i$ direction, then to all points in $I\times J^c\times K$ with the flat corner configurations in planes perpendicular to the $p_k$ direction. Thus $A_1$ is T2-equivalent to a set $A_2$ that 
satisfies $A_2\oplus B=\ZZ_M$, continues to obey (\ref{db-ee1})--(\ref{db-ee4}), and moreover has the property that 
$$
(I^c\times J\times K^c)\cup (I\times J^c\times K)\subset A_2.
$$
Finally, we apply Lemma \ref{flatcorner} and the fiber shifting argument to all points in $(I\times J\times K^c)\cup (I\times J^c\times K^c)$ that are not in $A_2$, with the flat corner configuration perpendicular to the $p_j$ direction. This proves that $A_2$ (therefore $A$) is T2-equivalent to $\Lambda$.
\end{proof}

\medskip\noindent
{\bf Claim 1'.} {\it Assume that (\ref{db-ee1}) holds. Then either $A\oplus B=\ZZ_M$ is T2-equivalent to $\Lambda\oplus B=\ZZ_M$, or }
\begin{equation}\label{db-ee2b}
(I^c\times J^c\times K) \cap A=\emptyset.
\end{equation}

\begin{proof}
This is identical to the proof of Claim 1, with the $p_j$ and $p_k$ directions interchanged, and with $J$ and $K^c$ replaced by $J^c$ and $K$.
\end{proof}

\medskip\noindent
{\bf Claim 2.} {\it Assume that (\ref{db-ee1}), (\ref{db-ee2}), and (\ref{db-ee2b}) hold. Then}
\begin{equation}\label{db-ee3}
I\times J^c\times K^c \subset  A.
\end{equation}

\begin{proof}
Let $x\in I\times J^c\times K^c$. Considering any $M$-cuboid with one vertex at $x$ and another in $I^c\times J\times K$, we see that it can only be balanced if $x\in A$.
\end{proof}

We can now finish the proof of the lemma. It suffices to consider the case when (\ref{db-ee1}) and (\ref{db-ee3}) hold,
so that
$$
Y:=( \ZZ_{p_i}\times J\times K)\cup ( \ZZ_{p_i}\times J^c\times K^c) \subset  A.
$$
If there are any elements $a\in (A\cap\Lambda)\setminus Y$, we can repeat the argument in the proof of Claim 1 to prove that the tiling $A\oplus B=\ZZ_M$ is T2-equivalent to $\Lambda\oplus B=\ZZ_M$. If on the other hand $A\cap\Lambda=Y$, then this set is fibered in the $p_i$ direction, contradicting the assumptions of the lemma.
\end{proof}

\begin{corollary}\label{db+fiber}
Let $A\oplus B=\ZZ_M$, $M=p_i^2p_j^2p_k^2$, $|A|=|B|=p_ip_jp_k$, and assume that $\Phi_M\mid A$. 
Let $D=D(M)$, and let $\Lambda$ be a $D$-grid such that $A\cap\Lambda\neq\emptyset$.
Assume that (\ref{db-e10}) holds.
Suppose that $A\cap\Lambda$ is not fibered, that it contains diagonal boxes (\ref{db-eeA0})
with $I,J,K,I^c,J^c,K^c$ as in Definition \ref{def-db}, and that it also contains at least one $M$-fiber disjoint from these boxes. Then the tiling $A\oplus B=\ZZ_M$ is T2-equivalent to $\Lambda\oplus B=\ZZ_M$. Consequently, $A$ and $B$ both satisfy T2.
\end{corollary}

\begin{proof}
We may assume without loss of generality that the $M$-fiber $F\subset A\cap\Lambda$ is in the $p_j$ direction, with $F\subset I^c\times \ZZ_{p_j}\times K$. By translational invariance, we may further assume that 
$$F=\{0\}\times \ZZ_{p_j}\times \{0\},
$$
with $0\in I^c$ and $0\in K$. 

We first claim that there is a set $A_1\subset\ZZ_M$, either equal to $A$ or T2-equivalent to it, such that
$A_1\oplus B=\ZZ_M$ and 
\begin{equation}\label{db+1}
\{0\}\times J\times K\subset A_1.
\end{equation}
Indeed, if $K=\{0\}$, then $\{0\}\times J\times K\subset F\subset A$ and there is nothing to prove. Otherwise, 
let $l_j\in J$ and $l_k\in K\setminus\{0\}$. Then the point $x=(0,l_j,l_k)$ either belongs to $A$, or else it satisfies the assumptions of Lemma \ref{flatcorner}, with $(0,l_j,0)$, $(l_i,l_j,0)$, and $(l_i,l_j,l_k)$ all in $A$ for any $l_i\in I$, so that the flat corner configuration is perpendicular to the $p_j$ direction. 
By Lemma \ref{flatcorner}, in the latter case we have $A_x\subset \ell_j(x)$, implying a $(1,2)$ cofibered structure. We can then use Lemma \ref{fibershift} to shift the $M$-cofiber in $A$ to $x$. After all such shifts have been performed, we arrive at $A_1$.

By a similar argument, but with $p_i$ and $p_k$ interchanged and with the possible flat corner configurations perpendicular to the $p_k$ direction, we may further replace $A_1$ by a T2-equivalent set $A_2$ such that $A_2\oplus B=\ZZ_M$, (\ref{db+1}) continues to hold for $A_2$, and 
\begin{equation}\label{db+2}
I^c\times J^c\times \{0\}\subset A_2.
\end{equation}

Next, we replace $A_2$ by a T2-equivalent set $A_3$ such that $A_3\oplus B=\ZZ_M$, (\ref{db+1}) and (\ref{db+2}) both continue to hold for $A_3$, and 
\begin{equation}\label{db+3}
I^c\times J\times \{0\}\subset A_3.
\end{equation}
Indeed, consider a point $z=(l'_i,l_j,0)\in I^c\times J\times \{0\}$, with $l'_i\neq 0$. (If no such $l'_i$ exists, (\ref{db+3}) holds with $A_3=A_2$.) If $z\not\in A_2$, then it satisfies the assumptions of Lemma \ref{flatcorner}, with $(0,l_j,0)$, $(0,l'_j,0)$, and $(l'_i,l'_j,0)$ all in $A_2$ for any $l'_j\in J^c$, so that the flat corner configuration is perpendicular to the $p_k$ direction. Then $A_z\subset\ell_k(z)$, and again our conclusion follows by Lemma \ref{fibershift}.

Finally, we may pass to another T2-equivalent set $A_4$ such that $A_4\oplus B=\ZZ_M$ and
\begin{equation}\label{db+4}
I^c\times J\times K\subset A_4.
\end{equation}
If $I^c=\{0\}$ or $K=\{0\}$, this follows from (\ref{db+1}) or (\ref{db+3}), with $A_4=A_3$. Otherwise, 
we let  $w=(l'_i,l_j,l_k)\in I^c\times J\times K$ with $l'_i\neq 0$ and $l_k\neq 0$, and repeat the
argument from the proof of (\ref{db+1}) with the first coordinate $0$ replaced by $l'_i$. 

With (\ref{db+4}) in place, the corollary now follows from Lemma \ref{db-nontrivial}.
\end{proof}


\subsection{Case (DB1)}\label{subsec-db1}

Assume that $A\oplus B=\ZZ_M$ is a tiling satisfying the assumptions of Theorem \ref{db-theorem}. Let $D=D(M)$, and let $\Lambda$ be the $D$-grid provided by the assumption of the theorem. Additionally, we assume that
\begin{equation}\label{db1-e0}
 \min( |I|, |J^c|, |K^c|)\geq 2.
\end{equation}

If $A\cap\Lambda$ has one of the structures described in Lemma \ref{misscorner_fullplane} (full plane) or Lemma \ref{oddcornerplus} (ii) (almost corner), then in both cases we have $A\cap\Lambda=(I\times J\times K)\cup (I^c\times J^c\times K^c)$ with no other points permitted, so that
$$
(I^c\times J\times K)\cap A=\emptyset.
$$
(Note that (\ref{db1-e0}) covers the cases of a $p_i$ full plane and a $p_j$ almost corner structure. See the end of this section for more details.)

If on the other hand (\ref{db-e1000}) holds and 
$I^c\times J\times K\subset A$, the conclusion of Theorem \ref{db-theorem} follows by Lemma \ref{db-nontrivial}. 
We may therefore assume that there exists a point
\begin{equation}\label{db1-e1}
x\in (I^c\times J\times K)\setminus A.
\end{equation} 

\begin{lemma}\label{db1-satline} 
Assume (DB1), and let $x$ satisfy (\ref{db1-e1}). 
Then for every $b\in B$ we have exactly one of the following:
\begin{equation}\label{db1-jsatura}
A_{x,b}\subset \ell_j(x), \hbox{ with }\bbA_{M/p_j^2}[x]\bbB_{M/p_j^2}[b]=\phi(p_j^2),
\end{equation} 
\begin{equation}\label{db1-ksatura}
A_{x,b}\subset \ell_k(x), \hbox{ with }\bbA_{M/p_k^2}[x]\bbB_{M/p_k^2}[b]=\phi(p_k^2).
\end{equation}
Furthermore:
\begin{itemize}
\item $\bbA_{M/p_j}[x]\cdot\bbA_{M/p_k}[x]=0$,
\item if $\bbA_{M/p_j}[x]>0$, then (\ref{db1-jsatura}) cannot hold for any $b\in B$, and if $\bbA_{M/p_k}[x]>0$, then (\ref{db1-ksatura}) cannot hold for any $b\in B$,
\item if (\ref{db1-jsatura}) holds for some $b\in B$, then the product $\langle \bbA[x],\bbB[b]\rangle$ is saturated by a $(1,2)$-cofiber pair in the $p_j$
direction, with the $A$-cofiber at distance $M/p_j^2$ from $x$ and the $B$-fiber rooted at $b$. The same is true for 
(\ref{db1-ksatura}), with $j$ and $k$ interchanged.
\end{itemize}
\end{lemma}

\begin{proof}
The assumptions (DB1) and (\ref{db1-e1}) imply in particular that $\bbA_{M/p_i}[x]\geq 2$. Corollary \ref{bispan-corollary} (ii) then gives
\begin{equation}\label{db1-e2}
A_x\subset \Pi(x,p_i^2).
\end{equation}
Moreover, by Lemma \ref{bispan-lemma} applied to $A_x$ and a single element $a\in A$ satisfying $(x-a,M)=M/p_jp_k$, we get $A_x\subset \Bispan(x,a)= \Pi(x,p_j^2)\cup \Pi(x,p_k^2)\cup \Pi(a,p_j^2)\cup \Pi(a,p_k^2)$. Taking the intersection for all such $a$, and using that $\min(|J^c|,|K^c|)\geq 2$, we get
$$
A_x\subset \bigcap_{a: (x-a,M)=M/p_jp_k} \Bispan(x,a) = \Pi(x,p_j^2)\cup \Pi(x,p_k^2),
$$
which together with (\ref{db1-e2}) proves that 
$$
A_x\subset \ell_j(x)\cup \ell_k(x).
$$
By Lemma \ref{triangles}, for each $b\in B$ we must in fact have either $A_{x,b}\subset\ell_j(x)$ or $A_{x,b}\subset\ell_k(x)$. If $A_{x,b}\subset\ell_j(x)$, then the second part of (\ref{db1-jsatura}) follows since $M/p_j\in\Div(A)$, and Lemma \ref{1dim_sat-cor} implies the existence of a cofiber pair as described in the last part of the lemma. The same applies with $j$ and $k$ interchanged.

Next, suppose that $\bbA_{M/p_j}[x]>0$ and let $a\in A$ with $(x-a,M)=M/p_j$. 
By Corollary \ref{bispan-corollary} (i), it follows that $A_x\subset\Pi(x,p_j^2)\cup \Pi(a,p_j^2)$. If we also assume that  
$A_x\subset \ell_j(x)$, this implies that $A_x\subset \{a,x\}$. However,$x\not\in A$ by the assumption (\ref{db1-e1}), and (DB1)  (in particular, (\ref{db1-e0})) implies that $M/p_j\in Div(A)$, so that it cannot contribute to the product $\langle\bbA[x], \bbB[b]\rangle$ for any $b\in B$. Hence the assumption that $A_x\subset \ell_j(x)$ is not compatible with (\ref{db1-jsatura}) for any $b\in B$. The same applies with $j$ and $k$ interchanged.
On the other hand, one of (\ref{db1-jsatura}) and (\ref{db1-ksatura}) must hold for each $b$, therefore we cannot have both $\bbA_{M/p_j}[x]>0$ and $\bbA_{M/p_j}[x]>0$. 

Finally, the last part of the lemma follows from Lemma \ref{1dim_sat-cor}.
\end{proof}

\begin{lemma}\label{db1-lemma2}
Assume (DB1). Suppose that there is a point $x\in\Lambda$ such that (\ref{db1-e1}) holds and
\begin{equation}\label{db1-e3}
\max(\bbA_{M/p_j}[x], \bbA_{M/p_k}[x])>0.
\end{equation}
Then the conclusion of Theorem \ref{db-theorem} holds.
\end{lemma}

\begin{proof}
Assume without loss of generality that $\bbA_{M/p_j}[x]> 0$. Then $A_x\subset \ell_k(x)$ by Lemma \ref{db1-satline}. By Lemma \ref{1dim_sat-cor}, the pair $(A,B)$ has a (1,2)-cofibered structure in the $p_k$ direction with a cofiber in $A$ at distance $M/p_k^2$ from $x$. 
We apply Lemma \ref{fibershift} to shift the $M$-cofiber in $A$ to $x$. Let $A_1$ be the set thus obtained. 

We note that (\ref{db1-e3}) does not hold for either the full plane structure or the almost corner structure. Therefore 
(\ref{db-e1000}) must hold, and since $A\cap\Lambda\subset A_1\cap\Lambda$, the same holds for $A_1$. 
Furthermore, $A_1$ contains the diagonal boxes inherited from $A$ as well as the added $M$-fiber through $x$ in the $p_j$ direction, disjoint from the boxes. By Corollary \ref{db+fiber}, the tiling $A_1\oplus B=\ZZ_M$ (therefore also $A\oplus B=\ZZ_M$)
 is T2-equivalent to $\Lambda\oplus B=\ZZ_M$. 
\end{proof}

It remains to consider the complementary case when 
\begin{equation}\label{db1-e4}
\bbA_{M/p_j}[x]= \bbA_{M/p_k}[x]=0 \ \ \forall x\in (I^c\times J\times K)\setminus A.
\end{equation}

\begin{lemma}\label{db1-lemma3}
Assume that (DB1) holds, that $(I^c\times J\times K)\setminus A\neq\emptyset$, and that
(\ref{db1-e4}) holds. Fix an index $l_i\in I^c$ such that $(\{l_i\}\times J\times K)\setminus A$ is nonempty, and define
$$
\calx=\calx(l_i):= \{l_i\}\times J\times K.
$$ 
Then either $A$ is T2-equivalent to a set $A'\subset\ZZ_M$ such that 
$A'\oplus B=\ZZ_M$ and every point $x'\in \calx$ belongs to a fiber in $A'$ in the $p_j$ direction, or else 
the same holds with $j$ and $k$
interchanged. 
\end{lemma}

\begin{proof} 
We fix $l_i$ as in the statement of the lemma, and keep it fixed for the duration of the proof.
Let also $\Pi:=\Pi(x,p_i^2)$ for any $x\in\calx$, and note that $\calx \subset \Pi$.

Let $x\in \calx\setminus A$. 
By (\ref{db1-e4}), we must have $x'\not\in A$ for all $x'\in \calx$ with $(x'-x,M)\in\{M/p_j,M/p_k\}$. 
Applying (\ref{db1-e4}) to all such $x'$, we see that
$\calx\cap A=\emptyset$.
Let
$$\calx_j=\{ x\in \calx:\ \exists b\in B \hbox{ such that }A_{x,b}\subset \ell_j(x)\},$$
and similarly for $\calx_k$.
By Lemma \ref{db1-satline}, we have $\calx=\calx_j\cup\calx_k$.

Suppose that $\calx_k\neq\emptyset$, with $x_0\in \calx_k$, and let $x_1,x_2,\dots,x_{|J|-1}$, be the distinct points in $\calx$ such that $(x_0-x_\nu,M)=M/p_j$ for $\nu\neq 0$. By the definition of $\calx_k$, there exists $b\in B$ such that 
\begin{equation}\label{db1-e5}
A_{x_\nu,b} \subset\ell_k(x_\nu)
\end{equation}
for $\nu=0$. We claim that (\ref{db1-e5}) also holds for $\nu=1,\dots,|J|-1$, with the same $b\in B$. Indeed, suppose that  $A_{x_\nu,b}\subset\ell_j(x_\nu)$ for some $\nu\geq 1$. By Lemma \ref{1dim_sat-cor}, the product $\langle \bbA[x_\nu],\bbB[b]\rangle$ would be saturated by a $(1,2)$ cofiber pair in the $p_j$ direction at distance $M/p_j^2$ from $x_\nu$. But then we would also have $A_{x_0,b}\subset\ell_j(x_0)$, with the product saturated by the same cofiber pair, contradicting the assumption that (\ref{db1-e5}) holds for $\nu=0$. 

By Lemma \ref{1dim_sat-cor} again, it follows that if $\calx_k\neq\emptyset$, then $A$ contains an $M$-fiber in the $p_k$ direction at distance $M/p_k^2$ from each of the points $x_0,x_1,x_2,\dots,x_{|J|-1}$. Since $\ell_k(x_\nu)\subset \Pi$, all these fibers are contained in $\Pi$.

If we had both $\calx_j\neq\emptyset$ and $\calx_k\neq\emptyset$, we would get such sets of fibers in both directions, all contained in $\Pi$. But then 
$\Pi$ contains $|J^c|\,|K^c|$ points of $I^c\times J^c\times K^c$, at least $|J|$ fibers in $A$ in the $p_k$ direction, and at least $|K|$ fibers in $A$ in the $p_j$ direction, 
disjoint from $\Lambda\cap \Pi$ and from each other. Thus
$$
|A\cap\Pi|\geq (p_j-|J|)(p_k-|K|)+p_k|J|+p_j|K|= p_jp_k+|J||K|.
$$
This contradicts Lemma \ref{planebound}.

It follows that at least one of $\calx_j$ and $\calx_k$ must be empty. Assume without loss of generality that $\calx_j=\calx$ and $\calx_k=\emptyset$. It follows that $A_x\subset \ell_j(x)$ for all $x\in\calx$. 
By Lemma \ref{1dim_sat-cor}, the pair $(A,B)$ has a $(1,2)$-cofibered structure in the $p_j$ direction, and for every $x\in\calx$ there is a cofiber in $A$ at distance $M/p_j^2$ from $x$. We can now use Lemma \ref{fibershift} to shift all of the aforementioned cofibers in the $p_j$ direction, obtaining the new set $A'$ as indicated in the conclusion of the lemma.
 \end{proof}

 We can now conclude the proof of Theorem \ref{db-theorem} under the assumption that (\ref{db1-e0}), (\ref{db1-e1}), and (\ref{db1-e4})  
 hold.

\begin{itemize}
\item If (\ref{db-e1000}) holds, then the set $A'\subset\ZZ_M$ from Lemma \ref{db1-lemma3} still satisfies (\ref{db-e1000}), and $A'\cap\Lambda$ contains the diagonal boxes inherited from $A$ as well as the additional fibers added in Lemma \ref{db1-lemma3}. By Corollary \ref{db+fiber}, $A'$ is T2-equivalent to $\Lambda$, and the theorem follows.

\item 
Suppose that $A\cap\Lambda$ has the $p_i$ full plane structure (Lemma \ref{misscorner_fullplane}), with $|I^c|=|J|=|K|=1$.
Then there is only one index $l_i\in I^c$, and for that $l_i$, $\calx=\{x\}$ is a single point. In this case, Lemma \ref{db1-lemma3} identifies a direction $p_\nu$ for some $\nu\in\{j,k\}$, such that the pair $(A,B)$ has a $(1,2)$-cofibered structure in the $p_\nu$ direction with a cofiber in $A$ at distance $M/p_\nu^2$ from $x$. Assume without loss of generality that $\nu=j$. After shifting the cofiber to $x$ as permitted by Lemma \ref{fibershift}, we arrive at a $p_j$ corner structure, where $A'\cap\Lambda$ is the union of one $M$-fiber in the $p_i$ direction and $(p_j-1)$ $M$-fibers in the $p_k$ directions, each in a different plane perpendicular to the $p_j$ direction.

\item Assume now that $A\cap\Lambda$ has the $p_j$ almost corner structure (Lemma \ref{oddcornerplus} (ii)), with 
$|K|=|I^c|=1$ and $|J|,|J^c|>1$. Then there is again only one index $l_i\in I^c$, and the set $\calx= \{l_i\}\times J\times K
$ for that value of $l_i$ has dimensions $1\times |J|\times 1$. 
We again apply Lemma \ref{db1-lemma3}. If the fiber identified and shifted in the lemma was in the $p_j$ direction, then we are in the $p_j$ corner situation again, with $A'\cap\Lambda$ consisting of non-overlapping fibers in the $p_i$ and $p_k$ directions. If on the other hand the shifted fibers were in the $p_k$ direction, then we have
$\{D|m|M\}\subset\Div(A')$, and
$A'\cap\Lambda $ contains the diagonal boxes inherited from $A$ as well as the added fibers disjoint from the boxes. By Corollary \ref{db+fiber}, the tiling $A'\oplus B=\ZZ_M$ (therefore also $A\oplus B=\ZZ_M$)
 is T2-equivalent to $\Lambda\oplus B=\ZZ_M$. 

\end{itemize}

 
 \subsection{Case (DB2)}\label{subsec-db2}

We now assume that $A\oplus B=\ZZ_M$ is a tiling satisfying the assumptions of Theorem \ref{db-theorem} on a $D$-grid $\Lambda$. Additionally, we assume that
\begin{equation}\label{db2-e0}
|I^c|=|J^c|=|K^c|=1.
\end{equation}
Since $p_i,p_j,p_k$ are all odd, the assumption (\ref{db2-e0}) implies that $\min(|I|,|J|,|K|)\geq 2$, and in particular
(\ref{db-e1000}) holds. Let also
$$
I^c\times J^c\times K^c=\{a_0\}.
$$

\begin{lemma}\label{db2-lemma0}
Assume (DB2). Suppose that there exists a point $x\in\Lambda\setminus A$ such that $A_x$ is contained in one of the lines $\ell_i(x)$, $\ell_j(x)$, $\ell_k(x)$. Then the tiling
$A\oplus B=\ZZ_M$ is T2-equivalent to $\Lambda\oplus B=\ZZ_M$. Hence the conclusion of Theorem \ref{db-theorem} holds.

\end{lemma}

\begin{proof}
Suppose that $x\in \Lambda\setminus A$ and $A_x\subset \ell_k(x)$. By Lemma \ref{1dim_sat-cor}, the pair $(A,B)$ has a (1,2)-cofibered structure in the $p_k$ direction. By Lemma \ref{fibershift}, the cofiber in $A\setminus\Lambda$ can be shifted to a fiber in $\Lambda$ rooted at $x$. This yields a new tiling $A'\oplus B=\ZZ_M$, T2-equivalent to $A\oplus B=\ZZ_M$, such that $A'\cap\Lambda$ contains both the diagonal boxes $\{a_0\}\cup (I\times J\times K)$ and the added $M$-fiber. In light of (\ref{db-e1000}), we can apply Corollary \ref{db+fiber} to conclude the proof.
\end{proof}

\begin{lemma}\label{db2-lemma1}
Assume (DB2). If $(A\cap\Lambda)\setminus(\{a_0\}\cup (I\times J\times K))\neq\emptyset$, then the tiling
$A\oplus B=\ZZ_M$ is T2-equivalent to $\Lambda\oplus B=\ZZ_M$. \end{lemma}

\begin{proof}
Observe first that if $A\cap\Lambda$ contains a flat corner configuration in the sense of Lemma \ref{flatcorner}, then the conclusion holds. Indeed, Lemma \ref{flatcorner} implies that the point $x\not\in A$ in the flat corner configuration satisfies the assumptions of Lemma \ref{db2-lemma0}, which we then apply.

Suppose that there is an element $a\in A\cap(I^c\times J\times K)$. In order to avoid a flat corner configuration perpendicular to the $p_j$ direction at points $x\in I^c\times\{\lambda_ja\}\times K$, all such points must belong to $A$. Similarly, we must have $I^c\times J\times\{\lambda_ka\}\subset A$, in order to avoid a flat corner perpendicular to the $p_k$ direction at those points. By Lemma \ref{smallcube} and (\ref{db-e1000}), we must in fact have
$I^c\times J\times K\subset A$. But then the conclusion follows by Lemma \ref{db-nontrivial}. Since the assumption (DB2) is symmetric with respect to all permutations of the indices $\{i,j,k\}$, the same argument applies when $A\cap(I\times J^c\times K)$ or $A\cap(I\times J\times K^c)$ is nonempty.

Assume now that $a\in  A\cap(I^c\times J^c\times K)$. Consider an $M$-cuboid with one vertex at $a_0$, another at $a$, and a third one at $a'\in I\times J\times K$ with $(a-a',M)=M/p_ip_j$ and $(a_0-a,M)=D$. In order to balance this cuboid, at least one of the vertices in $I^c\times J\times K$, $I\times J^c\times K$, or $I\times J\times K^c$, must be in $A$. But then we are in the situation from the last paragraph. Since this case is also symmetric with respect to all permutations of $\{i,j,k\}$, we are done.
\end{proof}

\begin{lemma}\label{db2-lemma2}
Assume (DB2), and let $N_k=M/p_k$. If $M/p_k^2\in\Div(A)$, then $\Phi_{N_k}|A$.
\end{lemma} 

\begin{proof}
Assume for contradiction that $\Phi_{N_k}|B$, and apply Lemma \ref{gen_top_div_mis} to $B$, on scale $N=N_k$. Since $N_k,N_k/p_k\not\in\Div(B)$, it follows that $B$ is $N_k$-fibered in one of the $p_i$ and $p_j$ directions. However, that is impossible, given that $\Div(B)\cap\{D|m|M\}=\{M\}$ due to (\ref{db-e1000}).
\end{proof}

We now begin the proof of Theorem \ref{db-theorem} under the assumption (DB2).
By Lemma \ref{db2-lemma1}, it suffices to consider the case when
 \begin{equation}\label{db2-e00}
(A\cap\Lambda)\setminus(\{a_0\}\cup (I\times J\times K))=\emptyset.
\end{equation}  
Let $x\in I\times J^c\times K$, so that in particular $x\not\in A$. 
Fix $b\in B$. Since $\bbA_{M/p_j}[x]\geq 2$, by Corollary \ref{bispan-corollary} (ii) we have $A_{x,b}\subset\Pi(x,p_j^2)$. Taking also into account that, by (\ref{bispan}), we must have $A_{x,b}\subset \Bispan(x,a_0)$, we get
$$
A_{x,b}\subset \ell_i(x)\cup\ell_i(a_0)\cup\ell_k(x)\cup\ell_k(a_0).
$$
We claim that $A_{x,b}$ cannot intersect more than one of the above lines.
Indeed, by (\ref{db-e1000}) we have
\begin{itemize}
\item $A_{x,b}\cap\ell_i(x)\neq\emptyset$ implies $\bbA_{M/p_i^2}[x]\bbB_{M/p_i^2}[b]>0$, hence $M/p_i^2p_k\in \Div(A)$ and $M/p_i^2\in \Div(B)$.
\item $A_{x,b}\cap\ell_k(x)\neq\emptyset$ implies $\bbA_{M/p_k^2}[x]\bbB_{M/p_k^2}[b]>0$, hence $M/p_ip_k^2\in \Div(A)$ and $M/p_k^2\in \Div(B)$.
\item $A_{x,b}\cap\ell_i(a_0)\neq\emptyset$ implies $\bbA_{M/p_i^2p_k}[x|\ell_i(a_0)]\bbB_{M/p_i^2p_k}[b]>0$, hence $M/p_i^2\in \Div(A)$ and $M/p_i^2p_k\in \Div(B)$.
\item $A_{x,b}\cap\ell_k(a_0)\neq\emptyset$ implies $\bbA_{M/p_ip_k^2}[x|\ell_k(a_0)]\bbB_{M/p_ip_k^2}[b]>0$, hence $M/p_k^2\in \Div(A)$ and $M/p_ip_k^2\in \Div(B)$.
\end{itemize}
It is then easy to check that $A_{x,b}$ cannot have nonempty intersection with any two of the above lines, either 
by Lemma \ref{triangles} or due to direct divisor conflict.

\begin{lemma}\label{db2-lemma3}
Assume (DB2) and (\ref{db2-e00}), and let $x$ be as above. Then $A_{x,b}\cap\ell_k(a_0)=\emptyset$.
\end{lemma} 

\begin{proof}
Assume for contradiction that 
\begin{equation}\label{db2-e2}
A_{x,b}\subset\ell_k(a_0).
\end{equation}
Then $M/p_k^2\in \Div(A)$, hence by Lemma \ref{db2-lemma2} we have $\Phi_{M/p_k}|A$. We are also assuming as part of (DB2) that $\Phi_M|A$. It follows that $A$ is null with respect to all cuboids of type $(M, (1,1,2),1)$. 

Consider any such cuboid with one vertex at $a_0$, a second one at a point $x_j\in I^c\times J\times K^c$ so that $(x_j-a_0,M)=M/p_j$, and a third one at a point $z\in \ell_k(x)$ with $(x-z,M)=M/p_k^2$. By (\ref{db2-e00}),
none of the cuboid vertices in the plane $\Pi(a_0,p_k^2)$ except for $a_0$ are in $A$. In order to balance this cuboid, one of the vertices $z_i$ and $z_j$ must be in $A$, where $(z_i-z,M)=M/p_i$ and $(z_j-z,M)=M/p_j$. However, if $z_j\in A$, then $M/p_ip_k^2 \in \Div(\{z_j\},I\times J\times K)\subset\Div(A)$, which contradicts (\ref{db2-e2}) due to divisor conflict. It follows that $z_i\in A$.

Repeating this argument for all $z\in \ell_k(x)$ with $(x-z,M)=M/p_k^2$, we get that
\begin{equation}\label{db2-e3}
\bbA_{M/p_k^2}[a_0]=\phi(p_k^2).
\end{equation}
If $p_i<p_k$, then we must in fact have $p_i\leq p_k-2$ (since both are odd primes), so that $\phi(p_k^2)=p_k(p_k-1)>p_ip_k$. In this case, (\ref{db2-e3}) contradicts Lemma \ref{planebound}.

Assume now that $p_i>p_k$. In this case, (\ref{db2-e2}) and (\ref{db2-e3}) imply that 
\begin{align*}
1&=\frac{1}{\phi(p_ip_k^2)} \bbA_{M/p_ip_k^2}[x|\ell_k(a_0)] \bbB_{M/p_ip_k^2}[b]
\\
&= \frac{1}{\phi(p_ip_k^2)} \bbA_{M/p_k^2}[a_0] \bbB_{M/p_ip_k^2}[b]\\
&= \frac{1}{\phi(p_i)}  \bbB_{M/p_ip_k^2}[b].
\end{align*}
But when $p_i>p_k$, this is not possible without introducing $M/p_i$ or $M/p_ip_k$ as divisors of $B$, which would contradict (\ref{db-e1000}).
\end{proof}

Repeating Lemmas \ref{db2-lemma2} and \ref{db2-lemma3} with $i$ and $k$ interchanged, we get that for each $b\in B$ we must have one of
$$
A_{x,b}\subset\ell_i(x) \hbox{ or } A_{x,b}\subset\ell_k(x).
$$
Now let $x'\in (I\times J\times K^c)\cap \Pi(x,p_i^2)$. The same analysis, with the $p_j$ and $p_k$ directions interchanged, shows that 
for each $b\in B$ we must have one of
$$
A_{x',b}\subset\ell_i(x') \hbox{ or } A_{x',b}\subset\ell_j(x').
$$
If $A_{x,b}\subset\ell_k(x)$ and $A_{x',b'}\subset\ell_j(x')$ for some $b,b'\in B$, we use the same fiber crossing argument as in the case (DB1). By Lemma \ref{1dim_sat-cor}, $A$ must contain an $M$-fiber in the $p_k$ direction at distance $M/p_k^2$ from $x$, as well as an $M$-fiber in the $p_j$ direction at distance $M/p_j^2$ from $x'$, both of them in the plane $\Pi:=\Pi(x,p_i^2)$ and disjoint from each other. The same plane also contains $(p_j-1)(p_k-1)$ elements of $I\times J\times K$. Therefore
$$
|\Pi\cap A|\geq (p_j-1)(p_k-1)+p_j+p_k=p_jp_k+1>p_jp_k,
$$
which contradicts Lemma \ref{planebound}.

It follows that either $A_{x,b}\subset\ell_i(x) $ for all $b\in B$, or else $A_{x',b}\subset\ell_i(x')$ for all $b\in B$. An application of Lemma \ref{db2-lemma0} concludes the proof of the theorem.


\section{Corners}\label{corner-section}


In this section we address the extended corner structure, defined in Definition \ref{corner}.
We will assume the following.

\medskip\noindent
{\bf Assumption (C):} Assume that $A\oplus B=\ZZ_M$, where $M=p_i^{n_i}p_j^{n_j}p_k^{n_k}$, $|A|=p_ip_jp_k$, 
and $\Phi_M|A$. Moreover, assume that $A$ contains a $p_i$ extended corner in the sense of Definition \ref{corner} on a $D(M)$-grid $\Lambda$, that is, there exist $a,a_i\in A\cap\Lambda$
 with $(a-a_i,M)=M/p_i$ such that
\begin{itemize}
\item $A\cap(a*F_j*F_k)$ is $M$-fibered in the $p_j$ direction but not in the $p_k$ direction,
\item $A\cap(a_i*F_j*F_k)$ is $M$-fibered in the $p_k$ direction but not in the $p_j$ direction.
\end{itemize}
In particular, we may choose $a,a_i$ so that
$$
a*F_j\subset A,\ \ a_i*F_k\subset A,
$$
$$
a_i*F_j\not\subset A,\ \ a*F_k\not\subset A.
$$
We fix such $a,a_i$ for the rest of this section.

\medskip

The following theorem is our main result for the extended corner structure.

\begin{theorem}\label{cornerthm}
Assume that (C) holds with $n_j=n_k=2$. 
Then the tiling $A\oplus B=\ZZ_M$ is T2-equivalent to $\Lambda\oplus B=\ZZ_M$ via fiber shifts.
By Corollary \ref{get-standard}, both $A$ and $B$ satisfy (T2).
\end{theorem}

We first note that when $A$ contains a $p_i$ corner, 
\begin{equation}\label{cornerdiv}
\{D(M)|m|M\}\setminus \Div(A)\subseteq\{M/p_jp_k\}
\end{equation}

\begin{lemma} [\bf{Size-Divisor Lemma}] \label{sizediv}
Assume that $A\subset \ZZ_M$, and that there exist $a,a_i\in A$
 with $(a-a_i,M)=M/p_i$ such that  
$$
a*F_j\subset A,\ \ a_i*F_k\subset A,\ \ a*F_k\not\subset A.
$$
(In particular, this is true if (C) holds.)
Suppose that $\Phi_M \Phi_{M/p_j}|A$. Then we have the following.

(i) For all $a_k\in A$ such that $(a_k-a_i,M)=M/p_k$ and $\bbA_{M/p_i}[a_k|\ell_k(a)]=0$, we have
$$
\bbA_{M/p_j^{2}}[a]+\bbA_{M/p_j^{2}}[a_k]\geq\phi(p_j^{2}),
$$

(ii) $M/p_j^{2},\,M/p_ip_j^{2} p_k\in \Div(A)$. Moreover, if $\bbA_{M/p_j^{2}}[a]>0$, then $M/p_ip_j^{2}\in \Div(A)$, and if $\bbA_{M/p_j^{2}}[a]<\phi(p_j^{2})$, then $M/p_j^{2} p_k\in \Div(A)$.

\end{lemma}

\begin{proof}
Under the assumptions of the lemma, $A$ is $\calt$-null with respect to the cuboid type $\calt=(M,\vec{\delta},1)$, where $\vec{\delta}=(1,2,1)$. Now the lemma follows by considering all possible $\calt$ cuboids with vertices at $a,a_i,a_k$ and $x_k$, where $x_k\not\in A$ satisfies $(x_k-a,M)=M/p_k$ and $(x_k-a_k,M)=M/p_i$.
\end{proof}

\begin{corollary}\label{N_jN_knot}
Assume (C).
Then $\Phi_{N_j}\Phi_{N_k}\nmid A$.
\end{corollary}

\begin{proof} 
Assume that $\Phi_{N_j}\Phi_{N_k}|A$. By Lemma \ref{sizediv} ($i$), we have
$$
|A\cap(\Pi(a,p_i^{n_i})\cup\Pi(a_i,p_i^{n_i}))|\geq  p_j+\phi(p_j^2)+p_k+\phi(p_k^2)=p_j^2+p_k^2>2p_jp_k
$$
hence $\max\{|A\cap\Pi(a,p_i^{n_i})|,|A\cap\Pi(a_i,p_i^{n_i})|\}> p_jp_k$, which contradicts Lemma \ref{planebound}. 
\end{proof}

By Corollary \ref{N_jN_knot}, it suffices to consider the following sets of assumptions.

\medskip\noindent
{\bf Assumption (C1):} Assume that (C) holds, $p_k<p_j$, and $\Phi_{N_k}|B$.

\medskip\noindent
{\bf Assumption (C2):} Assume that (C) holds,  $p_k<p_j$, $\Phi_{N_j}|B$, and 
\begin{equation}\label{dzielnikzabrany}
M/p_jp_k\notin \Div(B).
\end{equation}
(In particular, (\ref{dzielnikzabrany}) holds if
either $\bbA_{M/p_jp_k}[a]>0$ or $\bbA_{M/p_jp_k}[a_i]>0$.)

\medskip\noindent
{\bf Assumption (C3):} Assume that (C) holds, $p_k<p_j$,
$\Phi_{N_j}|B$, and 
$$\bbA_{M/p_jp_k}[a]=\bbA_{M/p_jp_k}[a_i]=0.$$

\medskip

We first prove in Section \ref{ec12} that under either of the assumptions (C1) or (C2), the conclusion of Theorem \ref{cornerthm} holds. We then prove in Section \ref{ec3} that if (C3) holds, then we must also have $\Phi_{N_k}|B$,
so that we may return to (C1) to complete the proof.

\subsection{Cases (C1) and (C2)}\label{ec12}
The first two cases are similar and will be considered together. 

\begin{lemma}\label{Bfiber} 
Assume that (C1) holds with $n_k=2$.
Then:

\smallskip
(i) Every element of $B$ belongs to an $N_k$-fiber in $B$ in either the $p_j$ or $p_k$ direction
(not necessarily the same for all $b$). 

\smallskip
(ii) Furthermore, suppose that there is a $b\in B$ that does not belong to an $N_k$-fiber in the $p_j$ direction in $B$.
Then for all $x_j\in \ZZ_M\setminus A$ such that $(a_i-x_j,M)=M/p_j$ we have $A_{x_j,b}\subset\ell_k(x_j)$,
with the product saturated by an $N_k$-fiber in $B$ in the $p_k$ direction, rooted at $b$, and 
an $M$-cofiber in $A$ at distance $M/p_k^2$ from $x_j$.

\smallskip\noindent
Moreover, this lemma does not require the assumption $p_k<p_j$, therefore the same conclusions with $j$ and $k$ interchanged hold under the assumptions (C2) and (C3).

\end{lemma}
\begin{proof}
Observe that, by (\ref{cornerdiv}),
\begin{equation}\label{rest_corn}
\bbB^{N_k}_{N_k/p_i}[b]=0 \text{ for all } b\in B.
\end{equation} 
Assume first that $p_i=2$, and suppose that there is a $b_0\in B$ that does not belong to an $N_k$-fiber in either the
$p_k$ or the $p_j$ direction.
In this case, by (\ref{rest_corn}) and Lemma \ref{even_struct} with $N=N_k$ we have 
\begin{equation}\label{divres}
N_k/p_k=M/p_k^2\in \Div(B).
\end{equation}
Let $x_j\in \ZZ_M\setminus A$ with $(a_i-x_j,M)=M/p_j$, and consider the saturating set $A_{x_j}$. Let $a_j$ be the element of $A$ satisfying $(a-a_j,M)=M/p_j,(x_j-a_j,M)=M/p_i$. Apply Corollary \ref{bispan-corollary} (i) to $A_{x_j,b_0}$, once with respect to $a$ and again with respect to $a_j$, to get 
$$
A_{x_j,b_0}\subset \ell_k(x_j)\cup\ell_k(a_i)\cup\ell_k(a)\cup\ell_k(a_j).
$$
By (\ref{cornerdiv}), no top level divisors except possibly $M/p_jp_k$ are available to contribute to the product identity
$\langle \bbA[x_j], \bbB[b_0]\rangle=1$. 
Therefore, if $A_{x_j,b_0}\cap\ell_k(a_j)\neq \emptyset$, we must have
$$
\bbA_{M/p_ip_k^2}[x_j|\ell_k(a_j)]=\bbA_{M/p_k^2}[a_j]>0 
$$ 
which contradicts (\ref{divres}). Similarly $A_{x_j,b_0}\cap\ell_k(a)\neq \emptyset$ implies
$$
\bbA_{M/p_ip_jp_k^2}[x_j|\ell_k(a)]=\bbA_{M/p_k^2}[a]>0 
$$ 
which, again, contradicts (\ref{divres}). 
Next,
if $A_{x_j,b_0}\cap\ell_k(a_i)\neq \emptyset$, then by (\ref{divres})
\begin{equation}\label{linesat2}
\bbA^{N_k}_{N_k/p_j}[x_j|\ell_k(a_i)]\bbB^{N_k}_{N_k/p_j}[b_0]=\bbA_{M/p_jp_k}[x_j|\ell_k(a_i)]\bbB_{M/p_jp_k}[b_0]>0,
\end{equation}
and if $A_{x_j,b_0}\cap\ell_k(x_j)\neq \emptyset$, then 
\begin{equation}\label{linesat1}
\bbA^{N_k}_{N_k/p_k}[x_j]\bbB^{N_k}_{N_k/p_k}[b_0]=\bbA_{M/p_k^2}[x_j]\bbB_{M/p_k^2}[b_0]>0.
\end{equation}
By Lemma \ref{triangles}, we cannot have both (\ref{linesat2}) and (\ref{linesat1}) for a fixed $b_0$. Hence either 
$A_{x_j,b_0}\subset\ell_k(a_i)$ or $A_{x_j,b_0}\subset\ell_k(x_j)$. In the first case, we have
$$
\bbA_{M/p_jp_k}[x_j|\ell_k(a_i)]\bbB_{M/p_jp_k}[b_0]=\bbA_{M/p_k}[a_i]\bbB_{M/p_jp_k}[b_0]=\phi(p_jp_k);
$$ 
since $\bbA_{M/p_k}[a_i]=\phi(p_k)$, it follows that $\bbB_{M/p_jp_k}[b_0]=\bbB^{N_k}_{N_k/p_j}[b_0]=\phi(p_j)$, 
hence $b_0$ belongs to an $N_k$-fiber in the $p_j$ direction, contradicting the choice of $b_0$.
In the second case, we have
$$
\bbA_{M/p_k^2}[x_j]\bbB_{M/p_k^2}[b_0]=\phi(p_k^2);
$$
since $M/p_k\in \Div(A)$, this is only possible if $A$ and $B$ contain a configuration as in part (ii) of the lemma.
In this case, it follows that $b_0$ belongs to an $N_k$-fiber in the $p_k$ direction, contradicting our initial choice of $b_0$. This concludes the proof of (i).

We now prove (ii) in the case $p_i=2$. Assume that $b\in B$ does not belong to an $N_k$-fiber in the $p_j$ direction in $B$. By (i), it must belong to an $N_k$-fiber in the $p_k$ direction in $B$, so that in particular (\ref{divres}) holds. With this in place, we now repeat the proof of (i) until we get to $A_{x_j,b}\subset\ell_k(a_i)$ or $A_{x_j,b}\subset\ell_k(x_j)$. By the assumption on $b$, we must be in the second case for all $x_j$ chosen as above.
This concludes the proof for $p_i=2$.

%

Assume now that $p_i>2$. Then part (i) of the lemma follows directly from
Lemma \ref{gen_top_div_mis}, (\ref{rest_corn}) and the fact that $B$ satisfies (\ref{bin_cond}) with $N=N_k$ and $c=1$.
For (ii), suppose that there is a $b\in B$ that does not belong to an $N_k$-fiber in $B$ in the $p_j$ direction.
By (i), there is an $N_k$-fiber in $B$ in the $p_k$ direction containing $b$. In particular, this implies that (\ref{divres}) 
holds. The rest of the proof of (ii) is the same as for the $p_i=2$ case.
\end{proof}

\begin{corollary}\label{Bfibercor}

\smallskip

(i) Assume (C1) with $n_k=2$. Then the pair $(A,B)$ has a $(1,2)$-cofibered structure in the $p_k$ direction. Moreover, for each $x_j\in \ZZ_M\setminus A$ such that $(a_i-x_j,M)=M/p_j$, $A$ has a cofiber
at distance $M/p_k^2$ from $x_j$.

\smallskip

(ii) Assume (C2) with $n_j=2$, then the same conclusion holds with $k$ and $j$ interchanged.
\end{corollary}

\begin{proof}
(i) Let $b\in B$. By Lemma \ref{Bfiber}, it suffices to prove that we cannot have $\bbB_{M/p_jp_k}[b]=\phi(p_j)$. In fact, we will show 
$$
\bbB_{M/p_jp_k}[b]\leq\phi(p_k);
$$
moreover, this holds independently of the assumption $\Phi_{N_k}|B$. Indeed, by the corner assumption we have $\bbA_{M/p_jp_k}[x_k]\geq\phi(p_j)$ for all $ x_k\in\ZZ_M\setminus A$ with $(x_k-a,M)=M/p_k$. Hence 
\begin{align*}
1=\langle \bbA[x_k],\bbB[b]\rangle &\geq\frac{1}{\phi(p_jp_k)}\bbA_{M/p_jp_k}[x_k]\bbB_{M/p_jp_k}[b]
\\
&\geq\frac{1}{\phi(p_k)}\bbB_{M/p_jp_k}[b]
\end{align*}
as required.

\medskip

(ii) By  (\ref{dzielnikzabrany}), we have $\bbB_{M/p_jp_k}[b]=0$.
The conclusion follows again from Lemma \ref{Bfiber}.
\end{proof}

The rest of the proof is the same in cases (C1) and (C2), except that $p_j$ and $p_k$ are interchanged. Without loss of generality, we assume that (C1) holds.

By Lemma \ref{fibershift}, we may shift each of the cofibers provided by Corollary \ref{Bfibercor} to its respective point $x_j$. Let $A'\subset\ZZ_M$ be the set thus obtained. Then $A$ is T2-equivalent to $A'$, we have $A'\oplus B=\ZZ_M$, and for every $x_j\in \ZZ_M\setminus A$ with $(x_j-a_i,M)=M/p_j$,
we have $x_j*F_k\subset A'$.
Furthermore, $A'$ differs from $A$ only along the lines $\ell_k(x_j)$ with $x_j$ as above,
hence it follows from (C) that 
for every $a_j\in A$ with $(a_j-a_i,M)=M/p_j$, we also have $a_j*F_k\subset A'$. Therefore
$a_i*F_j*F_k\subset A'$. 
By Lemma \ref{planebound}, we must in fact have
\begin{equation}\label{maxedplane}
A'\cap\Pi(a_i,p_i^{n_i}) = a_i*F_j*F_k.
\end{equation}

\begin{corollary}\label{wniosekwniesiony}
Assume that (C1) holds with $n_k=2$.
Then $A'\subset \Pi(a,p_i^{n_i-1})$, where $a$ is as in (C). Moreover, $\Phi_{p_i^{n_i}}|A'$.

\end{corollary}

\begin{proof}
By (\ref{maxedplane}) and (C), 
we have $|A'\cap\Pi(a_i,p_i^{n_i-1})|>p_jp_k$. The conclusion now follows from Corollary \ref{planegrid}. 
\end{proof}

If $n_i=n_k=2$, we have $p_i\parallel |B|$, hence Corollary \ref{wniosekwniesiony} and Theorem \ref{subgroup-reduction} imply that $A'$ and $B$ satisfy (T2). Since $A$ and $A'$ are T2-equivalent, the same is true for $A$. The additional arguments below are needed to prove the full conclusion of Theorem \ref{cornerthm}, namely, T2-equivalence between $A'$ (hence $A$) and $\Lambda$ for $n_i\geq 2$. This is needed for the classification result in Theorem \ref{unfibered-mainthm}, as well as for the applications in Section \ref{fibered-sec}.

\begin{corollary}\label{p_jfibercor}
Assume that (C1) holds with $n_j=n_k=2$, and define $A'$ as above. 
Then the pair $(A',B)$ has a $(1,2)$-cofibered structure in the $p_j$ direction, with cofibers in $A'$ at distance $M/p_j^2$ from each $x_k\in \ZZ_M\setminus A$ such that $(x_k-a,M)=M/p_k$.
\end{corollary}

\begin{proof}
By (\ref{maxedplane}) and (C), we have $\{D(M)|m|M\}\subset\Div(A')$. The corollary follows by applying Lemma \ref{flatcorner} to each $x_k$.
\end{proof}

We can now complete the proof under the assumption (C1). By Corollary \ref{wniosekwniesiony}, we have
$\Phi_{p_i^{n_i}}\mid A'$.
This together with (\ref{maxedplane}) implies that 
$$
A'\subset \bigcup_{z\in a_i*F_i} \Pi(z,p_i^{n_i}),
$$
and that for each plane $\Pi_z:= \Pi(z,p_i^{n_i})$ with $z\in a_i*F_i$, we have 
\begin{equation}\label{plane-sieve2}
|A'\cap\Pi_z|=p_jp_k.
\end{equation}

Applying Lemma \ref{flat-cuboids} to $A'$, and using (\ref{maxedplane}),
we get that each set $A'\cap\Pi_z$ is a disjoint union of $M$-fibers 
in the $p_j$ and $p_k$ directions.
However, we also know that
$B$ is ${N_j}$-fibered in the $p_j$ direction and ${N_k}$-fibered in the $p_k$ direction, hence
\begin{equation}\label{Bdivfinal}
	\{M/p_j^2,M/p_k^2,M/p_j^2p_k^2\}\subset \Div(B).
\end{equation}
It follows that each set $A'\cap\Pi_z$
must in fact be $M$-fibered in one of the $p_j$ and $p_k$ directions. Assuming without loss of generality that $A'\cap\Pi_z$ is $M$-fibered in the $p_j$ direction for some $z\in a_i*F_i$, and taking (\ref{plane-sieve2}) and (\ref{Bdivfinal}) into account, we get
$$
A'\cap\Pi_z=\{u_1,\dots,u_{p_k}\}*F_j,
$$
where for each $\nu\neq\nu'$ we have $(u_\nu - u_{\nu'},M)\in\{M/p_k,M/p_j^2p_k\}$.
Using the cofibered structure in the $p_j$ direction, and considering each $u_\nu*F_j$ as cofiber, we
apply Lemma \ref{fibershift} if necessary to reduce to the case where $p_j\mid a_i-u_\nu$ for all $\nu$.
This aligns the fibers in $\Pi_z$ to a grid $u_1*F_j*F_k$, fibered in both directions.
If we do not have $p_k\mid a_i-u_1$ at this point, we apply Lemma \ref{fibershift} again, this time in the
$p_k$ direction. 
Repeating this procedure in each plane, we get that $A'$ is
T2-equivalent to $A^\flat=\Lambda$. This ends the proof in this case.


\subsection{Case (C3)}\label{ec3}
 In this case, we are assuming that $\Phi_{N_j}|B$ and $\bbA_{M/p_jp_k}[a]=\bbA_{M/p_jp_k}[a_i]=0$. By (C), this implies that $\bbA_{M/p_k}[a]=\bbA_{M/p_j}[a_i]=0$.

\begin{lemma}\label{ksatset}
Assume (C3) with $n_k=2$. 
Let $a_j\in A$ and $x_j\in \ZZ_M\setminus A$ satisfy $(a_j-x_j,M)=M/p_i$ and $(a-a_j,M)=(a_i-x_j,M)=M/p_j$. If $\Phi_{N_k}\nmid B$, then
$A_{x_j}\subset\ell_k(a_i)$.
\end{lemma}

\begin{proof}
The assumptions of the lemma imply that $\Phi_{N_k}|A$. Applying Lemma \ref{sizediv} to $A$ with $j$ and $k$ interchanged, we have
\begin{equation}\label{divfest1}
M/p_k^2,\,M/p_ip_jp_k^2\in \Div(A),
\end{equation}
\begin{equation}\label{sizedivk}
\bbA_{M/p_k^2}[a_i]+\bbA_{M/p_k^2}[a_j]\geq \phi(p_k^2).
\end{equation} 
Let $b\in B$, and consider the saturating set $A_{x_j,b}$. Applying Corollary \ref{bispan-corollary} (i) to $A_{x_j,b}$, once with respect to $a_i$ and again with respect to $a_j$,  we have
$$
A_{x_j,b}\subset \ell_k(x_j)\cup\ell_k(a_i)\cup\ell_k(a)\cup\ell_k(a_j).
$$
Using (\ref{cornerdiv}) and (\ref{divfest1}), we conclude that $A_{x_j,b}\cap\ell_k(x_j)=\emptyset$ since $M/p_k,M/p_k^2\notin \Div(B)$, and similarly $A_{x_j,b}\cap\ell_k(a)=\emptyset$ since 
$M/p_ip_j,M/p_ip_jp_k,M/p_ip_jp_k^2\notin \Div(B)$. 

Suppose that $A_{x_j,b}\cap\ell_k(a_j)\neq\emptyset$. By (\ref{cornerdiv}) again, this implies that 
$$
\bbA_{M/p_ip_k^2}[x_j|\ell_k(a_j)]\bbB_{M/p_ip_k^2}[b]>0,
$$
so that $\bbA_{M/p_ip_k^2}[a|\ell_k(a_i)]=\bbA_{M/p_k^2}[a_i]=0$. By (\ref{sizedivk}) we have $\bbA_{M/p_k^2}[a_j]=\phi(p_k^2)$ for all $(a_j-a,M)=M/p_j$. We will prove this contradicts Lemma \ref{planebound}. Indeed, we have 
$$
|A\cap\Pi(a,p_i^{n_i})|\geq \bbA_M[a]+\bbA_{M/p_j}[a]+\sum_{(a-a_j,M)=M/p_j}\bbA_{M/p_k^2}[a_j] = p_j +\phi(p_jp_k^2).
$$
A simple calculation shows that the last expression exceeds $p_jp_k$ if and only if $p_k>1+\frac{1}{\phi(p_j)}$, which holds true since $p_j>p_k\geq 2$. We therefore conclude that $A_{x_j,b}\subset\ell_k(a_i)$.
\end{proof}


\begin{proposition}\label{N_j->N_k}
(i) Assume (C3) with $n_j=n_k=2$. Then $\Phi_{N_k}|B$.

\smallskip
(ii) Assume that $p_i=2$, and that (C3) holds with $n_k=2$. (In this case, we do not need to assume that $n_j=2$.) Then $\Phi_{N_k}|B$.
%
%
\end{proposition}

\begin{proof}
We prove (ii) first, since the proof in this case is immediate and straightforward. Assume, by contradiction, that $p_i=2$ and $\Phi_{N_k}\nmid B$, so that $\Phi_{N_k}|A$. By Lemma \ref{sizediv}, (\ref{sizedivk}) holds for all $a_j\in A$ with $(a-a_j,M)=M/p_j$. By Lemma \ref{planebound} we have
\begin{align*}
p_ip_k & \geq |A\cap\Pi(a,p_j^{n_j})|\\
& \geq \bbA_M[a]+\bbA_M[a_i]+\bbA_{M/p_k}[a_i]+\bbA_{M/p_k^2}[a_i]\\
& =p_k+1+\bbA_{M/p_k^2}[a_i].
\end{align*}
Since $p_i=2$, we get $\bbA_{M/p_k^2}[a_i]\leq \phi(p_k)$. From (\ref{sizedivk}) we deduce $\bbA_{M/p_k^2}[a_j]\geq \phi(p_k^2)-\phi(p_k)$, hence
\begin{align*}
|A\cap\Pi(a,p_i^{n_i})| & \geq
\bbA_M[a]+\bbA_{M/p_j}[a]+\sum_{a_j:(a_j-a,M)=M/p_j}\bbA_{M/p_k^2}[a_j]\\
& \geq p_j +\phi(p_j)(\phi(p_k^2)-\phi(p_k))\\
& =p_j +\phi(p_j)\phi(p_k)^2
\end{align*}
We show the latter exceeds $p_jp_k$, which contradicts Lemma \ref{planebound}. Indeed, $p_j +\phi(p_j)\phi(p_k)^2>p_jp_k$ if and only if $\phi(p_j)\phi(p_k)^2>p_j\phi(p_k)$. This, in turn, is equivalent to $p_k>2+\frac{1}{\phi(p_j)}$, which clearly holds true since $p_k\geq3$. 

We now prove (i). Assume that $p_i>2$, and let $b\in B$. Applying Lemma \ref{Bfiber} to $B$ with $p_j$ and $p_k$ interchanged, we get that at least one of the following holds: 
\begin{equation}\label{jfiber}
\bbB^{N_j}_{N_j/p_j}[b]=\phi(p_j),
\end{equation}
\begin{equation}\label{kfiber}
\bbB^{N_j}_{N_j/p_k}[b]=\phi(p_k).
\end{equation} 
 
Assume, by contradiction, that $\Phi_{N_k}\nmid B$. Then $\Phi_{N_k}|A$. 
Let $a,a_i, a_j$, and $x_j$ be as in Lemma \ref{ksatset}.
By Lemma \ref{ksatset}, we have $A_{x_j,b}\subset\ell_k(a_i)$, hence
\begin{equation}\label{restsat}
\frac{1}{\phi(p_jp_k)}\bbA_{M/p_jp_k}[x_j|\ell_k(a_i)]\bbB_{M/p_jp_k}[b]+\frac{1}{\phi(p_jp_k^2)}\bbA_{M/p_jp_k^2}[x_j|\ell_k(a_i)]\bbB_{M/p_jp_k^2}[b]=1.
\end{equation}
Notice that
\begin{equation}\label{fact1}
\bbA_{M/p_jp_k}[x_j|\ell_k(a_i)]=\phi(p_k),\,\bbB_{M/p_jp_k}[b]\leq\phi(p_k).
\end{equation} 
Indeed, the first part of (\ref{fact1}) follows from the corner structure. For the second part, recall from (C3) that $p_k<p_j$, so that $\bbB_{M/p_jp_k}[b]>\phi(p_k)$ would imply $M/p_k\in \Div(B)$, contradicting divisor exclusion since the corner structure implies $M/p_k\in\Div(A)$. 

Applying (\ref{fact1}) to (\ref{restsat}), we get
\begin{equation}\label{fact2}
\bbA_{M/p_jp_k^2}[x_j|\ell_k(a_i)]\bbB_{M/p_jp_k^2}[b]>0,
\end{equation}
and in particular $M/p_jp_k^2\in \Div(B)$. We claim that this implies that
\begin{equation}\label{fact3}
\bbA_{M/p_jp_k^2}[x_j|\ell_k(a_i)]=\bbA_{M/p_k^2}[a_i]=\phi(p_k^2).
\end{equation}
Indeed, suppose that (\ref{fact3}) fails, then $\bbA_{M/p_k^2}[a_i]<\phi(p_k^2)$. By (\ref{sizedivk}), we have $\bbA_{M/p_k^2}[a_j]>0$, hence $\bbA_{M/p_j p_k^2}[a]>0$. It follows that $M/p_jp_k^2\in\Div(A)\cap \Div(B)$, which is a contradiction.

\begin{lemma}\label{diagonal}
Assume that (C3) holds with $n_j=n_k=2$, but $\Phi_{N_k}\nmid B$. Then
there is at least one $b_0\in B$ for which (\ref{jfiber}) holds and (\ref{kfiber}) fails.
\end{lemma}

\begin{proof}
Suppose that the lemma is false, so that (\ref{kfiber}) holds for all 
$b\in B$. We first prove that then 
\begin{equation}\label{diag-divide}
\forall b\in B,\ \ p_k\,|\, \bbB_{M/p_jp_k^2}[b].
\end{equation}
Indeed, by the corner structure and (\ref{divfest1}) we have
\begin{equation}\label{nostraightdivisors}
M/p_j,M/p_k,M/p_k^2\not\in\Div(B),
\end{equation}

Let $b\in B$. By (\ref{kfiber}), there is a
$p_k$-tuple of elements $\{b_0=b,b_1,\dots,b_{p_k-1}\}\subset B$ such that $b_\nu\in\Lambda(b,M/p_jp_k)$.
By (\ref{nostraightdivisors}), 
we must have $(b_\nu- b_{\nu'},M)= M/p_jp_k$ for all $\nu\neq\nu'$.

Suppose now that $b'\in B$ satisfies $(b-b',M)=M/p_jp_k^2$, then $b'$ belongs to a similar
$p_k$-tuple  $\{b'_0=b',b'_1,\dots,b'_{p_k-1}\}\subset B$ with $(b'_\nu- b'_{\nu'},M)= M/p_jp_k$ for all $\nu\neq\nu'$.
By (\ref{nostraightdivisors}) again, we must have $(b_\nu-b'_{\nu'},M)= M/p_jp_k^2$
for all $\nu,\nu'$. Hence the $p_k$-tuples associated with different elements of $B$ are either identical or disjoint. This proves (\ref{diag-divide}).

Applying (\ref{fact1}), (\ref{kfiber}),
(\ref{fact3}), and (\ref{diag-divide}) to (\ref{restsat}), we get that
$$
1=\frac{\phi(p_k)^2}{\phi(p_jp_k)}+\frac{\phi(p_k^2)\cdot Cp_k}{\phi(p_jp_k^2)}=\frac{\phi(p_k)+Cp_k}{\phi(p_j)} 
$$
for some integer $C$. But this implies $p_j=p_k(C+1)$, which  is impossible since $p_j$ is prime.
This proves the lemma.
\end{proof}

By Lemma \ref{diagonal}, there exists a $b_0\in B$ such that $\bbB^{N_j}_{N_j/p_k}[b_0]<\phi(p_k)$ and $\bbB^{N_j}_{N_j/p_j}[b_0]=\phi(p_j)$. In particular, 
\begin{equation}\label{fact777}
M/p_j^2\in\Div(B,b_0).
\end{equation}
Applying (\ref{fact2}) to $b_0$, we get that $M/p_jp_k^2\in\Div(B,b_0)$. Finally, if $b',b''$ are elements of $B$ with
$(b_0-b',M)=M/p_j^2$ and $(b_0-b',M)=M/p_jp_k^2$, then $(b''-b',M)=M/p_j^2p_k^2$. Therefore
\begin{equation}\label{Bdiv}
M/p_j^2,M/p_jp_k^2,M/p_j^2p_k^2\in \Div(B).
\end{equation} 

\begin{lemma}\label{diagonal2}
Assume that (C3) holds with $n_j=n_k=2$, but $\Phi_{N_k}\nmid B$. Then
\begin{equation}\label{planebound2}
|A\cap \Pi(a_i,p_i^{n_i}) |=p_k^2.
\end{equation} 
\end{lemma}

\begin{proof}
Using (\ref{fact3}) and the fact that $\bbA_{M/p_k}[a_i]=\phi(p_k)$, we have
$$
|A\cap \ell_k(a_i)|=p_k^2.
$$
On the other hand, we claim that
\begin{equation}\label{planebound2tetrium}
\bbA_m[a_i]=0 \ \forall m|M \hbox{ such that }p_i^{n_i}|m|M,\  m\neq M, M/p_k,M/p_k^2.
\end{equation}
Indeed, we can exclude the divisors in (\ref{Bdiv}). 
Furthermore, $M/p_j,M/p_jp_k \not\in\Div(A,a_i)$
by the $p_i$ corner assumption.
It remains to check that $M/p_j^2p_k \not\in\Div(A,a_i)$. If we had an element $a'\in A$ with 
$(a_i-a',M)= M/p_j^2p_k$, then there would be an element $a_k\in a_i*F_k\subset A$ with $(a_k-a',M)=M/p_j^2$,
which is again prohibited by (\ref{Bdiv}).

By (\ref{planebound2tetrium}), all elements of $A$ in the plane $\Pi(a_i,p_i^{n_i})$ must in fact lie on the line $\ell_k(a_i)$, so that
$$
|A\cap \Pi(a_i,p_i^{n_i}) |=|A\cap \ell_k(a_i)|=p_k^2
$$
as claimed.
\end{proof}

\begin{lemma}\label{diagonal3}
Assume that (C3) holds with $n_j=n_k=2$, but $\Phi_{N_k}\nmid B$. 
Let $a_k\in A, x_k\in\ZZ_M\setminus A$ with $(a_k-x_k,M)=M/p_i,(a_i-a_k,M)=(a-x_k,M)=M/p_k$.
Then for $b_0$ as above,
$$
A_{x_k,b_0}\subset\ell_j(x_k).
$$
\end{lemma}
\begin{proof}
Applying Corollary \ref{bispan-corollary} (i) to $A_{x_k,b_0}$, once with respect to $a$ and again with respect to $a_k$, we get 
$$
A_{x_k,b_0}\subset \ell_j(x_k)\cup\ell_j(a_i)\cup\ell_j(a)\cup\ell_j(a_k).
$$
Observe first that $A_{x_k,b_0}\cap(\ell_j(a_k)\cup\ell_j(a_i))=\emptyset$, for otherwise, by (\ref{cornerdiv}) we would have either $\bbA_{M/p_j^2}[a_i]>0$ or $\bbA_{M/p_j^2}[a_k]>0$; both are not allowed due to (\ref{Bdiv}).

Next, we claim that we cannot have both $A_{x_k,b_0}\cap\ell_j(a)\neq \emptyset$ and $A_{x_k,b_0}\cap\ell_j(x_k)\neq \emptyset$. 
Indeed, the former implies $\bbA_{M/p_jp_k}[x_k|\ell_j(a)]\bbB_{M/p_jp_k}[b_0]>0$ and the latter implies
$\bbA_{M/p_j^2}[x_k]\bbB_{M/p_j^2}[b_0]>0$; having both would contradict Lemma \ref{triangles}.

Suppose that $A_{x_k,b_0}\subset\ell_j(a)$. Then
\begin{equation}\label{ell-j-a}
1=\frac{1}{\phi(p_jp_k)}\bbA_{M/p_jp_k}[x_k|\ell_j(a)]\bbB_{M/p_jp_k}[b_0]
+ \frac{1}{\phi(p_j^2 p_k)}\bbA_{M/p_j^2 p_k}[x_k|\ell_j(a)]\bbB_{M/p_j^2 p_k}[b_0].
\end{equation}
We chose $b_0$ so that $\bbB_{M/p_jp_k}[b_0]<\phi(p_k)$. By the corner assumption, $\bbA_{M/p_jp_k}[x_k|\ell_j(a)]=\phi(p_j)$. We also have $\bbA_{M/p_j^2p_k}[x_k|\ell_j(a)]=\bbA_{M/p_j^2}[a]=0$,
by (\ref{Bdiv}). Therefore the right side of (\ref{ell-j-a}) is strictly less than 
$\frac{1}{\phi(p_jp_k)}\phi(p_j)\phi(p_k)=1$, a contradiction.
The lemma follows.
\end{proof}

We now complete the proof of the proposition. Recall from (C3) that 
\begin{equation}\label{fact888}
\bbA_{M/p_k}[a]=\bbA_{M/p_jp_k}[a]=0.
\end{equation}
Let $x_k\in \ZZ_M\setminus A$ with $(x_k-a,M)=M/p_k$.
By (\ref{fact888}), we have $x_k\not\in A$, and it follows from
Lemma \ref{diagonal3} that $A_{x_k,b_0}\subset\ell_j(x_k)$. Since $M/p_j\in\Div(A)$ and $M/p_j^2\in\Div(B)$ by (\ref{fact777}), this implies that 
$$
\bbA_{M/p_j^2}[x_k]\bbB_{M/p_j^2}[b_0]=\phi(p_j^2),
$$ 
with
\begin{equation}\label{xk-sat}
\bbA_{M/p_j^2}[x_k]=p_j
\end{equation} 
and $\bbB_{M/p_j^2}[b_0]=\phi(p_j)$.  In particular, $|A\cap\ell_j(x_k)|\geq p_j$.
Taking also (\ref{fact888}) into account, we see that $|A\cap\ell_j(x_k)|= p_j$ for each $x_k$ as above. 
This together with $a*F_j\subset A$ yields $|A\cap(\Pi(a,p_i^{n_i}))|\geq p_jp_k$.
By Lemma \ref{planebound}, we must in fact have
$$
|A\cap(\Pi(a,p_i^{n_i}))|= p_jp_k > p_k^2.
$$
On the other hand, by (\ref{planebound2}) we have $|A\cap \Pi(a_i,p_i^{n_i}) |=p_k^2$. It follows that
$\Phi_{p_i^{n_i}}\nmid A$, since otherwise the number of elements of $A$ in both planes would be the same.

It remains to prove that we must also have $\Phi_{p_i^{n_i}}\nmid B$, which provides the final contradiction.
 Indeed, we use the following 
$$
\frac{1}{\phi(p_j)}\bbA_{M/p_j}[a]=\frac{1}{p_j}\bbA_{M/p_j^2}[x_k]=\frac{1}{\phi(p_k^2)}\bbA_{M/p_k^2}[a_i]=\frac{1}{\phi(p_k)}\bbA_{M/p_k}[a_i]=1,
$$
where the first and last part follow from the corner structure, the second part from (\ref{xk-sat}), 
and the third part from (\ref{fact3}). It is easy to see that this configuration implies that $\{p_i^{n_i-1}\parallel m|M\}\subset \Div(A)$. This means that for every $b\in B$
$$ 
|B\cap(\Pi(b,p_i^{n_i-1}))|=|B\cap(\Pi(b,p_i^{n_i}))|,
$$
hence $\Phi_{p_i^{n_i}}\nmid B$. 
This gives the desired contradiction and ends the proof of the proposition.
\end{proof}

Proposition \ref{N_j->N_k} implies that (C1) holds, and we can now follow the rest of the proof for that case.


\section{Fibered grids} \label{fibered-sec}


Throughout most of this section we will work under the following assumption. 

\medskip\noindent
{\bf Assumption (F):} We have $A\oplus B=\ZZ_M$, where $M=p_i^{2}p_j^{2}p_k^{2}$ is odd. Furthermore,
$|A|=|B|=p_ip_jp_k$, $\Phi_M|A$, and $A$ is fibered on $D(M)$-grids.

\medskip

Let $\cali$ be the set of elements of $A$ that belong to an $M$-fiber in the $p_i$ direction, that is, 
$$
\cali=\{a\in A\,|\,\bbA_{M/p_i}[a]=\phi(p_i)\}.
$$
The sets $\calj$ and $\calk$ are defined similarly.  
The assumption (F) implies that every element of $A$ belongs to an $M$-fiber in some direction, hence $A=\cali\cup\calj\cup\calk$. We emphasize that this does {\em not} have to be a disjoint union and that it is possible for an element of $A$ to belong to two or three of these sets.

Our main result on fibered grids is the following theorem.

\begin{theorem}\label{fibered-mainthm}
Assume that (F) holds.

\medskip\noindent
{\rm (I)}  
If $\cali\cap\calj\cap\calk\neq\emptyset$, then the tiling $A\oplus B=\ZZ_M$ is T2-equivalent to $\Lambda\oplus B=\ZZ_M$, where $\Lambda:=\Lambda(0,D(M))$. By Corollary \ref{get-standard}, both $A$ and $B$ satisfy (T2).

\medskip\noindent
{\rm (II)} Assume that $\cali\cap\calj\cap\calk=\emptyset$. Then, after a permutation of the $i,j,k$ indices if necessary, the following holds.

\begin{itemize}
\item [(II\,a)] At least one of the sets $\cali,\calj,\calk$ is empty. Without loss of the generality, we may assume that $\cali=\emptyset$, so that $A\subset \calj\cup\calk$.

\item [(II\,b)] If $A\subset \calj$ or $A\subset\calk$, then $A$ is $M$-fibered in the $p_j$ or $p_k$ direction, respectively. Consequently, 
the conditions of Theorem \ref{subtile} are satisfied in that direction. By Corollary \ref{slab-reduction}, both $A$ and $B$ satisfy (T2).

\item [(II\,c)] Suppose that $\cali=\emptyset$, and that $\calj\setminus\calk$ and $\calk\setminus\calj$ are both nonempty. 

\begin{itemize}
\item [$\bullet$]  If $\Phi_{p_i}|A$, then, after interchanging $A$ and $B$, the conditions of Theorem \ref{subtile} are satisfied in the $p_i$ direction. By Corollary \ref{slab-reduction}, both $A$ and $B$ satisfy (T2).

\item [$\bullet$]  If $\Phi_{p_i^2}|A$, then $A\subset \Pi(a,p_i)$ for any $a\in A$.
By Theorem \ref{subgroup-reduction}, both $A$ and $B$ satisfy (T2).

\end{itemize}
\end{itemize}

\end{theorem}

The proof of Theorem \ref{fibered-mainthm} is organized as follows.
We will consider the following sets of assumptions.

\medskip\noindent
{\bf Assumption (F'):} We have $A\oplus B=\ZZ_M$, where $M=p_i^{2}p_j^{2}p_k^{2}$. (Note that $M$ is not required to be odd). Furthermore,
$|A|=|B|=p_ip_jp_k$, $\Phi_M|A$, and $A$ is fibered on $D(M)$-grids.

\medskip\noindent
\textbf{Assumption (F1):} We have $A\oplus B=\ZZ_M$, where $M=p_i^{2}p_j^{2}p_k^{2}$ is odd. Furthermore, $|A|=|B|=p_ip_jp_k$, $\Phi_M|A$, $A$ is fibered on $D(M)$-grids,
and $\cali,\calj,\calk$ are pairwise disjoint.

\medskip\noindent
\textbf{Assumption (F2):} We have $A\oplus B=\ZZ_M$, where $M=p_i^{2}p_j^{2}p_k^{2}$ is odd. Furthermore, $|A|=|B|=p_ip_jp_k$, $\Phi_M|A$, $A$ is fibered on $D(M)$-grids,
$\cali\cap \calj\cap\calk =  \emptyset$, and $\calj\cap\calk\neq \emptyset$.

\medskip\noindent
\textbf{Assumption (F3):} We have $A\oplus B=\ZZ_M$, where $M=p_i^{2}p_j^{2}p_k^{2}$ is odd. Furthermore,  $|A|=|B|=p_ip_jp_k$, $\Phi_M|A$, $A$ is fibered on $D(M)$-grids, $\cali=\emptyset$, $\calj\setminus \calk\neq \emptyset$, and $\calk\setminus \calj\neq \emptyset$.

\medskip

We prove part (I) of Theorem \ref{fibered-mainthm} in Corollary \ref{triple-intersection}; in fact, this part holds under the weaker assumption (F'). Assume now that $\cali\cap\calj\cap\calk=\emptyset$. In that case, we first prove part (II\,a) of Theorem \ref{fibered-mainthm}. While the conclusion is the same, the methods of proof will be very different, so that it is preferable to split this part into two results.

\begin{proposition}\label{F1emptyset}
Assume that (F1) holds. Then one of the sets $\cali,\calj,\calk$ is empty.
\end{proposition}

\begin{proposition}\label{F2emptyset}
Assume that (F2) holds. Then $\cali=\emptyset$.
\end{proposition}

Relabeling the primes if necessary, we may assume that $\cali=\emptyset$ in the case (F1) as well. If $A$ is $M$-fibered in one of the $p_j$ or $p_k$ directions, then (II\,b) holds, and we are done. It remains to consider the case covered in (F3). The following result completes the proof of the theorem.

\begin{proposition}\label{F3structure}
Assume (F3). Then the conclusion (II\,c) of Theorem \ref{fibered-mainthm} holds.
\end{proposition}

\medskip

We briefly discuss the notation used in this section.
For $N|M$, we will use $\bbI^N$, $\bbJ^N$, $\bbK^N$ to denote the $N$-boxes associated with $\cali$, $\calj$, $\calk$.
We continue to write
$$
F_i=\{0,M/p_i,\dots,(p_i-1)M/p_i\},
$$
with $F_j,F_k$ defined similarly. Recall also that
$$
M_i=M/p_i^2,\ \ M_j=M/p_j^2,\ \ M_k= M/p_k^2.
$$
Note that $M_i=p_j^2p_k^2$ has only two distinct prime factors, and similarly for $M_j,M_k$. In particular, all $M_\nu$-cuboids are 2-dimensional for $\nu\in\{i,j,k\}$, and all conclusions of Lemma \ref{2d-cyclo} apply on that scale. Thus, if $\Phi_{M_i}|A$, then $A$ mod $M_i$ is a linear combination of $M_i$-fibers in the $p_j$ and $p_k$ directions, with non-negative integer coefficients. In particular, if $\Phi_{M_i}|A$
and
$$
\bbA^{M_i}_{M_i}[x]\in\{0,c_0\}\ \ \forall x\in\ZZ_M,
$$
then $A$ is $M_i$-fibered in one of the $p_j$ and $p_k$ directions on every $D(M_i)$-grid. Similar statements hold with $A$ replaced by $B$, as well as for other permutations of the indices $i,j,k$.

\subsection{Intersections of $\cali,\calj,\calk$}

It is possible for any of the sets $\cali, \calj$ and $\calk$ to intersect the others. Furthermore, given a $D(M)$-grid $\Lambda$, $A\cap\Lambda$ can contain fibers in some direction without necessarily being fibered in that direction. For example, consider the set
\begin{equation}\label{fifi-e1}
A_0=F_i* \big[ (F_j*a)\cup (F_k*a) \big].
\end{equation}
Then $a\in \cali\cap\calj\cap\calk$, but 
$A_0$ is not fibered in either the $p_j$ or the $p_k$ direction.

Nonetheless, the condition (F) places significant limits on the ways in which $\cali,\calj,\calk$ may intersect, as provided by the following structure lemma. In fact, the weaker condition (F') is sufficient for this purpose.

\begin{lemma}\label{fiberedstructure}
Assume (F'), and suppose that $a\in \calj\cap\calk$ for some $a\in A$. Let $D=D(M)$.

\smallskip
(i) Suppose that $A\cap \Lambda(a,D)$ is $M$-fibered in at least one of the $p_j$ and $p_k$ directions. (In particular, this holds if $a\not\in\cali$.) then 
$$A\cap\Pi(a,p_i^{2})=a*F_j*F_k.$$

\smallskip
(ii) If $A\cap\Lambda(a,D)$ is $M$-fibered in the $p_i$ direction (so that $a\in\cali\cap\calj\cap\calk$), then $A_0\subset A\cap\Lambda(a,D)$, where
$A_0$ is the set in (\ref{fifi-e1}). 
\end{lemma}

\begin{proof}
(i) Assume, without loss of generality, that $A\cap\Lambda(a,D)$ is $M$-fibered in the $p_k$ direction. Since $a\in\calj$, we have $a*F_j\subset A$, and the fibering assumption implies that $a*F_j*F_k\subset A$. The set $a*F_j*F_k$ is contained in the plane $\Pi(a,p_i^2)$ and has cardinality $p_jp_k$. By Lemma \ref{planebound}, there are no other elements of $A$ in that plane. 

If $a\not\in \cali$, then by (F') the grid $\Lambda(a, D)$ must be fibered in at least one of the other two directions, so that the above statement applies.

Part (ii) is obvious.
\end{proof}

\begin{corollary}\label{triple-intersection}
Assume (F'), and suppose that $\cali\cap\calj\cap\calk\neq\emptyset$. Then 
the tiling $A\oplus B=\ZZ_M$ is T2-equivalent to
$\Lambda(a,D(M)) \oplus B=\ZZ_M$. Consequently, $A$ and $B$ satisfy (T2).
\end{corollary}

\begin{proof}
Let $a\in \cali\cap\calj\cap\calk$. Without loss of generality, we may assume that $A\cap\Lambda(a,D(M))$ is $M$-fibered in the $p_i$ direction. It follows that $\{D(M)|m|M\}\subset \Div(A)$, and, by Lemma \ref{fiberedstructure} (ii), we have $A_0\subset A$. 

For every $a_{jk}\in A$ with $(a-a_{jk},M)=M/p_jp_k$, we have $a_{jk}*F_i\subset A$. 
Suppose now that there is a $z\in \ZZ_M\setminus A$ with $(a-z,M)=M/p_jp_k$. By Lemma \ref{flatcorner}, we have
$A_z\subset\ell_i(z)$, so that the pair $(A,B)$ has a (1,2)-cofibered structure in the $p_i$ direction with the cofiber in $A$ at distance $M/p_i^2$ from $z$. For each such $z$, we apply Lemma \ref{fibershift} to shift the cofiber, obtaining a new set $A'$ with $A'\oplus B=\ZZ_M$ such that $z*F_i\subset A'$ and $A'$ is T2-equivalent to $A$. 
After all such shifts have been performed, we see that $A$ is T2-equivalent to $\Lambda(a,D(M))$.
\end{proof}


\subsection{Toolbox for fibered grids}

We start by pointing out a special case when Theorem \ref{fibered-mainthm} is very easy to prove.

\begin{lemma}
Let $A\oplus B=\ZZ_M$, where $M=p_i^{2}p_j^{2}p_k^{2}$. If $\Phi_M$ divides both $A$ and $B$, then at least one of the sets $A$ and $B$ is $M$-fibered in some direction. By Corollary \ref{slab-reduction}, both $A$ and $B$ satisfy (T2).
\end{lemma}

\begin{proof}
Each of the differences $M/p_i$, $M/p_j$, $M/p_k$ can belong to at most one of $\Div(A)$ and $\Div(B)$. By pigeonholing, at least one of $\Div(A)$ and $\Div(B)$ must avoid at least two of these differences. 
Assume without loss of generality that $M/p_i,M/p_j\notin \Div(A)$. By the assumptions of the lemma, we also have $\Phi_M|A$. Applying Lemma \ref{gen_top_div_mis} (ii) to $A$, we get that $A$ is $M$-fibered in the $p_k$ direction, as claimed.
\end{proof}


Next, we discuss cyclotomic divisibility. 
Since $F_i(X)=(X^M-1)/(X^{M/p_i}-1)$, we have
\begin{equation}\label{fibercyc}
\Phi_s|F_i \ \Leftrightarrow   \ p_i^{2}|s|M,
\end{equation}
and similarly for $F_j$ and $F_k$. In particular,
\begin{equation}\label{fibercyc2}
\prod_{\alpha_i=1}^{2}\Phi_{M/p_i^{\alpha_i} }(X)\prod_{\alpha_j=1}^{2 }\Phi_{M/p_j^{\alpha_j}}(X)\Big|\calk(X),
\end{equation}
and similarly for other permutations of the indices $i,j,k$.

\begin{lemma}\label{fibercyc3}
Assume (F'). For each $\alpha_k\in\{1,2\}$, we have
$$\Phi_{M/p_k^{\alpha_k} }|A \ \Leftrightarrow \ \Phi_{M/p_k^{\alpha_k} }|\calk.
$$
In particular, if $\Phi_{M/p_k^{\alpha_k} }\nmid A$ for some $\alpha_k\in\{1,2\}$, then $\calk\neq\emptyset$.
Similar conclusions hold for other permutations of the indices $i,j,k$.
\end{lemma}

\begin{proof}
The lemma follows immediately from (\ref{fibercyc2}) if $\cali,\calj,\calk$ are mutually disjoint, since then we have $A(X)=\cali(X)+\calj(X)+\calk(X)$. In the general case, we need a mild workaround as follows.

Write $A(X)=\cali'(X)+\calj'(X)+\calk'(X)$, where the sets $\cali',\calj',\calk'$ are pairwise disjoint and $\cali',\calj',\calk'$ are $M$-fibered in the $p_i,p_j$, and $p_k$ direction, respectively. This can be done by
splitting up $\ZZ_M$ into pairwise disjoint $D(M)$-grids $\Lambda_\tau$ and adding $A\cap\Lambda_\tau$ to one of $\cali',\calj',\calk'$, according to the direction in which $A\cap\Lambda_\tau$ is $M$-fibered. (If $A\cap\Lambda_\tau$ is fibered in more than one direction, choose one arbitrarily and add $A\cap\Lambda_\tau$ to the corresponding set.)

It follows from (\ref{fibercyc2}) that $\Phi_{M/p_k^{\alpha_k} }|A$ if and only if $\Phi_{M/p_k^{\alpha_k} }|\calk'$. To pass from $\calk'$ to $\calk$, we write $\calk(X)=\calk'(X)+\calk_i(X)+\calk_j(X)$, where $\calk_i\subset\cali'$ and $\calk_j\subset\calj'$. By Lemma \ref{fiberedstructure}, $\calk_i$ is a union of pairwise disjoint sets of the form $a*F_i*F_k$, where $a\in A$. In particular, by (\ref{fibercyc2}) we have $\Phi_{M/p_k^{\alpha_k} }|\calk_i(X)$. Applying the same argument to $\calk_j$, we see that $\Phi_{M/p_k^{\alpha_k} }|\calk'$ if and only if $\Phi_{M/p_k^{\alpha_k} }|\calk$, and the lemma follows.
\end{proof}

We finish with two counting lemmas.

\begin{lemma}\label{Knotunfiberedunstructured}
Let $A\oplus B=\ZZ_M, M=p_i^2p_j^2p_k^2, |A|=|B|=p_ip_jp_k$. Then
\begin{equation}\label{Kdivrest}
\{m/p_k^\alpha :\ \alpha\in\{0,1,2\}, m\in \{M,M/p_i,M/p_j,M/p_ip_j\} \}\cap \Div(B)\neq	\{M\} .
\end{equation}
\end{lemma}
\begin{proof}
Suppose that (\ref{Kdivrest}) fails. It follows that any $M/p_ip_jp_k^2$-grid may contain at most one element of $B$. Since $\ZZ_M$ is a disjoint union of $p_ip_j$ such grids, it follows that $|B|\leq M/p_ip_jp_k^2$, contradicting our assumption that $|B|=M/p_ip_jp_k$.
\end{proof}

\begin{lemma}\label{planesubsetcalk}
Assume that (F') holds. 

\smallskip

(i) Let $a_k\in \calk$ and $\Pi_k:=\Pi(a_k,p_i^{\alpha_i})$ for some $\alpha_i\in\{1,2\}$.
Suppose that $|A\cap\Pi_k|=p_jp_k$ and $\cali\cap\Pi_k=\emptyset$. Then $A\cap\Pi_k\subset\calk$.

\smallskip

(ii) Let $a_i\in \cali$ and $\Pi_i:= \Pi(a_i,p_i)$.
Suppose that
\begin{equation}\label{Pi_isize}
|A\cap\Pi_i|=p_jp_k
\end{equation}
and $A\cap\Pi_i\subset\cali\cup\calk$. Then $p_i<p_j$.

\smallskip
The same conclusions hold with $j$ and $k$ interchanged.
\end{lemma}

\begin{proof}
(i) 
Suppose first that $\calj\cap\calk\cap\Pi_k \neq\emptyset$, and let $a\in \calj\cap\calk\cap\Pi_k$. Then $a\not\in\cali$. By Lemma \ref{fiberedstructure} (i), we have $A\cap\Pi_k =a*F_j*F_k$, so that $A\cap\Pi_k\subset\calj \cap\calk$.

Assume now that $\calj\cap\calk\cap\Pi_k=\emptyset$. Since $a_k\in\calk$, 
there exists a nonnegative integer $c_j$ and a positive integer $c_k$ such that 
$$
p_jp_k=|A\cap\Pi_k|=c_jp_j+c_kp_k.
$$
This clearly implies $c_j=0$ and $c_k=p_j$, thus proving $\calj\cap\Pi_k=\emptyset$. 

\medskip
(ii) Assume first
\begin{equation}\label{nointesection}
\cali\cap\calj\cap\calk\cap\Pi_i\neq\emptyset,
\end{equation}
and let  $a\in \cali\cap\calj\cap\calk\cap\Pi_i$. Denote $\Lambda_0=\Lambda(a,D(M))$. By assumption, $A\cap\Lambda_0$ is $M$-fibered in the $p_\nu$ direction for some $\nu\in \{i,j,k\}$. Suppose that the latter holds with $\nu=j$. Since $a*F_k\subset A\cap\Lambda_0$, in fact we have $a*F_k*F_j\subset A\cap\Lambda_0$. But $a*F_i$ is also contained in $A\cap\Lambda_0$, hence
$$
|A\cap\Pi_i|\geq|A\cap\Lambda_0|\geq p_jp_k+p_i-1>p_jp_k,
$$
thus contradicting (\ref{Pi_isize}). Interchanging $j$ and $k$, the same argument shows $A\cap\Lambda_0$ cannot be $M$-fibered in the $p_k$ direction. 

It follows that $A\cap\Lambda_0$ is $M$-fibered in the $p_i$ direction. Since $(a*F_j)\cup (a*F_k)\subset A\cap\Lambda_0$, the set $A\cap\Lambda_0$ must contain a structure as in (\ref{fifi-e1}). By (\ref{Pi_isize})
$$
p_jp_k = |A\cap\Pi_i|\geq |A\cap\Lambda_0| >p_ip_k,
$$
which is what we wanted to show.

For the rest of the proof we assume (\ref{nointesection}) fails. We show that $A\cap \Pi_i$ can be written (not necessarily uniquely) as a disjoint union of $M$-fibers in the $p_i$ and $p_k$ directions. Indeed, if $\cali\cap\Pi_i$ and $\calk\cap\Pi_i$ are disjoint, we are done. Assume now that 
\begin{equation}\label{IKintersection}
\cali\cap\calk\cap\Lambda\neq\emptyset,
\end{equation}
for some $D(M)$-grid $\Lambda\subset\Pi_i$. From the failure of (\ref{nointesection}), we deduce that $A\cap\Lambda$ is $M$-fibered in the $p_\nu$ direction, with either $\nu=i$ or $\nu=k$. By Lemma \ref{fiberedstructure} (i), $\cali\cap\calk\cap\Lambda$ is a union of pairwise disjoint cosets of $F_i*F_k$. We write all such cosets as unions of $M$-fibers in the $p_i$ direction. Then, we add all the remaining $M$-fibers in the $p_i$ and $p_k$ directions in $A\cap\Lambda$, disjoint from $\cali\cap\calk\cap\Lambda$ and from each other. Since $\Lambda$ is an arbitrary grid satisfying (\ref{IKintersection}), the claim follows.

Since $A\cap\Pi_i$ can be written a disjoint union of $M$-fibers in the $p_i$ and $p_k$ directions, there must exist a positive integer $c_i$ and a nonnegative integer $c_k$ such that 
$$
p_jp_k=|A\cap\Pi_i|=c_ip_i+c_kp_k.
$$
It follows that $c_i=c'_ip_k$ for some positive integer $c'_i$. But then $p_j=c'_ip_i + c_k$, so that $p_j>p_i$ as claimed.
\end{proof}


\subsection{Fibering on lower scales}


\begin{lemma}\label{fiberfibering}
Assume that (F) holds.

\smallskip

(i) Let $\Lambda:= \Lambda(a_0,D(N_i))$ for some $a_0\in\cali$. If $\cali\cap\calj\cap\Lambda=\emptyset$ and $A$ is $N_i$-fibered on $\Lambda$, it cannot be fibered in the $p_j$ direction.

\smallskip

(ii) If $\Phi_{N_i}|B$ and
\begin{equation}\label{kwartet}
	\{M/p_j,M/p_k,M/p_ip_j,M/p_ip_k\}\cap\Div(B)=\emptyset,
\end{equation}
then $B$ must be $N_i$-fibered in the $p_i$ direction. Note in particular that if $\{D(M)|m|M\}\subset\Div(A)$, then (\ref{kwartet}) holds.
\end{lemma}

\begin{proof}
(i) Assume, by contradiction, that $A$ is $N_i$-fibered in the $p_j$ direction on $\Lambda$. 
Then $\bbA^{N_i}_{N_i}[a]=p_i$ and $\bbA^{N_i}_{N_i/p_j}[a]=p_i \phi(p_j)$ for all $a\in\cali\cap\Lambda$, meaning that $\cali\cap\Lambda$ is also $M$-fibered in the $p_j$ direction. This contradicts the assumption that $\cali\cap\calj\cap\Lambda$ is empty.

\medskip

(ii) By (\ref{kwartet}), we have $\bbB^{N_i}_{N_i/p_j}[b]=\bbB^{N_i}_{N_i/p_k}[b]=0$ for all $b\in B$. It follows from Lemma \ref{gen_top_div_mis} (ii) that $B$ is $N_i$-fibered in the $p_i$ direction.
\end{proof}

By Lemma \ref{fibercyc3}, if $\Phi_{N_k}|A$, then $\Phi_{N_k}|\calk$. 
We consider the question of whether, in these circumstances, $\calk$ is permitted to have an unfibered grid on a lower scale.

\begin{lemma}\label{Kunfibered}
Assume that (F) holds. Suppose that $\Phi_{N_k}|A$ and that there exists a $D(N_k)$-grid on which $\calk$ is not fibered. Then:
\begin{itemize}
\item $\{D(M)|m|M\}\cup\{M/p_k^2,M/p_ip_jp_k^2\} \subset \Div(A)$,
\item there exists an $x\in \ZZ_M$ such that $\mathbb{K}^{N_k}_{N_k/p_k}[x]=\phi(p_k^2)$.
\end{itemize}
\end{lemma}

\begin{proof}
Assume that $\Phi_{N_k}|A$. By Lemma \ref{fibercyc3}, we have
 $\Phi_{N_k}|\calk$. Consider $\calk$ as a multiset in $\ZZ_{N_k}$ with constant multiplicity $p_k$. We claim that if $\calk$ is not fibered on a $D(N_k)$-grid $\Lambda$, it must satisfy the conclusion of either Lemma \ref{misscorner_fullplane} or Lemma \ref{oddcornerplus} with $N=N_k$ on that grid. Indeed, if this is not the case, then we must have 
$$
\{m:\ D(N_k)|m|N_k \}\subset \Div_{N_k}(\calk).
$$
But then, for every $D(N_k)|m|N_k$ there exist $a,a'\in\calk$ (depending on $m$) such that $(a-a',N_k)=m$. Since $\bbK^{N_k}_{N_k}[a]=\bbK^{N_k}_{N_k}[a']=\phi(p_k)$, we have 
$$
\big\{m/p_k^\alpha :\ \alpha\in\{0,1,2\}, m\in \{M,M/p_i,M/p_j,M/p_ip_j\} \big\}\subset \Div(\calk).
$$
The latter contradicts Lemma \ref{Knotunfiberedunstructured}.

On the other hand, suppose that $\calk\in\calm(\ZZ_{N_k})$ has at least $p_j+1$ distinct points in some plane $\Pi(x,p_i^2)$ in $\ZZ_{N_k}$, each of multiplicity $p_k$. Then $|A\cap \Pi(x,p_i^2)|\geq (p_j+1)p_k$ in $\ZZ_M$,
contradicting Lemma \ref{planebound}. The same argument applies with $i$ and $j$ interchanged.

Among the structures described in Lemmas \ref{misscorner_fullplane} and \ref{oddcornerplus} with $N=N_k$, the only ones that avoid configurations as in the last paragraph are as follows.

\begin{itemize}
\item $\calk\cap\Lambda$ has the $p_k$ full plane structure as in Lemma \ref{misscorner_fullplane}, so that for some $x\in\ZZ_M$ we have $\mathbb{K}^{N_k}_{N_k/p_k}[x]=\phi(p_k^2)$ and $\mathbb{K}^{N_k}_{N_k/p_ip_j}[x]=p_k\phi(p_ip_j)$.

\item $\calk\cap\Lambda $ has a $p_i$ or $p_j$ corner structure as in Lemma \ref{oddcornerplus} (i).

\item $\calk\cap\Lambda $ has a $p_i$ or $p_j$ almost corner structure as in Lemma \ref{oddcornerplus} (ii), so that (possibly after a permutation of $i$ and $j$) there exist $x_1, x_2,x_3,x_4\in \ZZ_M$ such that $(x_\nu-x_{\nu'},N_k)=N_k/p_i$ for $\nu\neq\nu'$, 
$\mathbb{K}^{N_k}_{N_k/p_k}[x_1]=\mathbb{K}^{N_k}_{N_k/p_k}[x_2]=\phi(p_k^2)$ and $\mathbb{K}^{N_k}_{N_k/p_j}[x_3]=\mathbb{K}^{N_k}_{N_k/p_j}[x_4]=p_k\phi(p_j)$.

\end{itemize}

We address the second case, the first and third case being similar.
Indeed, a $p_i$ corner structure in $\calk$ on the $N_k$ scale means that there exist $a,a'\in\calk$ with $(a-a',N_k)=N_k/p_i$ so that 
\begin{equation}\label{latka}
\mathbb{K}^{N_k}_{N_k/p_k}[a]= \mathbb{K}_{M/p_k^2}[a]=\phi(p_k^2)
\end{equation} 
and 
\begin{equation}\label{latka2}
\mathbb{K}^{N_k}_{N_k/p_j}[a']= \mathbb{K}_{M/p_j}[a'] + \mathbb{K}_{M/p_jp_k}[a']
=p_k \phi(p_j).
\end{equation} 
Now consider $\calk$ on scale $M$.
We have $(a-a',M)\in\{M/p_i,M/p_ip_k\}$, with the fiber chain in (\ref{latka}) attached to $a$.
By (\ref{latka2}) and the fact that $a'\in \calk$, we also have $a'*F_j*F_k\subset\calk$. Hence the conclusions of the lemma hold with $x=a$.
\end{proof}

\begin{corollary}\label{subset}
Assume (F). If $\Phi_{N_k}|A$ and $p_k>\min_\nu p_\nu$, then $\calk$ is $N_k$-fibered on each $D(N_k)$-grid in one of the $p_i$ and $p_j$ directions. In particular, $\calk\subset\cali\cup\calj$.
\end{corollary}
\begin{proof}
Assume without loss of generality that $p_i=\min_\nu p_\nu$. By Lemma \ref{fibercyc3}, 
we have $\Phi_{N_k}|A$ if and only if $\Phi_{N_k}|\calk$. By Lemma \ref{Kunfibered}, if there exists a $D(N_k)$-grid on which $\calk$ is unfibered, then there must exist $x\in \ZZ_M$ such that $\mathbb{K}^{N_k}_{N_k/p_k}[x]=\phi(p_k^2)$, thus 
\begin{align*}
|A\cap \Pi(x,p_j^2)|& \geq	\mathbb{K}^{N_k}_{N_k/p_k}[x] =\phi(p_k^2)\\
& >p_ip_k
\end{align*}
which contradicts Lemma \ref{planebound}. Thus $\calk$ must be fibered on all $D(N_k)$-grids. On the other hand, by the same argument as above, $\calk$ cannot be $N_k$-fibered in the $p_k$ direction on any $D(N_k)$-grid.
\end{proof}


\subsection{Proof of Proposition \ref{F1emptyset}}


In this section, we are assuming (F1), which we state here again for the reader's convenience.

\medskip\noindent
\textbf{Assumption (F1).} We have $A\oplus B=\ZZ_M$, where $M=p_i^{2}p_j^{2}p_k^{2}$ is odd. Furthermore, $|A|=|B|=p_ip_jp_k$, $\Phi_M|A$, $A$ is fibered on $D(M)$-grids,
and
\begin{equation}\label{pairwise-disjoint}
\cali,\calj,\calk \hbox{ are pairwise disjoint.}
\end{equation}

\medskip

We must prove that at least one of the sets $\cali, \calj$ or $\calk$ has to be empty.
To this end we assume the contrary, i.e.,
\begin{equation}\label{nonempty}
\cali,\calj,\calk\neq\emptyset,
\end{equation}
and prove by contradiction that (\ref{nonempty}) cannot hold.
We may assume, without loss of generality, that 
\begin{equation}\label{maleduze}
p_i<p_j<p_k.
\end{equation}
By Corollary \ref{subset}, this implies that $\Phi_{N_j}$ and $\Phi_{N_k}$ cannot divide $A$, so that 
\begin{equation}\label{BPhi_{N_j}Phi_{N_k}}
\Phi_{N_j}\Phi_{N_k}|B.
\end{equation}

We start with a cyclotomic divisibility result.
\begin{proposition}\label{initialstatement}
Assume that (F1), (\ref{nonempty}), and (\ref{BPhi_{N_j}Phi_{N_k}}) hold. Then $\Phi_{N_i}\nmid A$.
\end{proposition}

\begin{proof}
The proof is divided into several steps. In each of the following claims, the assumptions of the proposition are assumed to hold.

\medskip
\noindent
{\bf Claim 1.} {\it If $\Phi_{N_i}|A$, then $\cali$ is $N_i$-fibered in the $p_i$ direction, so that
for every $a_i\in\cali$,
\begin{equation}\label{califiber2}
\mathbb{I}_M[a_i]+\mathbb{I}_{M/p_i}[a_i]+\mathbb{I}_{M/p_i^2}[a_i]=p_i^2.
\end{equation}
} 
	
\begin{proof}
It suffices to prove that $\cali$ is $N_i$-fibered on each $D(N_i)$-grid. Once we know that, Lemma \ref{fiberfibering} (i) together with (\ref{pairwise-disjoint}) implies that the $N_i$-fibering must be in the $p_i$ direction, and the claim follows.

Assume, by contradiction, that there exists a $D(N_i)$-grid over which $\cali$ is not fibered. By Lemma \ref{Kunfibered}, 
\begin{equation}\label{Inotfibereddiv}
\{D(M)|m|M\}\cup\{M/p_i^2,M/p_i^2p_jp_k\}\subset \Div(A)
\end{equation}
and there exists $x_0\in\ZZ_M$ with 
\begin{equation}\label{ifiberchain}
\mathbb{I}^{N_i}_{N_i/p_i}[x_0]=\phi(p_i^2).
\end{equation}
The proof of Lemma \ref{Kunfibered} implies further that $\cali\cap\Lambda(x_0,D(N_i))$ must contain one of the structures described in Lemmas \ref{misscorner_fullplane} and \ref{oddcornerplus} with $N=N_i$. Additionally,  (\ref{BPhi_{N_j}Phi_{N_k}}) and Lemma \ref{fiberfibering} (ii) imply that $B$ is $N_j$-fibered in the $p_j$ direction and $N_k$-fibered in the $p_k$ direction, hence
\begin{equation}\label{Bfiberdiv} 
M/p_j^2,M/p_k^2,M/p_j^2p_k^2\in \Div(B).
\end{equation}

We claim that 
\begin{equation}\label{voidatthevertex}
\mathbb{I}^{N_i}_{N_i}[x_0]=0.
\end{equation}
Suppose this is not true, then 
\begin{equation}\label{califiber}
\mathbb{A}_{M}[x_0]+\mathbb{A}_{M/p_i}[x_0]+\mathbb{A}_{M/p_i^2}[x_0]=p_i^2.
\end{equation}
		
Let $a_j\in\calj$ and $a_k\in\calk$. Recall from Lemma \ref{fibershift} that the fibering in $B$ allows one to shift the fibers rooted at $a_j$ and $a_k$ by distance $M/p_j^2$ and $M/p_k^2$, respectively. Let $x_j,x_k\in\ZZ_M$ with
$$
(a_\nu-x_\nu,M)=M/p_\nu^2\ \text{ and }\  p_\nu^2|x_0-x_\nu\ \text{ for }\ \nu\in\{j,k\}.
$$
Let $A'$ be the set obtained by shifting the fibers rooted at $a_j$ and $a_k$ to $x_j$ and $x_k$ respectively, so that $x_j*F_j,x_k*F_k\subset A'$. By Lemma \ref{fibershift}, $A'\oplus B=\ZZ_M$ and $A'$ is T2-equivalent to $A$. 
		
Let $a_1,a_2\in \ell_i(x_0)$ be the points with $p_i^2|a_1-x_j$, $p_i^2|a_2-x_k$.  By (\ref{califiber}), $\ell_i(x_0)\subset A\cap A'$. It follows from (\ref{Bfiberdiv}) that we cannot have $M/p_k^2\in\Div(A')$, hence 
$(a_1-x_j,M)=M/p_k$. Similarly, $(a_2-x_k,M)=M/p_j$.
But now $M/p_i^2p_j, M/p_i^2p_k\in \Div(A')$, hence $M/p_i^2p_j, M/p_i^2p_k\notin \Div(B)$. Together with (\ref{Inotfibereddiv}), this contradicts Lemma \ref{Knotunfiberedunstructured}. This proves (\ref{voidatthevertex}).
		
We therefore conclude that $\cali\cap\Lambda(x_0,D(N_i))$ has either a full plane structure (Lemma \ref{misscorner_fullplane}) or an almost corner structure (Lemma \ref{oddcornerplus} (ii)). 
In either one of these cases, there exists a point 
\begin{equation}\label{caliLambda}
a_i^*\in \cali\cap\Lambda(x_0,D(M))
\end{equation}
such that for each $\nu\in\{j,k\}$, $A\cap (a_i^* *F_i*F_\nu)$ is $M$-fibered in the $p_i$ direction, but is not $M$-fibered in the $p_\nu$ direction.

Let $a_j\in\calj$, and let $x_j\in\ell_j(a_j)$ be the point such that $p_j^2|x_0-x_j$. We consider two cases.
\begin{itemize}
\item[(a)] If it is possible to choose $a_j$ so that 
$x_j\in a_j*F_j$, we fix that choice, and let $A':=A$.

\item[(b)] Otherwise, let $A'$ be the set obtained from $A$ by shifting the fiber $a_j*F_j$ to $x_j$ if necessary, so that $x_j*F_j \subset A'$. By Lemma \ref{fibershift}, $A'\oplus B=\ZZ_M$ and $A'$ is T2-equivalent to $A$. 
Since we are not in case (a), 
$A'\cap(x_j*F_i*F_j)$ contains no other $M$-fibers in the $p_j$ direction.
\end{itemize}

We show that either $x_j\in \Lambda(x_0,D(M))$ or 
\begin{equation}\label{M/p_i^2p_knotinB}
M/p_i^2p_k\notin \Div(B).
\end{equation} 
\begin{itemize}
\item Suppose that $p_i|x_0-x_j$. If $(x_0-x_j,M)\in\{M, M/p_i, M/p_k, M/p_ip_k\}$, then clearly $x_j\in \Lambda(x_0,D(M))$. If on the other hand $(x_0-x_j,M)\in\{M/p_k^2, M/p_ip_k^2\}$, 
then by (\ref{caliLambda}), there exists a fiber $a*F_i\subset \cali$ with $a\in  \Lambda(x_0,D(M))$ and $p_i^2|a-x_j$. 
Then $M/p_k^2\in \Div(a,x_j*F_j)\subset\Div(A')$, contradicting (\ref{Bfiberdiv}).

\item Assume now that $p_i\nmid x_0-x_j$. If $(x_0-x_j,M)=M/p_i^2p_k$, then by (\ref{ifiberchain}) together with the fact that $p_i\geq 3$, there must be an $a\in A$ with $(a-x_0,M)=M/p_i^2$ and $(a-x_j,M)=M/p_i^2p_k$, proving (\ref{M/p_i^2p_knotinB}) since $x_j\in A'$. Otherwise, we have $(x_0-x_j,M)=M/p_i^2p_k^2$, but then by (\ref{ifiberchain}), there must be $a\in A$ with $(a-x_0,M)=M/p_i^2$ and $(a-x_j,M)=M/p_k^2$, contradicting (\ref{Bfiberdiv}).
\end{itemize} 

Suppose now that $x_j\in \Lambda(x_0,D(M))$. We claim that, in this case, $A'$ contains a $p_k$ extended corner structure
consisting of the fiber in $x_j*F_i* F_j$ in the $p_j$ direction and
at least one fiber in $a_i^* *F_i*F_j$ in the $p_i$ direction. Indeed, by
(\ref{caliLambda}) and (\ref{pairwise-disjoint}), we have
$$
\Lambda(x_j,D(M))\cap (\calj\cup \calk)=\emptyset.
$$
In particular, we must be in case (b) above, and $A\cap(x_j*F_i*F_j)=x_j*F_j$.
This, together with the choice of $a_i^*$, proves that the conditions of Definition \ref{corner} (ii) hold.

In that case, however, we proved in Theorem \ref{cornerthm} that $A'$ (therefore $A$) is T2-equivalent to a $D(M)$-grid. It follows that $A$ satisfies (T2), therefore $A\oplus B^\flat=\ZZ_M$, where $\Phi_{p_\nu}|B^\flat$ for all $\nu\in\{i,j,k\}$. This contradicts (\ref{Inotfibereddiv}), since clearly $M/p_i^2\in \Div(B^\flat)$.
		
We are therefore left with (\ref{M/p_i^2p_knotinB}). By the same argument with $p_j$ and $p_k$ interchanged, we
must also have $M/p_i^2p_j\notin \Div(B)$. Together with (\ref{Inotfibereddiv}), this again contradicts Lemma \ref{Knotunfiberedunstructured}.
\end{proof}

\medskip
\noindent
{\bf Claim 2.} {\it If $\Phi_{N_i}|A$, then $\Phi_{M_i}\nmid A$, and therefore $\Phi_{M_i}|B$.} 
\begin{proof}
By Lemma \ref{fibercyc3},  if $\Phi_{M_i}|A$, then $\Phi_{M_i}|\cali$. Since $M_i$ has only 2 distinct prime divisors, 
we may apply Lemma \ref{2d-cyclo} on the scale $M_i$. We conclude that every element of $\cali$ belongs to either an $M_i$-fiber in the $p_j$ direction, in which case we have $|A\cap\Pi(a_i,p_k^2)|\geq p_i^2p_j$, or to an $M_i$-fiber in the $p_k$ direction, in which case we have $|A\cap\Pi(a_i,p_j^2)|\geq p_i^2p_k$. Since both bounds contradict Lemma \ref{planebound}, we deduce $\Phi_{M_i}|B$, and the lemma follows.
\end{proof}
	
The next claim is a direct consequence of Claim 2, (\ref{califiber2}),  and Lemma \ref{2d-cyclo}.
	
\medskip
\noindent
{\bf Claim 3.} {\it If $\Phi_{N_i}|A$, then $\bbB^{M_i}_{M_i}[y]\in\{0,1\}$ for all $y\in\ZZ_M$. Moreover, for every $b\in B$, either 
\begin{equation}\label{BM/p_i^2fiberp_jeq}
\bbB^{M_i}_{M_i/p_j}[b']=\phi(p_j) \text{ for all } b'\in B\cap\Lambda(b,D(M_i)),
\end{equation}
and, since $p_i=\min_\nu p_\nu$,
\begin{equation}\label{Bfiberdiv2}
M/p_ip_j,M/p_i^2p_j\in \Div(B), 
\end{equation}
or 
\begin{equation}\label{BM/p_i^2fiberp_keq}
\bbB^{M_i}_{M_i/p_k}[b']=\phi(p_k) \text{ for all } b'\in B\cap\Lambda(b,D(M_i))
\end{equation}
and 
\begin{equation}\label{Bfiberdiv3}
M/p_ip_k,M/p_i^2p_k\in \Div(B).
\end{equation}
}

\medskip
\noindent
{\bf Claim 4.} {\it If $\Phi_{N_i}|A$, then for all $a_i\in\cali$ we have 
$\mathbb{I}^{M_i}_{M_i/p_j}[a_i]=\mathbb{I}^{M_i}_{M_i/p_k}[a_i]=0$.}
\begin{proof}
If there exist $b_1,b_2\in B$ such that (\ref{BM/p_i^2fiberp_jeq}) holds with $b'=b_1$ and (\ref{BM/p_i^2fiberp_keq}) holds with $b'=b_2$, then the claim follows from (\ref{Bfiberdiv2}) and (\ref{Bfiberdiv3}). Assume therefore that (\ref{BM/p_i^2fiberp_jeq}) holds for all $b\in B$, and, consequently, $\mathbb{I}^{M_i}_{M_i/p_j}[a_i]=0$ for all $a_i\in\cali$.

Assume, by contradiction, that $\mathbb{I}^{M_i}_{M_i/p_k}[a_i]>0$ for some $a_i\in\cali$. 
It follows from (\ref{califiber2}) that 
\begin{equation}\label{Afibereddiv}
M/p_i, M/p_k,M/p_i^2,M/p_ip_k,M/p_i^2p_k\in \Div(A).
\end{equation}

We claim that
\begin{equation}\label{pjpk-niedzielnik}
M/p_jp_k\notin \Div(B).
\end{equation}
Assuming this, we prove Claim 4 as follows. By (\ref{nonempty}), (\ref{Afibereddiv}), and
(\ref{pjpk-niedzielnik}), we have $N_k/p_i,N_k/p_j\not\in\Div_{N_k}(B)$. 
By (\ref{BPhi_{N_j}Phi_{N_k}}) and Lemma \ref{gen_top_div_mis} (ii), B must be $N_k$-fibered in the $p_k$ direction, so that
$$
\bbB^{N_k}_{N_k/p_k}[b]=\phi(p_k)\ \ \forall b\in B.
$$

Let $a_k\in \calk$. Since $a_i$ satisfies (\ref{califiber2}), we may assume (moving $a_i$ to a different point in the same fiber chain if necessary) that 
\begin{equation}\label{a_ia_ksameplane}
p_i^2|a_k-a_i.
\end{equation}
Then the pair $(A,B)$ has a (1,2)-cofibered structure, with $a_k*F_k$ as a cofiber. 
Let $x'_k\in\ell_k(a_k)$ be the point such that $p_k^2|a_i-x'_k$. If $p_k|a_k-x'_k$, we note that $x'_k\in A$, and let $A':=A$. If on the other hand $(a_k-x'_k,M)=M/p_k^2$, we use
Lemma \ref{fibershift} to shift $a_k*F_k$ to $x'_k$, obtaining a new set $A'$ such that $x'_k\in A'$, $A'\oplus B=\ZZ_M$, and $A'$ is T2-equivalent to $A$.

By (\ref{Bfiberdiv2}) and (\ref{a_ia_ksameplane}), we  must have $(a_i-x'_k,M)=M/p_j^2$, so that $M/p_j^2\in \Div(A')$. In particular, $M/p_j^2\notin \Div(B)$. On the other hand, we also have $\Phi_{N_j}|B$, and
$$
\bbB^{N_j}_{N_j/p_k}[b]=\bbB^{N_j}_{N_j/p_j}[b]=0 \text{ for all } b\in B.
$$
By Lemma \ref{gen_top_div_mis}, B must be $N_j$-fibered in the $p_i$ direction, so that
$\bbB^{N_j}_{N_j/p_i}[b]=\phi(p_i)$ for every $b\in B$. Together with (\ref{nonempty}), this means that every grid $\Lambda(b,M/p_ip_j)$ with $b\in B$ contains exactly $p_i$ points of $B$.
On the other hand, the assumption that (\ref{BM/p_i^2fiberp_jeq}) holds for all $b\in B$ implies that 
every such grid contains exactly $p_j$ points of $B$. This contradiction proves the claim, assuming (\ref{pjpk-niedzielnik}).
		
We now prove (\ref{pjpk-niedzielnik}). Assume, by contradiction, that $b,b'\in B$ with $(b-b',M)=M/p_jp_k$. Let $y,y'\in \ZZ_M\setminus B$ with $(b-y,M)=(b'-y',M)=M/p_k$, $(b-y',M)=(b'-y,M)=M/p_j$, and consider the saturating set $B_{y,a_i}$.  Applying Corollary \ref{bispan-corollary} (i) to $B_y$, once with respect to $b$ and again with respect to $b'$, we get
$$
B_{y,a_i}\subset\ell_i(b)\cup\ell_i(y)\cup\ell_i(b')\cup\ell_i(y').
$$
If $B_{y,a_i}\cap(\ell_i(y)\cup \ell_i(y'))$ is nonempty, then $\{M/p_ip_k,M/p_i^2p_k\}\cap \Div(B)$ must be nonempty, while if $B_{y,a_i}\cap(\ell_i(b)\cup\ell_i(b'))$ is nonempty, then $\{M/p_i,M/p_i^2\}\cap \Div(B)$ must be nonempty. Both of these contradict (\ref{Afibereddiv}).
\end{proof}

\medskip
\noindent
{\bf Claim 5.} {\it If $\Phi_{N_i}|A$ and $M/p_jp_k\notin \Div(B)$, then there must exist $b_j,b_k\in B$ such that (\ref{BM/p_i^2fiberp_jeq}) holds with $b=b_j$ and (\ref{BM/p_i^2fiberp_keq}) holds with $b=b_k$. }

\begin{proof}
Assume, by contradiction, that the conclusion is false. Without loss of generality, we may assume that (\ref{BM/p_i^2fiberp_jeq}) holds for all $b\in B$.

Since $M/p_jp_k\notin \Div(B)$, we have $\bbB^{N_j}_{N_j/p_k}[b]=0$ for all $b\in B$. By (\ref{BPhi_{N_j}Phi_{N_k}}) and Lemma \ref{gen_top_div_mis}, B must be fibered on $D(N_j)$-grids, so that for every $b\in B$ either 
\begin{equation}\label{BN_jp_jfibered}
	\bbB^{N_j}_{N_j/p_j}[b]=\phi(p_j),
\end{equation}
or 
\begin{equation}\label{BN_jp_ifibered}
	\bbB^{N_j}_{N_j/p_i}[b]=\bbB^M_{M/p_ip_j}[b]=\phi(p_i).
\end{equation}
Suppose that there exists $b_0\in B$ satisfying (\ref{BN_jp_jfibered}). Applying (\ref{BM/p_i^2fiberp_jeq}) to all $b'\in B$ with $(b_0-b',M)=M/p_j^2$, we get 
$$
|B\cap\Pi(b_0,p_k^2)|\geq p_j^2>p_ip_j,
$$
which contradicts Lemma \ref{planebound}. 

Hence (\ref{BN_jp_ifibered}) must hold for all $b\in B$, so that $B$ is a union of disjoint $N_j$-fibers in the $p_i$ direction, each of cardinality $p_i$. On the other hand, by Claim 3 and (\ref{BM/p_i^2fiberp_jeq}), for any $b\in B$ we have
\begin{align*}
	\bbB^{N_j/p_i^2}_{N_j/p_i^2}[b]
	&= \bbB^{M_i/p_j}_{M_i/p_j}[b]\\
	& =1+\bbB^{M_i}_{M_i/p_j}[b]\\
	& =1+ \phi(p_j) = p_j .
\end{align*}
This implies that $p_j$ is divisible by $p_i$, which is obviously false.
\end{proof}

\medskip
\noindent
{\bf Claim 6.} {\it If $M/p_jp_k\notin \Div(B)$, then $\Phi_{N_i}\nmid A$.}
	
\begin{proof}
Assume, by contradiction, that $M/p_jp_k\notin \Div(B)$
and $\Phi_{N_i}|A$. Then the conclusions of Claims 1-5 apply. By Claim 5, we may find $b_j,b_k$ such that (\ref{BM/p_i^2fiberp_jeq}) holds with $b=b_j$ and (\ref{BM/p_i^2fiberp_keq}) holds with $b=b_k$. We fix these elements for the duration of the proof.

We claim that
\begin{equation}\label{Bfiberdiv4}
M/p_j^2,M/p_k^2\in \Div(B).
\end{equation}
We prove the first part of (\ref{Bfiberdiv4}), the second part being identical with $p_j$ and $p_k$ interchanged. 

As in the proof of Claim 5, we use (\ref{BPhi_{N_j}Phi_{N_k}}) and the assumption that
$M/p_jp_k\notin \Div(B)$ to conclude that either (\ref{BN_jp_jfibered}) or (\ref{BN_jp_ifibered}) holds
for every $b\in B$.
Suppose that (\ref{BN_jp_ifibered}) holds for $b=b_j$. Then $B\cap \Lambda(b_j,D(N_j))$ is a union of disjoint $N_j$-fibers in the $p_i$ direction, and the same argument as in the proof of Claim 5 
implies that $p_j$ is divisible by $p_i$, which is obviously false. Thus (\ref{BN_jp_jfibered}) holds with $b=b_j$, and in particular $M/p_j^2\in \Div(B)$.

With (\ref{Bfiberdiv4}) in place, we complete the proof as follows. 
Fix $a_j\in \calj$ and $a_i\in \cali$ such that $p_i^2|a_j-a_i$. (This is possible by (\ref{califiber2}).) 
Taking (\ref{Bfiberdiv2}), (\ref{Bfiberdiv3}) and (\ref{Bfiberdiv4}) into account, we see that $$
(a_i-a_j,M)\in\{ M/p_j^2p_k^2, M/p_j^2p_k\}.
$$ 
Suppose first that 
$(a_i-a_j,M)=M/p_j^2p_k^2= p_i^2$. By (\ref{califiber2}) again, we must in fact have $1,p_i,p_i^2\in\Div(A)$. 
We deduce that $M_i/p_j^2p_k^2\not\in\Div_{M_i}(B)$, and, therefore, 
one of $B\subset b* p_j\ZZ$ or $B\subset b* p_k\ZZ$ must hold for any fixed $b\in B$. That, however, contradicts (\ref{Bfiberdiv4}).

We are therefore left with 
\begin{equation}\label{a_ia_jdiff}
(a_i-a_j,M)=M/p_j^2p_k.
\end{equation}
We claim that in this case $\Phi_{M_j}\nmid A$. Indeed, otherwise $\Phi_{M_j}|A$ and, by Lemma \ref{fibercyc3}, we have $\Phi_{M_j}|\calj$. Notice that by (\ref{Bfiberdiv4}), $\bbJ^{M_j}_{M_j}[x]\in \{0,p_j\}$ for all $x\in \ZZ_M$. By Lemma \ref{2d-cyclo}, $a_j$ belongs to an $M_j$-fiber in the $p_\nu$ direction, with either $\nu=i$ or $\nu=k$. Assume first that $a_j$ belongs to an $M_j$-fiber in the $p_k$ direction. This means that for every $y\in \ZZ_M$ with $(y-a_j,M)=M/p_k$ we have $\bbJ^{M_j}_{M_j}[y]=p_j$, with $\calj\cap\ell_j(y)$ consisting of a single $M$-fiber in the $p_j$ direction. By (\ref{a_ia_jdiff}), there must be $y_0$ with $(y_0-a_j,M)=M/p_k$, such that $a_i\in \ell_j(y_0)$. Let $a_j'\in \calj\cap\ell_j(y_0)$. 
Let also $a'_i\in\ell_i(a_i)\subset A$ be the element such that $p_i^2|a'_i-a'_j$. We also have $p_k^2|a'_i-a'_j$, and, by (\ref{a_ia_jdiff}), $p_j\nmid a'_i-a'_j$, Hence $(a'_i-a_j',M)=M/p_j^2$, which contradicts (\ref{Bfiberdiv4}).

It follows that $a_j$ must belong to an $M_j$-fiber in the $p_i$ direction. By the same argument as above, for every $y\in\ZZ_M$ with $(y-a_j,M)=M/p_i$, we have $\bbJ^{M_j}_{M_j}[y]=p_j$. We get $|A\cap \Pi(a_j,p_k^2)|\geq p_ip_j$, and by Lemma \ref{planebound}, we have
\begin{equation}\label{fullp_k^2plane}
|A\cap \Pi(a_j,p_k^2)|=p_ip_j.
\end{equation}
Therefore
\begin{align*}
|A\cap \Pi(a_j,p_k)|&\geq |A\cap \Pi(a_j,p_k^2)|+|A\cap \Pi(a_i,p_k^2)|\\
&>p_ip_j,
\end{align*}
where the latter inequality follows from the fact that  $\ell_i(a_i)\subset A$, as follows from Claim 1.
By Corollary \ref{planegrid} we have $\Phi_{p_k^2}|A$. Combining the latter with (\ref{fullp_k^2plane}), we conclude
$$
|A\cap \Pi(a_i,p_k^2)|=|A\cap \Pi(a_j,p_k^2)|=p_ip_j>|\ell_i(a_i)|=p_i^2.
$$
This means that there must be at least one element $a\in A\cap \Pi(a_i,p_k^2)$ which does not belong to $\ell_i(a_i)$. The latter is not possible, because $\hbox{Div}(a,\ell_i(a_i))\subset \hbox{Div}(A)$ would cause a divisor conflict with either (\ref{Bfiberdiv2}) or (\ref{Bfiberdiv4}). This contradiction gives $\Phi_{M_j}|B$. 

Note that by (\ref{nonempty}) $M/p_j\notin \hbox{Div}(B)$, thus $B'=\{b \mod  N_j:\ b\in B\}\subset \ZZ_{N_j}$ is a set (and not a multiset). By (\ref{BPhi_{N_j}Phi_{N_k}}), we have $\Phi_{N_j}\Phi_{M_j}|B$ and therefore $\Phi_{N_j}\Phi_{M_j}|B'$. By Example (1) at the end of Section 2.6 applied with $M$ replaced by $N_j$, $B'$ is $\calt$-null with respect to the cuboid type $\calt=(N_j,(1,0,1),1)$. Since each individual cuboid of type $\calt$ is in fact contained in a 2-dimensional $D(M_j)$-grid in $\ZZ_{N_j}$, this implies that $\Phi_{M_j}|(B'\cap\Lambda)(X)$, for any such grid $\Lambda$. By Lemma \ref{2d-cyclo} (ii), $B'\cap\Lambda$ is $N_j$-fibered in either the $p_i$ or $p_k$ direction on each such grid. 
However, since $M/p_j, M/p_jp_k\notin \hbox{Div}(B)$, we get $N_j/p_k\notin \hbox{Div}_{N_j}(B')$. Hence $B'$, and therefore $B$, is $N_j$-fibered in the $p_i$ direction. As already proved in Claim 5, the latter is not allowed. The claim follows.
\end{proof}
	
Claim 6 proves Proposition \ref{initialstatement} in the case when $M/p_jp_k\notin \Div(B)$. From here on, we will therefore assume that 
\begin{equation}\label{M/p_jp_kas}
M/p_jp_k\in \Div(B).
\end{equation}
	
\medskip
\noindent
{\bf Claim 7.} {\it If $\Phi_{N_i}|A$ and (\ref{M/p_jp_kas}) holds, then $\Phi_{M_i/p_j}\Phi_{M_i/p_k}|B$.}

\begin{proof}
	
Assume, by contradiction, that $\Phi_{M_i/p_j}| A$. Let $a_i\in\cali$, and let $x_k\in \ZZ_M$ with $(a_i-x_k,M)=M/p_k$.
By Claim 4, we have $\bbA^{M_i}_{M_i}[x_k]=0$. Furthermore, (\ref{califiber2}) and the assumption that 
$M/p_jp_k\notin \Div(A)$ imply that
$$
\bbA^{M_i/p_j}_{M_i/p_j}[x_k]=\bbA^{M_i}_{M_i}[x_k]+\sum_{x_k':(x_k-x_k',M)=M/p_j} \bbA^{M_i}_{M_i}[x_k']=0.
$$
Considering all $M_i/p_j$ cuboids with vertices at $a_i$ and $x_k$, we see that for every $x_j\in\ZZ_M$ with $(a_i-x_j,M)=M/p_j^2$ we have 
$$
\bbA^{M_i/p_j}_{M_i/p_j}[x_j]\geq p_i^2,
$$
thus
$$
|A\cap\Pi(a_i,p_k^2)|\geq p_jp_i^2>p_ip_j,
$$
contradicting Lemma \ref{planebound}. Since this argument is symmetric with respect to $j$ and $k$, the claim follows.
\end{proof}
	
\medskip
\noindent
{\bf Claim 8.} {\it If $\Phi_{N_i}|A$ and (\ref{M/p_jp_kas}) holds, then $B$ is $M_i$-fibered in both of the $p_j$ and $p_k$ directions, so that
 for all $b\in B$ we have
\begin{equation}\label{BM/p_i^2fibering}
\frac{1}{\phi(p_j)}\bbB^{M_i}_{M_i/p_j}[b]=\frac{1}{\phi(p_k)}\bbB^{M_i}_{M_i/p_k}[b]=1.
\end{equation}}
	
\begin{proof}
By Claims 2 and 7,
we have $\Phi_{M_i}\Phi_{M_i/p_k}|B$. Therefore $B$ is null with respect to all cuboids of type $(M_i,\vec{\delta},1)$, where $\vec{\delta}=(0,1,\delta_k)$ with $\delta_k\in\{1,2\}$.
Suppose that for some $b_0\in B$ we have $\bbB^{M_i}_{M_i/p_j}[b_0]< \phi(p_j)$. Fix $y_j\in \ZZ_M$ such that $(b_0-y_j,M)=M/p_j$ and $\bbB^{M_i}_{M_i}[y_j]=0$, and consider all cuboids as above with vertices at $b_0$ and $y_j$. In order to balance these cuboids, we must have $\bbB^{M_i}_{M_i}[y]=1$ for all $y\in \ZZ_M$ with $(b_0-y,M)\in\{M/p_k,M/p_k^2\}$. But now we get $|B\cap\Pi(b_0,p_j^2)|\geq p_k^2>p_ip_k$, which contradicts Lemma \ref{planebound}. Since this argument is symmetric in $j$ and $k$, the claim follows.
\end{proof}
	
\medskip
\noindent
{\bf Claim 9.} {\it If $\Phi_{N_i}|A$ and (\ref{M/p_jp_kas}) holds, then  for all $b\in B$ we have $\bbB^{N_j}_{N_j/p_i}[b]=\phi(p_i)$.}
	
\begin{proof}
Assume, by contradiction, that there exists $b_0\in B$ with $\bbB^{N_j}_{N_j/p_i}[b_0]<\phi(p_i)$. Since $p_k>p_j$ and $M/p_k\in \Div(A)$, we must also have $\bbB^{N_j}_{N_j/p_k}[b_0]<\phi(p_k)$. Let $y_i,y_k\in \ZZ_M$ with $(b_0-y_i,N_j)=N_j/p_i, (b_0-y_k,N_j)=N_j/p_k$ and such that $\bbB^{N_j}_{N_j}[y_i]=\bbB^{N_j}_{N_j}[y_k]=0$. 
Recall that $\Phi_{N_j}|B$, and consider all $N_j$ cuboids with vertices at $b,y_i$ and $y_k$. We get that 
for every $z$ with $(z-b_0,N_j)=N_j/p_j$, we have
\begin{equation}\label{Njpjplanes}
\bbB^{N_j}_{N_j}[z]+\bbB^{N_j}_{N_j/p_ip_k}[z]\geq 1.
\end{equation}
But by Claim 8, we also have 
$$
|B\cap \Pi(b,p_j)|\geq p_jp_k \ \ \forall b\in B.
$$
Applying this to all $b\in B$ contributing to (\ref{Njpjplanes}), and summing over $z$ with $(z-b_0,N_j)=N_j/p_j$,
we get $|B|\geq p_j\cdot (p_jp_k)>p_ip_jp_k$, a contradiction. 
\end{proof}

\medskip
\noindent
{\bf Claim 10.} {\it If (\ref{M/p_jp_kas}) holds, then $\Phi_{N_i}\nmid A$.}
	
\begin{proof}
Assume, for contradiction, that (\ref{M/p_jp_kas}) holds and $\Phi_{N_i} | A$.
By Claim 8, $B$ must satisfy (\ref{BM/p_i^2fibering}), and in particular
$$
\bbB^{M_i}_{M_i/p_j}[b]=\phi(p_j) \text{ for all } b\in B.
$$ 
By Claim 9,
$$
\bbB^{N_j}_{N_j/p_i}[b]=\bbB_{M/p_ip_j}[b]=\phi(p_i) \text{ for all } b\in B.
$$
As in the proof of Claim 6, these two properties imply that $p_j$ is divisible by $p_i$, which is false.
\end{proof}

Claim 10 concludes the proof of Proposition \ref{initialstatement}.	
\end{proof}


By Proposition \ref{initialstatement}, it remains to prove Proposition \ref{F1emptyset} under the assumption that
\begin{equation}\label{BPhi_{N_i}Phi_{N_j}Phi_{N_k}}
\prod_{\nu\in\{i,j,k\}} \Phi_{M/p_\nu}|B.
\end{equation} 

\begin{lemma}
Assume (F1) and (\ref{BPhi_{N_i}Phi_{N_j}Phi_{N_k}}).
Then at least one of the sets $\cali, \calj$, or $\calk$ must be empty.
\end{lemma}

\begin{proof}
Assume, by contradiction, that (\ref{nonempty}) holds. By (\ref{BPhi_{N_i}Phi_{N_j}Phi_{N_k}}), $\Phi_{N_i}|B$. Since $p_i=\min_\nu p_\nu$ and $M/p_j,M/p_k\notin \Div(B)$, we must have 
$$
\frac{1}{\phi(p_j)}\bbB^{N_i}_{N_i/p_j}[b], \frac{1}{\phi(p_k)}\bbB^{N_i}_{N_i/p_k}[b]<1 \text{ for all } b\in B,
$$
and, in particular, $B$ cannot be $N_i$-fibered in either the $p_j$ or the $p_k$ direction on any $D(N_i)$-grid. By Lemma \ref{gen_top_div_mis}, it follows that $M/p_i^2\in \Div(B)$.

Suppose first that $\Phi_{M_i}|A$, so that $\Phi_{M_i}|\cali$. We have $\bbI^{M_i}_{M_i}[x]\in \{0,p_i\}$ for all $x\in\ZZ_M$, 
and $N_i/p_i\not\in\Div_{N_i}(A)$. By Lemma \ref{2d-cyclo}, $\cali$ must be $M_i$-fibered in one of the $p_j$ and $p_k$ directions, on each $D(M_i)$-grid. In particular, for a given $a_i\in \cali$, we have either
\begin{equation}\label{IM/p_i^2fibering}
\mathbb{I}^{M_i}_{M_i}[y]=p_i, \text{ for all } y\in\ZZ_M \text{ with } (y-a_i,M)=M/p_j,
\end{equation}
or
\begin{equation}\label{IM/p_i^2fiberingp_k}
\mathbb{I}^{M_i}_{M_i}[y]=p_i, \text{ for all } y\in\ZZ_M \text{ with } (y-a_i,M)=M/p_k.
\end{equation}

Assume that there exists $a_i\in\cali$ satisfying (\ref{IM/p_i^2fibering}), and fix such an element. It follows that $|\cali\cap\Pi(a_i,p_k^2)|\geq p_ip_j$. By Lemma \ref{planebound}, the last inequality holds as equality, i.e.,
\begin{equation}\label{maxp_k^2plane}
|\cali\cap\Pi(a_i,p_k^2)|= p_ip_j \text{ and so } A\cap\Pi(a_i,p_k^2)\subset \cali.
\end{equation}
We now consider two cases.
\begin{itemize}
\item If $\Phi_{p_k^2}|A$, then (\ref{maxp_k^2plane}) implies $A\subset \Pi(a_i,p_k)$. But since $\calk$ is nonempty, we may find $a_k\in \calk\subset \Pi(a_i,p_k)$. It follows that there must exist $a_k'\in (a_k*F_k)\cap\Pi(a_i,p_k^2)$. By (\ref{maxp_k^2plane}) we have $a_k'\in \cali$, thus contradicting (\ref{pairwise-disjoint}). 
\item If $\Phi_{p_k}|A$, then 
\begin{equation}\label{maxplaneallA}
|A\cap\Pi(a,p_k)|= p_ip_j \text{ for all } a\in A.
\end{equation}
Evidently, the latter implies that (\ref{IM/p_i^2fibering}) must hold for all $a_i'\in \cali$. Indeed, otherwise there would be $a_i'\in\cali$ satisfying (\ref{IM/p_i^2fiberingp_k}), and so $|A\cap\Pi(a_i',p_k)|\geq p_ip_k$. But now (\ref{maleduze}) implies $|A\cap\Pi(a_i',p_k)|>p_ip_j$, which contradicts (\ref{maxplaneallA}). Hence (\ref{IM/p_i^2fibering}) holds for all $a_i\in \cali$. Consequently (\ref{maxp_k^2plane}) holds for all $a_i\in \cali$. 

Let $a_k\in\calk$. If $\cali\cap \Pi(a_k,p_k)$ is nonempty, then (\ref{maxp_k^2plane}) and (\ref{maxplaneallA}) imply that $A\cap\Pi(a_k,p_k)$ is contained in $\cali$. In particular we would have $a_k\in \cali$, thus contradicting (\ref{pairwise-disjoint}). Hence $A\cap \Pi(a_k,p_k)\subset \calj\cup\calk$.
Lemma \ref{planesubsetcalk} (ii) implies now that $p_i>p_k$, contradicting (\ref{maleduze}).
\end{itemize}
We note that the complementary case, in which (\ref{IM/p_i^2fiberingp_k}) holds for all $a_i\in \cali$, is proved in the exact same way, interchanging $j$ and $k$. The only difference between the two cases is that in the analogous subcase $\Phi_{p_j}|A$, we do not require the argument showing (\ref{IM/p_i^2fiberingp_k}) holds for all $a_i\in \cali$.

We conclude that  $\Phi_{M_i}\nmid A$, hence $\Phi_{N_i}\Phi_{M_i}|B$. This implies that $B$ is $\calt$-null with respect to the cuboid type $\calt=(N_i,\vec{\delta},T)$, where $\vec{\delta}=(0,1,1)$ and $T(X)=1$. (See the first of the two examples at the end of Section \ref{cuboid-section}, with $M$ replaced by $N_i$.)
 Since $M/p_i\notin \Div(B)$, we have $\bbB^{\calt}[y]\in\{0,1\}$ for all $y\in \ZZ_M$. Note that $\calt$ is a 2-dimensional cuboid type, so that for every $b\in B$ we either have 
$$
\bbB^{\calt}[y_j]=\bbB_M[y_j]+\bbB_{M/p_i}[y_j]=1 \text{ for all } y_j\in \ZZ_M \text{ with } (b-y_j,M)=M/p_j,
$$
or
$$
\bbB^{\calt}[y_k]=\bbB_M[y_k]+\bbB_{M/p_i}[y_k]=1 \text{ for all } y_k\in \ZZ_M \text{ with } (b-y_k,M)=M/p_k.
$$
As $p_i=\min_\nu p_\nu$, the former implies $M/p_j\in \Div(B)$, and the latter implies $M/p_k\in \Div(B)$, both contradicting (\ref{nonempty}). 
\end{proof}



\subsection{Proof of Proposition \ref{F2emptyset}}


In this section, we will prove that $\cali=\emptyset$ under the following conditions.

\medskip\noindent
\textbf{Assumption (F2).} We have $A\oplus B=\ZZ_M$, where $M=p_i^{2}p_j^{2}p_k^{2}$ is odd. Furthermore, $|A|=|B|=p_ip_jp_k$, $\Phi_M|A$, $A$ is fibered on $D(M)$-grids, $\calj\cap\calk\neq \emptyset$, and
\begin{equation}\label{pairwise-disjoint-2}
\cali\cap \calj\cap\calk =  \emptyset.
\end{equation}

\medskip

Let $a\in \calj\cap\calk$. 
By Lemma \ref{fiberedstructure} (i), we have 
\begin{equation}\label{Aplanegrid}
a*F_j*F_k=A\cap\Pi(a,p_i^{2}),
\end{equation}
and in particular
\begin{equation}\label{Adivintersec}
\{M/p_j,M/p_k,M/p_jp_k\}\subset \Div(A).
\end{equation}

We first prove Proposition \ref{F2emptyset} under the assumption that $\Phi_{p_i^{2}}|A$.

\begin{lemma}\label{subgroup}
Assume (F2), and let $a\in \calj\cap\calk$. If $\Phi_{p_i^{2}}|A$, then 
\begin{equation}\label{Asubgroup}
A\subset\Pi(a,p_i).
\end{equation}
\end{lemma}

\begin{proof}
The assumption $\Phi_{p_i^2}|A$, together with (\ref{Aplanegrid}), implies that $|A\cap\Pi(a,p_i)|=p_ip_jp_k=|A|$. 
\end{proof}

\begin{corollary}
Assume that (F2) holds, and that $\Phi_{p_i^{2}}|A$. Then $\cali=\emptyset$.
\end{corollary}
\begin{proof}
Let $a\in\calj\cap\calk$. Suppose, by contradiction, that $\cali$ is nonempty. By (\ref{Asubgroup}), there must exist an element $a_i\in \cali\cap \Pi(a,p_i^2)$. It follows from (\ref{Aplanegrid}) that $a_i\in a*F_j*F_k$. But then $a_i\in\cali\cap\calj\cap\calk$, contradicting (\ref{pairwise-disjoint-2}).
\end{proof}

In the rest of this section, it remains to consider the following case.

\medskip\noindent
\textbf{Assumption (F2').} Assume that (F2) holds, and that 
\begin{equation}\label{notthehighest}
\Phi_{p_i^2} \nmid  A.
\end{equation}

\begin{lemma}\label{saturatingplanes}
Assume (F2'). Then:
\begin{enumerate}
\item [(i)]  
We have $\Phi_{p_i}|A$. Moreover, 
\begin{equation}\label{podzial1}
A\cap\Pi(a,p_i)=A\cap\Pi(a,p_i^{2})\ \ \forall a\in \calj\cap\calk,
\end{equation}	
\begin{equation}\label{podzial2}
|A\cap\Pi(x',p_i)|=p_jp_k\ \ \forall x'\in \ZZ_M.
\end{equation}

\item [(ii)] Let $a\in \calj\cap\calk$.
For every $x\in \ZZ_M\setminus A$ with $(a-x,M)=M/p_i$, we have
\begin{equation}\label{asaturatingplane}
A_x\subset \Pi(a,p_i^{2}).
\end{equation}
\end{enumerate}
\end{lemma}
\begin{proof}
Since $|A|=p_ip_jp_k$, exactly one of $\Phi_{p_i}$ and $\Phi_{p_i^2}$ must divide $A$.
By (\ref{notthehighest}), we must in fact have $\Phi_{p_i}|A$.
Let $a\in\calj\cap \calk$. 
Then 
\begin{equation}\label{podzial10}
p_jp_k\geq |A\cap\Pi(a,p_i)|\geq |A\cap\Pi(a,p_i^{2})|=p_jp_k,
\end{equation}
by (\ref{Aplanegrid}). Hence the above must hold with equality, which proves (\ref{podzial1}). It also follows that for any $x'\in\ZZ_M$, we must have $|A\cap\Pi(x',p_i)|=p_jp_k$, so that (\ref{podzial2}) holds. This proves (i).

Next, let $a,x$ be as in (ii). By (\ref{podzial10}) again, we have
\begin{equation}\label{Ap_iplanebound}
A\cap\Pi(a,p_i)\subset \Pi(a,p_i^2).
\end{equation}
Let $b\in B$. Applying Corollary \ref{bispan-corollary} (i) to $A_{x,b}$, from (\ref{two-planes}) we get $A_{x ,b}\subset \Pi(a,p_i^{2})\cup \Pi(x,p_i^{2})$. Now (\ref{asaturatingplane}) follows directly from (\ref{Ap_iplanebound}).
\end{proof}

\begin{lemma}\label{planeBsaturating}
Assume (F2'). Then:
\begin{enumerate}
\item [(i)]  For every $b\in B$, and for every $y\in \ZZ_M$ with $(b-y,M)=M/p_i$, we have
\begin{equation}\label{Bplanerest}
\bbB_{M}[y]+\bbB_{M/p_j}[y]+\bbB_{M/p_k}[y]+\bbB_{M/p_jp_k}[y]=1.
\end{equation}
	
\item [(ii)] For all $d$ with $p_i^{2}|d|(M/p_jp_k)$, we have $\Phi_d|B$. Additionally,
\begin{equation}\label{Btopdiv}
\{M/p_i,M/p_ip_j,M/p_ip_k,M/p_ip_jp_k\}\cap \Div(B)\neq \emptyset .
\end{equation}		
\end{enumerate}
\end{lemma}
\begin{proof}
Fix $b\in B$, and write $N_{jk}=M/p_jp_k$ for short. 
Let $a\in \calj\cap\calk$, and let $x\in \ZZ_M\setminus A$ with $(a-x,M)=M/p_i$. 

By (\ref{e-ortho2}), (\ref{Aplanegrid}), and (\ref{asaturatingplane}), we have
\begin{align*}
1&=
\frac{1}{\phi(p_i)}\bbA_{M/p_i}[x|\Pi(a,p_i^{2})]\,\bbB_{M/p_i}[b]+\frac{1}{\phi(p_ip_j)}\bbA_{M/p_ip_j}[x|\Pi(a,p_i^{2})]\,\bbB_{M/p_ip_j}[b]\\
&\qquad+\frac{1}{\phi(p_ip_k)}\bbA_{M/p_ip_k}[x|\Pi(a,p_i^{2})]\,\bbB_{M/p_ip_k}[b]
+\frac{1}{\phi(p_ip_jp_k)}\bbA_{M/p_ip_jp_k}[x|\Pi(a,p_i^{2})]\,\bbB_{M/p_ip_jp_k}[b]\\
&=\frac{1}{\phi(p_i)}(\bbB_{M/p_i}[b]+\bbB_{M/p_ip_j}[b]+\bbB_{M/p_ip_k}[b]+\bbB_{M/p_ip_jp_k}[b])\\
&=\frac{1}{\phi(p_i)}\sum_{y:(y-b,M)=M/p_i}\bbB_{N_{jk}}^{N_{jk}}[y].
\end{align*}	
On the other hand, by (\ref{Adivintersec}) and divisor exclusion, any $N_{jk}$-grid may contain at most one element of $B$, so that $\bbB_{N_{jk}}^{N_{jk}}[y]\leq 1$ for all $y\in\ZZ_M$. Therefore
$$
\bbB_{N_{jk}}^{N_{jk}}[y]= 1 \text{ for all } y \text{ such that } (y-b,M)=M/p_i,
$$
which is (\ref{Bplanerest}). We also note that the first equation in the above calculation implies (\ref{Btopdiv}).

Finally, let $d\neq p_i^{2}$ and $p_i^{2}|d|(M/p_jp_k)$, and consider any $d$-cuboid with one vertex at $a\in \calj\cap\calk$. By (\ref{Aplanegrid}), we have $\bbA^d_d[a]\geq p_jp_k$. However, we also have $|A\cap\Pi(a,p_i)|=p_jp_k$ by the third claim in the lemma, so that $\bbA^d_{d'}[a]=0$ for all $d'<d$ with $D(d)|d'$. It follows that the cuboid cannot be balanced. Therefore	$\Phi_d\nmid A$, which proves (ii). 
\end{proof}

\begin{lemma}\label{nodoubleinters}
Assume (F2'). Then $\cali\cap(\calj\cup\calk)=\emptyset$.
\end{lemma}

\begin{proof}
Suppose that the conclusion fails. Without loss of generality, we may assume that 
$\cali\cap\calk\neq\emptyset$. Taking into account that $\calj\cap\calk\neq \emptyset$, we have from (\ref{Adivintersec})
$$
M/p_i,M/p_j,M/p_k,M/p_ip_k,M/p_jp_k\in \Div(A)
$$
so that (\ref{Btopdiv}) reduces to $\{M/p_ip_j,M/p_ip_jp_k\}\cap \Div(B)\neq\emptyset$, and (\ref{Bplanerest}) to
$$
\bbB_{M/p_j}[y]+\bbB_{M/p_jp_k}[y]=1 \text{ for all } y \text{ with } (y-b,M)=M/p_i.
$$ 
This means we must have 
\begin{equation}\label{smallracoon}
p_i<p_j,
\end{equation} 
otherwise one must introduce $M/p_i$ or $M/p_ip_k$ as differences in $B$. 
	
We now repeat the same procedure with $i$ and $j$ interchanged. Let $a'\in \cali\cap\calk$. We note that $A\cap \Lambda(a',D(M))$ cannot be $M$-fibered in the $p_j$ direction, since that would contradict (\ref{pairwise-disjoint-2}).
It follows from Lemma \ref{fiberedstructure} (i) that $a'*F_i*F_k\subset A$, and, together with Lemma \ref{planebound}, this implies that
\begin{equation}\label{smallracoon2}
a'*F_i*F_k= A\cap\Pi(a',p_j^{2}).
\end{equation} 
Let $x'\in \ZZ_M\setminus A$ satisfy $(a'-x',M)=M/p_j$. As in the proof of Lemma \ref{saturatingplanes}, by Corollary \ref{bispan-corollary} (i)
we have
$$A_{x'}\subset\Pi(a',p_j^{2})\cup \Pi(x',p_j^{2}).
$$
Assume, by contradiction, that there exists a $b\in B$ such that $A_{x' ,b}\subset \Pi(a',p_j^{2})$. Repeating the proof of Lemma \ref{planeBsaturating} with that $b$, we get the analogues of (\ref{Bplanerest}) and 
(\ref{Btopdiv}) with $i$ and $j$ interchanged.
The same argument as in the proof of (\ref{smallracoon}) shows then that $p_j<p_i$, a contradiction.

It follows that $A_{x' }\cap \Pi(x',p_j^{2})\neq\emptyset$, and, in particular, 
$|A\cap \Pi(a',p_j)|>p_ip_k$. By Corollary \ref{planegrid}, 
$A\subset \Pi(a',p_j)$. In particular, 
$\calj\subset \Pi(a',p_j)$, with each fiber in $\calj$ containing a point in $A\cap \Pi(a',p_j^{2})$.
But by (\ref{smallracoon2}), any such point would belong to $\cali\cap\calj\cap\calk$, contradicting (\ref{pairwise-disjoint-2}).
\end{proof}

\begin{lemma}\label{noPhi_Ni}
Assume (F2'). If $\cali\neq\emptyset$, then $\Phi_{N_i}\nmid A$. 
\end{lemma}

\begin{proof}
By Lemma \ref{fibercyc3}, it suffices to prove that $\Phi_{N_i}\nmid\cali$. Let $a_i\in \cali$. By Corollary \ref{nodoubleinters}, we have $a_i\not\in\calj$ and $a_i\not\in\calk$, so that there must exist $x_j, x_k\in \ZZ_M\setminus A$ with $(a_i-x_j,N_i)=N_i/p_j, (a_i-x_k,N_i)=N_i/p_k$ and such that $\mathbb{I}^{N_i}_{N_i}[x_j]=\mathbb{I}^{N_i}_{N_i}[x_k]=0$. 
	
Consider the $N_i$ cuboid with one face containing vertices at $a_i, x_j$ and $x_k$, and the other face in $\Pi(a,p_i)$, where $a\in\calj\cap\calk$. In order for this cuboid to be balanced, $\cali\cap\Pi(a,p_i)$ must be nonempty,
and in particular $\cali\cap\Pi(a,p_i^{2})\neq \emptyset$. But this together with (\ref{Aplanegrid}) contradicts
(\ref{pairwise-disjoint-2}).
\end{proof}

\begin{lemma}\label{emptycali}
Assume (F2'). If $p_i=\min_\nu p_\nu$, then $\cali=\emptyset$.
\end{lemma}

\begin{proof}
Assume by contradiction that $\cali\neq\emptyset$. By Lemma \ref{noPhi_Ni}, $\Phi_{N_i}\nmid A$, hence $\Phi_{N_i}|B$. By assumption we have $M/p_i\in\Div(A)$, so that $\bbB^{N_i}_{N_i}[y]\in \{0,1\}$ for all $y\in \ZZ_M$. By (\ref{notthehighest}), $\Phi_{p_i^2}|B$. 
Hence $\Phi_{p_i}\nmid B$, and in particular $B$ cannot be $N_i$-fibered in the $p_i$ direction. It follows that there must exist $b_0\in B$ and $y\in \ZZ_M$ with $(b_0-y,N_i)=N_i/p_i$ and $\bbB^{N_i}_{N_i}[y]=0$. 
	
In order to simplify notation, we shall denote $\beta_\nu=\bbB^{N_i}_{N_i/p_\nu}[b_0]$ for $\nu\in\{j,k\}$. By (\ref{Bplanerest}) we must have 
\begin{equation}\label{Bplanegridbnd}
\beta_j+\beta_k+1\leq |B\cap\Lambda(b_0,D(M))| \leq p_i
\end{equation}
 thus 
\begin{equation}\label{avgbound}
\beta_\nu\leq (p_i-1)/2 \text{ for some } \nu\in \{j,k\}.
\end{equation}
In addition, considering all $N_i$ cuboids with vertices at $b_0$ and $y$ such that the vertices at distance $N_i/p_j$ and $N_i/p_k$ from $b_0$ do not belong to $B$, we see that
$$
\bbB^{N_i}_{N_i/p_jp_k}[y]\geq (p_j-\beta_j-1)(p_k-\beta_k-1).
$$

Now, if $\beta_j=1$ then $\beta_k\leq p_i-2$ and so 
\begin{align*}
(p_j-\beta_j-1)(p_k-\beta_k-1)& \geq (p_j-2)(p_k-p_i+1)\\
&\geq (p_k-p_i+1)p_i\\
&>p_i
\end{align*}
which contradicts (\ref{Bplanerest}). We may therefore assume $\beta_\nu\geq2$ for $\nu=j,k$. In this case, however, assuming (\ref{avgbound}) for $\nu=k$, applying (\ref{Bplanegridbnd}) and the fact that $p_j-p_i\geq2$, we have
\begin{align*}
(p_j-\beta_j-1)(p_k-\beta_k-1)& = (p_j-p_i+p_i-\beta_j-1)(p_k-\beta_k-1)\\
&\geq (p_j-p_i+\beta_k)(p_k-\beta_k-1)\\
&\geq 4(p_i-(p_i-1)/2-1)\\
&= 2p_i-2
\end{align*}
The latter exceeds $p_i$ whenever $p_i>2$. Since $M$ is odd, again we get a contradiction to (\ref{Bplanerest}) and the lemma follows.
\end{proof}

Lemma \ref{emptycali} proves Proposition \ref{F2emptyset}, assuming that (F2') holds and that $p_i$ is the smallest prime. From now on, we will therefore assume that
\begin{equation}\label{pinotthesmallest}
p_i>\min_\nu p_\nu.
\end{equation}
The rest of the proof will be split into the following cases:

\begin{itemize}
\item Assume (F2'),  (\ref{pinotthesmallest}), and 
$\Phi_{N_j}\Phi_{N_k}|A$. This case is addressed in Lemma  \ref{p_kmincalksubcalj} and Corollary \ref{p_kmin}.

\item Assume (F2'), (\ref{pinotthesmallest}), and (interchanging $j$ and $k$ if necessary)
$\Phi_{N_j}\nmid A$, $ \Phi_{N_k}\nmid B$.
This case is addressed in Lemmas  \ref{calinonempty}, \ref{Phi_N_jnmidAPhi_N_knmidB},
and  \ref{lastoneformixedcase}.

\item Assume (F2'),  (\ref{pinotthesmallest}), and 
$\Phi_{N_j}\Phi_{N_k}|B$. This case is addressed in Lemma \ref{calknointer}, Corollary \ref{p_nuorder}, and Lemma \ref{p_ismallest}.

\end{itemize}

\begin{lemma}\label{p_kmincalksubcalj}
Assume (F2') and (\ref{pinotthesmallest}).
If $p_k<p_i<p_j$ and $\Phi_{N_k}|A$, then $\calk\subset\calj$.
\end{lemma}
\begin{proof}
We first claim that $\calk$ is $N_k$-fibered on each $D(N_k)$-grid in one of the $p_j$ and $p_k$ directions.
Indeed, if $\calk$ were not $N_k$-fibered on some $D(N_k)$-grid, then it would follow from Lemma \ref{Kunfibered} that $\{D(M)|m|M\}\subset\Div(A)$; however, that is not compatible with (\ref{Btopdiv}). 
Furthermore, by Lemma \ref{nodoubleinters} and Lemma \ref{fiberfibering} (i), 
$\calk$ cannot be $N_k$-fibered in the $p_i$ direction on any $D(N_k)$-grid. This proves the claim.

Recall from Lemma \ref{saturatingplanes} (i) that for any $a_0\in A$,
\begin{equation}\label{tosamo}
	|A\cap \Pi(a_0,p_i)|=p_jp_k.
\end{equation}
If $\calk$ is $N_k$-fibered in the $p_j$ direction on some $D(N_k)$-grid $\Lambda_j$,
then for every $a_k\in\calk\cap\Lambda_j $ we have
\begin{equation}\label{KJinter}
	\mathbb{K}^{N_k}_{N_k/p_j}[a_k]=p_k\phi(p_j) ,
\end{equation}
so that $a_k\in \calj\cap\calk$. By (\ref{tosamo}),
$A\cap \Pi(a_k,p_i)= a_k* F_j * F_k$ is fibered in both directions, and in particular $\Pi(a_k,p_i)$ contains no elements of $A$ outside of $\Lambda_j$. 

Assume now that there exists $a'_k\in\calk$ such that $\calk$ is $N_k$-fibered in the $p_k$ direction on $\Lambda_k:=\Lambda(a'_k,D(N_k))$.
Then 
\begin{equation}\label{fullkline}
	\mathbb{K}^{N_k}_{N_k/p_k}[a']=p_k\phi(p_k) \text{ for all } a'\in\calk\cap\Pi(a'_k,p_i),
\end{equation}
so that
\begin{equation}\label{fullklinecount}
	\bbA_M[a']+\bbA_{M/p_k}[a']+\bbA_{M/p_k^2}[a']=p_k^2\text{ for all } a'\in\calk\cap\Pi(a'_k,p_i).		
\end{equation}

Fix $a'_k\in\calk$ satisfying (\ref{fullkline}) and (\ref{fullklinecount}). We first claim that 
\begin{equation}\label{nicniema}
\cali\cap\Pi(a'_k,p_i)=\emptyset.
\end{equation}
Indeed, suppose that (\ref{nicniema}) fails, and let $a_i\in\cali\cap\Pi(a'_k,p_i^2)$. Since $\calk\cap \Lambda_k$ is $N_k$-fibered in the $p_k$ direction, by Lemma \ref{nodoubleinters}
we must have $a_i\notin\Lambda_k$, so that $a_i$ must be at distance $M/p_j^2$ from the fiber chain in the $p_k$ direction rooted at $a'_k$. We can now extend (\ref{Adivintersec}) to 
\begin{equation}\label{Ap_kdivlist}
	M/p_j,M/p_k,M/p_jp_k,M/p_j^2,M/p_k^2,M/p_j^2p_k^2,M/p_j^2p_k\in \Div(A),
\end{equation}
where all the differences that do not appear in (\ref{Adivintersec}), come from the interaction between $a_i$ and $\ell_k(a'_k)$. 

By (\ref{Ap_kdivlist}), we see that for every $b\in B$ 
\begin{align*}
	|B\cap\Pi(b,p_i^2)|&\leq 1+\bbB_{M/p_jp_k^2}[b]\\
	&\leq p_k	.
\end{align*}
But, since $\ZZ_M$ has only $p_i^2$ residue classes modulo $p_i^2$, it follows that
$$
p_ip_jp_k=|B|\leq p_i^2p_k,
$$
so that $p_j\leq p_i$, contradicting the assumption that (\ref{nicniema}) fails.

It, therefore, follows that $A\cap \Pi(a'_k,p_i)\subset \calj\cup\calk$. By (\ref{tosamo}) and Lemma \ref{planesubsetcalk} (i), we have $A\cap \Pi(a'_k,p_i)\subset \calk$. But since all $a'\in\calk\cap\Pi(a'_k,p_i)$ satisfy (\ref{fullkline}) and (\ref{fullklinecount}), we get 
$$
p_jp_k=cp_k^2
$$
for some positive integer $c$. The latter implies $p_k$ divides $p_j$, which is not allowed. This completes the proof of the lemma.
\end{proof}

\begin{corollary}\label{p_kmin}
Assume (F2') and (\ref{pinotthesmallest}). Assume further that
$\Phi_{N_j}\Phi_{N_k}|A$ and that $p_k<p_j$.
Then $\cali=\emptyset$, and $A$ is $M$-fibered in the $p_k$ direction. 
\end{corollary}

\begin{proof}
By (\ref{pinotthesmallest}), we have $p_k=\min_\nu p_\nu$.
We first apply Corollary \ref{subset}, with $j$ and $k$ interchanged. Since $\Phi_{N_j}|A$ and $p_j\neq \min_\nu p_\nu$,
we get that $\calj\subset\cali\cup\calk$. However, by Lemma \ref{nodoubleinters} we have $\cali\cap\calj=\emptyset$, so that we must in fact have $\calj\subset\calk$.

Assume, by contradiction, that $\cali$ is nonempty. We first prove that this implies 
\begin{equation}\label{p_jp_i}
p_i<p_j.
\end{equation}
Let $a_i\in \cali$. Observe that $\Pi(a_i,p_i)$ cannot contain any elements $a'\in\calj$, since any such element would be associated with a grid $a' * F_j * F_k\subset A$. This together with Lemma \ref{nodoubleinters} would imply 
$$
|A\cap\Pi(a_i,p_i)|\geq |a_i*F_i|+|a' * F_j * F_k|>p_jp_k,
$$
contradicting (\ref{tosamo}). We get $\calj\cap\Pi(a_i,p_i)=\emptyset$. 
Hence $A\cap \Pi(a_i,p_i)\subset \cali\cup\calk$, and (\ref{p_jp_i}) follows from Lemma \ref{planesubsetcalk} (ii).

Applying Lemma \ref{p_kmincalksubcalj}, we see that $\calk\subset\calj$. This, together with the first part of the proof, implies $\calk=\calj$. Hence any element $a_k\in\mathcal{K}$ is associated with a grid $a_k * F_j * F_k\subset A$, and as shown above, such grids cannot intersect $\Pi(a_i,p_i)$. Therefore $A\cap\Pi(a_i,p_i)$ must be contained in $\cali$. That, however, is clearly false since $p_jp_k=|A\cap\Pi(a_i,p_i)|$ cannot be a multiple of $p_i$. This contradiction concludes the proof. 
\end{proof}

Before we move on to the next two cases, we need two lemmas on the fibering properties of $B$.

\begin{lemma}\label{BfiberedAfibered}
Assume (F2). If $\Phi_{N_k}|B$, then $B$ is $N_k$-fibered on each $D(N_k)$-grid, either in the $p_k$ direction or in the $p_i$ direction. The same is true with $j$ and $k$ interchanged. 
\end{lemma}
\begin{proof}
Suppose that $\Phi_{N_k}|B$. By (\ref{Adivintersec}), we have $\bbB^{N_k}_{N_k/p_j}[b]=0$ and $\bbB^{N_k}_{N_k}[b]=1$ for all $b\in B$. Since $M$ is odd, the lemma follows from Lemma \ref{gen_top_div_mis}.
\end{proof}

\begin{lemma}\label{BnoN_ip_jfiber}
Assume (F2'). If $\cali\neq\emptyset$ and $B$ is $N_i$-fibered on a $D(N_i)$-grid, then it must be fibered in the $p_i$ direction on that grid.
\end{lemma}
\begin{proof}
We argue by contradiction. Let $\Lambda_0:=\Lambda(b,D(N_i))$ for some $b\in B$, and assume that $B\cap \Lambda_0$ is $N_i$-fibered in one of the other directions, say $p_j$. Let also $\Lambda:=\Lambda(b,D(M))$. 
By (\ref{Bplanerest}), we have
\begin{equation}\label{card-pi}
|B\cap\Lambda|=p_i.
\end{equation}
On the other hand, the $N_i$-fibering assumption means that 
$B\cap\Lambda_0$ can be divided into mutually disjoint $N_i$-fibers in the $p_j$ direction, each one of cardinality $p_j$, and each one either entirely contained in $\Lambda$ or disjoint from it. This implies that $p_j$ divides $|B\cap\Lambda|$. That, however, contradicts (\ref{card-pi}). 
\end{proof}

Next, we consider the case $\Phi_{N_j}\nmid A, \Phi_{N_k}\nmid B$. This case will be split further, according to whether $p_j$ or $p_k$ is the smallest prime.

\begin{lemma}\label{calinonempty}
Assume (F2') and (\ref{pinotthesmallest}).
Assume further that $\Phi_{N_j}\nmid A$, $\Phi_{N_k}\nmid B$, and 
$p_j=\min_\nu p_\nu$. Then $\calk\subseteq\calj$. In addition, if $\cali\neq \emptyset$, then:
\begin{itemize}
\item $B$ is $N_j$-fibered in the $p_j$ direction,
\item $|A\cap\Pi(a',p_i)|=p_jp_k$ for any $a'\in A$,
\item $p_i<p_k$.
\end{itemize}
\end{lemma}
\begin{proof}
We first note that the claim $\calk\subseteq\calj$ follows from Corollary \ref{subset} and Lemma \ref{nodoubleinters}, while the second bullet point follows from Lemma \ref{saturatingplanes} (i).
Next, let $a_i\in \cali$. Then $\cali\cap\Pi(a_i,p_i)$ is nonempty and, by Lemma \ref{nodoubleinters}, disjoint from $\calj\cup\calk=\calj$. By Lemma \ref{planesubsetcalk} (ii) with $j$ and $k$ interchanged, it follows that $p_i<p_k$.

It remains to prove the first point. We have $\Phi_{N_j}|B$, hence  by Lemma \ref{BfiberedAfibered}, $B$ is $N_j$-fibered on each $D(N_j)$-grid in one of the $p_i$ and $p_j$ directions.
Since $p_j<p_i$, it follows that $\bbB^{N_j}_{N_j/p_i}[b]<\phi(p_i)$ for all $b\in B$, hence $B$ must be $N_j$-fibered in the $p_j$ direction.
\end{proof}

\begin{lemma}\label{Phi_N_jnmidAPhi_N_knmidB}
Assume (F2') and (\ref{pinotthesmallest}). Assume further that
$\Phi_{N_j}\nmid A$, $\Phi_{N_k}\nmid B$, and $p_j=\min_\nu p_\nu$, 
Then $\cali= \emptyset$, and $A$ is $M$-fibered in the $p_j$ direction. 
\end{lemma}

\begin{proof}
Assume, by contradiction, that $\cali\neq \emptyset$. By Lemma \ref{noPhi_Ni}, $\Phi_{N_i}|B$. 

Assume first that $N_i/p_i\notin \Div(B)$. By Lemma \ref{gen_top_div_mis}, 
$B$ is fibered on $D(N_i)$-grids in one of the $p_j$ and $p_k$ direction.
That, however, contradicts Lemma \ref{BnoN_ip_jfiber}.

Suppose now that $b,b'\in B$ with $(b-b',N_i)=N_i/p_i$. By Lemma \ref{calinonempty}, $B$ is $N_j$-fibered in the $p_j$ direction. This together with
(\ref{Bplanerest}) implies that $|B\cap\Pi(b,p_k)|> p_ip_j$. By Corollary \ref{planegrid}, we have $\Phi_{p_k^2}|B$, hence $\Phi_{p_k}|A$. The latter, in turn, implies
\begin{equation}\label{pipjdystrybucja}
|A\cap\Pi(a',p_k)|= p_ip_j\ \ \forall a'\in A.
\end{equation}
On the other hand, let $a\in \calj\cap\calk$ with $a*F_j*F_k\subset A$ as provided by (\ref{Aplanegrid}). Then
$$
|A\cap\Pi(a,p_k)|\geq p_jp_k> p_ip_j,
$$
where at the last step we used Lemma \ref{calinonempty} again. 
This contradicts (\ref{pipjdystrybucja}).

This proves that $\cali=\emptyset$. By the second claim in Lemma \ref{calinonempty}, we have $\calk\subseteq\calj$, hence $A$ is $M$-fibered in the $p_j$ direction. 
\end{proof}

\begin{lemma}\label{lastoneformixedcase}
Assume (F2') and (\ref{pinotthesmallest}). Assume further that
$\Phi_{N_j}\nmid A$, $\Phi_{N_k}\nmid B$, and $p_k =\min_\nu p_\nu$.
Then $\cali=\emptyset$ and $ \calk\subseteq\calj $. Consequently, $A$ is $M$-fibered in the $p_j$ direction. 
\end{lemma}

\begin{proof}
The proof splits between two cases.

\medskip

{\it Case 1:} $p_k<p_i<p_j$.
In this case, by Lemma \ref{p_kmincalksubcalj} we have $\calk\subseteq\calj$. 
Assume, by contradiction, that $\cali$ is nonempty, and let 
$a_i\in\cali$. 
By Lemma \ref{saturatingplanes} (i), we have $|A\cap\Pi(a_i,p_i)|=p_jp_k$. By Lemma \ref{planesubsetcalk} (ii) with $j$ and $k$ interchanged,
it follows that $p_k>p_i$, contradicting our assumption. Therefore $\cali=\emptyset$.

\medskip

{\it Case 2:} $p_k<p_j<p_i$.
We first note that if $\cali$ is nonempty, then $\bbB^{N_j}_{N_j/p_i}[b]<\phi(p_i)$ for all $b\in B$. Since $\Phi_{N_j}|B$, by Lemma \ref{BfiberedAfibered} 
\begin{equation}\label{Bp_jfibered}
B \text{ must be } N_j \text{-fibered in the } p_j \text{ direction}.
\end{equation}

Next, we follow the first part of the proof of Lemma \ref{p_kmincalksubcalj}.
(This part does not use the Lemma \ref{p_kmincalksubcalj} assumption that $p_i<p_j$.) By the same argument as there, $\calk$ must be $N_k$-fibered on every $D(N_k)$-grid in either the $p_j$ or the $p_k$ direction, and for every $a_k\in\calk$ we have one of the following:
\begin{itemize}
\item (\ref{KJinter}) holds, hence $a_k\in \calj\cap\calk$ and $a_k*F_j*F_k=A\cap \Pi(a_k,p_i)$,
\item (\ref{fullkline}) and (\ref{fullklinecount}) hold for all $a'\in\calk\cap\Lambda(a_k,D(N_k))$.
\end{itemize}

If (\ref{KJinter}) holds for all $a'\in \calk$, then $\calk\subseteq\calj$, and it follows by the same argument as in Case 1 that $\cali=\emptyset$.

We now prove that the second case is impossible. Indeed,
assume by contradiction that there exists $a_k\in\calk$ such that (\ref{fullkline}) and (\ref{fullklinecount}) hold for all $a'\in\calk\cap\Lambda(a_k,D(N_k))$. We first claim that
\begin{equation}\label{nicniema2}
\cali\cap\Pi(a_k,p_i)=\emptyset.
\end{equation}
Indeed, if (\ref{nicniema2}) fails, we continue as in the proof of Lemma \ref{p_kmincalksubcalj} and find an element $a_i\in \cali$ at distance $M/p_j^2$ from the fiber chain through $a_k$. But (\ref{Bp_jfibered}) implies $M/p_j^2\in \Div(B)$, which is a contradiction. Hence (\ref{nicniema2}) holds.

By (\ref{podzial2}), we have $|A\cap\Pi(a_k,p_i)|=p_jp_k$. 
It follows from (\ref{nicniema2}) and Lemma \ref{planesubsetcalk} (i) that $A\cap\Pi(a_k,p_i)\subset\calk$. As in the proof of Lemma \ref{p_kmincalksubcalj}, all elements of $A$ in $\Pi(a_k,p_i)$ must satisfy (\ref{fullkline}). Hence 
$$
p_jp_k=|A\cap\Pi(a_k,p_i)|=cp_k^2,
$$
for some positive integer $c$, so that $p_k$ divides $p_j$, a contradiction. This completes the proof of the lemma.
\end{proof}

We now address the case in which $\Phi_{N_\nu}|B$ for $\nu\in \{j,k\}$.

\begin{lemma}\label{calknointer}
Assume (F2') and (\ref{pinotthesmallest}). Assume further that
 $B$ is $N_k$-fibered in the $p_k$ direction, $M/p_j^2\in \Div(B)$, and that $\cali\neq\emptyset$. Then for all $a_i\in\cali$ we have $\calk\cap \Pi(a_i,p_i)=\emptyset$.
\end{lemma}

\begin{proof}
Assume, by contradiction, that $a_i\in \cali$ and $\calk\cap\Pi(a_i,p_i)\neq\emptyset$. Replacing $a_i$ by a different element of $a_i*F_i$ if necessary, we may further assume that $p_i^{2}|a_i-a_k$ for some $a_k\in\calk\cap\Pi(a_i,p_i)$
Moreover, it follows from the fibering assumption on $B$ that the pair $(A,B)$ has a (1,2)-cofibered structure in the $p_k$ direction, with the cofiber in $A$ rooted at $a_k$.
	
Suppose that $(a_i-a_k,M)=M/p_jp_k^2$. Applying Lemma \ref{fibershift}, we could then shift the cofiber $a_k*F_k$ in the $p_k$ direction, obtaining a new T2-equivalent tiling $A'\oplus B=\ZZ_M$ in which the shifted cofiber $a'_k*F_k$ satisfies 
$(a_i-a_k,M)=M/p_j$. 
We claim that $A'$ contains a $p_j$ extended corner structure. Indeed, by Lemma \ref{nodoubleinters} we have 
$\Lambda(a_i,D(M)) \cap(\calj\cup\calk)=\emptyset$. Hence $A\cap(a_i*F_i*F_k)$ is $M$-fibered in the $p_i$ direction but not in the $p_k$ direction, and $A\cap(a'_k*F_i*F_k)$ must be empty, so that $A'\cap(a'_k*F_i*F_k)=a'_k*F_k$.
This proves the claim.
However, Theorem \ref{cornerthm} now implies that $\Phi_{p_i^2}|A$, 
contradicting (F2').

Since $M/p_k^2\in \Div(B)$ by the fibering assumption, we are now left with $(a_i-a_k,M)=M/p_j^2p_k^2$. But then, by the same fiber-shifting argument as above, we get a T2-equivalent tiling $A''\oplus B=\ZZ_M$, where $M/p_j^2\in \Div(A'')$. This contradicts the assumption that $M/p_j^2\in \Div(B)$.
\end{proof}

\begin{corollary}\label{p_nuorder}
Assume (F2') and (\ref{pinotthesmallest}).
 If $B$ is $N_\nu$-fibered in the $p_\nu$ direction for both $\nu=j$ and $\nu=k$, then $\cali=\emptyset$. 
\end{corollary}
\begin{proof}
By the fibering assumption, 
$$
M/p_j^2,M/p_k^2\in \Div(B).
$$
Suppose that $\cali\neq\emptyset$, and let $a_i\in \cali$.
It follows from Lemma \ref{calknointer} that $\calj$ and $\calk$ are both disjoint from $\Pi(a_i,p_i)$.  
Therefore $A\cap\Pi(a_i,p_i)\subset\cali$, and, in particular, $p_i$ divides $|A\cap\Pi(a_i,p_i)|$. But this contradicts
(\ref{podzial2}). 
\end{proof}

\begin{lemma}\label{p_ismallest}
Assume (F2') and (\ref{pinotthesmallest}). Assume further that
$\Phi_{N_j}\Phi_{N_k}|B$. Then $\cali=\emptyset$.
\end{lemma}
\begin{proof}
Suppose that $\cali\neq\emptyset$. Without loss of generality, we may assume that $p_k=\min_\nu p_\nu$.
By Lemma \ref{BfiberedAfibered}, $B$ is $N_k$-fibered on $D(N_k)$-grids in one of the $p_i$ and $p_k$ directions. However, $B$ cannot be $N_k$-fibered in the $p_i$ direction on any $D(N_k)$-grid, since the assumptions that $\cali\neq\emptyset$ and $p_k<p_i$ imply that $\bbB^{N_k}_{N_k/p_i}[b]<\phi(p_i)$ for all $b\in B$.
Hence
\begin{equation}\label{BN_kfiber}
B \text{ must be } N_k \text{-fibered in the } p_k \text{ direction}.
\end{equation}
By the same argument as above, if $p_j<p_i$, then $B$ must also be $N_j$-fibered in the $p_j$ direction, and an application of  Corollary \ref{p_nuorder} concludes the proof.
	
It remains to consider the case when 
$$
p_k < p_i < p_j.
$$
We claim that in this case,
\begin{equation}\label{jeszczejeden}
M/p_ip_k\in\Div(B).
\end{equation}
Indeed, since $\calk\neq \emptyset$, the failure of (\ref{jeszczejeden}) would imply that $N_i/p_k\not\in\Div_{N_i}(B)$. Since $\cali\neq \emptyset$, Lemma \ref{noPhi_Ni} implies $\Phi_{N_i}| B$.
It would then follow from Lemma \ref{gen_top_div_mis} that $B$ is $N_i$-fibered on $D(N_i)$-grids. By 
Lemma \ref{BnoN_ip_jfiber}, $B$ can only be $N_i$-fibered in the $p_i$ direction. 
In particular, $\Phi_{p_i}|B$, contradicting Lemma \ref{saturatingplanes} (i).

Let $b_1,b_2\in B$ with 
\begin{equation}\label{M/p_ip_kelements}
(b_1-b_2,M)=M/p_ip_k.
\end{equation}
By (\ref{Bplanerest}), one may find $b_3,\ldots, b_{p_i}\in B$ satisfying $(b_\nu-b_{\nu'},M)\in \{M/p_ip_k,M/p_ip_j,M/p_ip_jp_k\}$ for all $\nu\in\{1,\dots,p_i\}$ with $\nu\neq\nu'$. 
	
By Lemma \ref{BfiberedAfibered}, $B$ is $N_j$-fibered on each $D(N_j)$-grid in one of the $p_i$ and $p_j$ directions. 
On the other hand, we have
\begin{align*}
p_i& =|B\cap \Lambda(b_1,D(M))| 
\geq \bbB^{N_j}_{N_j/p_i}[b_1] + \bbB^M_{M/p_ip_k}[b_1]
\\
&\geq \bbB^{N_j}_{N_j/p_i}[b_1] +1,
\end{align*}
where the first equality follows from (\ref{Adivintersec}) and from (\ref{Bplanerest}) applied to $b_2,\dots,b_{p_i}$, and the last inequality follows from 
(\ref{M/p_ip_kelements}).
In particular, $\bbB^{N_j}_{N_j/p_i}[b_1]<\phi(p_i)$, so that $B$ cannot be $N_j$-fibered in the $p_i$ direction on the grid $\Lambda:=\Lambda(b_1, D(N_j))$.
It follows that $B$ is $N_j$-fibered in the $p_j$ direction on $\Lambda$. 

Taking also (\ref{BN_kfiber}) into account, we see that 
\begin{equation}\label{podwojne}
|B\cap\Pi(b_\nu,p_i^2)|\geq p_jp_k
\text{ for all } \nu\in\{1,2,\ldots,p_i\}.
\end{equation}
Since $|B|=p_ip_jp_k$, (\ref{podwojne}) must in fact hold with equality for each $\nu$, and 
\begin{equation}\label{plaszczyzny}
B\subset\Pi(b_1,p_i).
\end{equation}

Now, let $b\in B$ be arbitrary. We claim that
\begin{equation}\label{jeszczejeden2}
M/p_ip_k\in\Div(B\cap \Lambda(b,D(N_i))) 
\end{equation}
To prove this, we start by arguing as in the proof of (\ref{jeszczejeden}) that if (\ref{jeszczejeden2}) fails, then $B$ is $N_i$-fibered in the $p_i$ direction on $ \Lambda(b,D(N_i))$. However, if that were the case, then we would have $M/p_i^2\in\Div(B)$, contradicting (\ref{plaszczyzny}). 
	
We further note that by (\ref{plaszczyzny}), $B\cap \Lambda(b,D(N_i))= B\cap \Lambda(b,D(M))$, so that in fact we have
$$
M/p_ip_k\in\Div(B\cap \Lambda(b,D(M)) \ \ \forall b\in B.
$$
With this in place, we repeat the argument starting with (\ref{M/p_ip_kelements}) to prove that $B$ is $N_j$-fibered in the $p_j$ direction on all $D(N_j)$-grids.
	
It follows that $B$ is $N_\nu$-fibered in the $p_\nu$ direction for both $\nu=j$ and $\nu=k$. 
By Corollary \ref{p_nuorder}, we have
 $\cali=\emptyset$ as claimed.
\end{proof}


\subsection{Proof of Proposition \ref{F3structure}}


In this section, we are working under the following assumption.

\medskip\noindent
\textbf{Assumption (F3):} We have $A\oplus B=\ZZ_M$, where $M=p_i^{2}p_j^{2}p_k^{2}$ is odd. Furthermore,  $|A|=|B|=p_ip_jp_k$, $\Phi_M|A$, $A$ is fibered on $D(M)$-grids, $\cali=\emptyset$, and
\begin{equation}\label{symdif}
	\text{the sets } \calj\setminus\calk \text{ and } \calk\setminus\calj \text{ are nonempty}.
\end{equation}


The proof below works regardless of whether $\calj$ and $\calk$ are disjoint or not. If $\calj\cap\calk\neq\emptyset$, then
(since $\cali=\emptyset$) any element $a\in\calj\cap\calk$ must satisfy the conditions of Lemma \ref{fiberedstructure} (i), so that 
$$
A\cap\Pi(a,p_i^{2})=a*F_j*F_k.
$$
It follows that the set $\calj\setminus\calk$ is $M$-fibered in the $p_j$ direction, and 
$\calk\setminus\calj$ is $M$-fibered in the $p_k$ direction.

We begin with the case when at least one of $\Phi_{M_j}$ and $\Phi_{M_k}$ divides $A$.

\begin{lemma}\label{Kp_i^2plane} 
Assume (F3), and that $\Phi_{M_k}|A$. Then 
\begin{equation}\label{Phi_p_j^2|A}
\Phi_{p_j^2}|A.
\end{equation}
Furthermore, $\calk$ is $M_k$-fibered in the $p_j$ direction, so that for every $a_k\in\calk$ we have
\begin{equation}\label{KM/p_k^2jfiber}
\mathbb{K}^{M_k}_{M_k/p_j}[a_k]=p_k\cdot\phi(p_j).
\end{equation}
and 
\begin{equation}\label{Kp_i^2p_jplane}
A\cap\Pi(a_k,p_i^2)\subset\Lambda(a_k,p_i^2p_j).
\end{equation}
The same holds with $p_k$ and $p_j$ interchanged.
\end{lemma}

\begin{proof}
Notice first that the fibering statement holds trivially for all $a_k\in \calk\cap\calj$. 

Let now $\call:=\calk\setminus\calj$, with the corresponding $N$-boxes $\mathbb{L}^N$ for $N|M$. 
The assumption $\Phi_{M_k}|A$ implies, by (\ref{fibercyc2}), that $\Phi_{M_k}|\call$. By Lemma \ref{2d-cyclo}, $\call$ mod $M_k$ is a disjoint union of $M_k$-fibers in the $p_i$ and $p_j$ direction. Hence, any element
 $a_0\in \call$ which does not belong to an $M_k$-fiber in the $p_j$ direction must satisfy
\begin{equation}\label{KM/p_k^2ifiber}
\mathbb{L}^{M_k}_{M_k/p_i}[a_0]=p_k\cdot\phi(p_i)
\end{equation}
and $|A\cap \Pi(a_0,p_j^2)|\geq p_ip_k$. By Lemma \ref{planebound},
$$
|A\cap \Pi(a_0,p_j^2)|= p_ip_k,
$$
in particular
\begin{equation}\label{Ksubset}
A\cap \Pi(a_0,p_j^2)\subset \call \text{ and } \calj\cap \Pi(a_0,p_j^2)=\emptyset.
\end{equation}

We first prove (\ref{Phi_p_j^2|A}). Indeed, assume for contradiction that $\Phi_{p_j}|A$. Then
\begin{equation}\label{maxjplane2}
|A\cap \Pi(a,p_j)|= p_ip_k \text{ for all } a\in A.
\end{equation}
Let $a_j\in \calj$, and consider the plane system $\Pi(a_j,p_j)$. These planes cannot contain any elements $a\in\call$ satisfying (\ref{KM/p_k^2ifiber}), since any such element would belong to an $M_k$-fiber in the $p_i$ direction in $\call$, of cardinality $p_ip_k$  and contained in $\Pi(a_j,p_j)$, and this would leave no room for the additional element $a_j\not\in\call$.  

Thus every element of $A\cap\Pi(a_j,p_j)$ must either belong to $\calj$, or else it must be an element of $\call$ 
belonging to an $M_k$-fiber in the $p_j$ direction in $\call$, of cardinality $p_jp_k$. Since the two sets are disjoint, 
(\ref{maxjplane2}) implies that
$$
p_ip_k= |A\cap \Pi(a,p_j)|= c_1 p_j + c_2 p_j p_k,
$$
where $c_1,c_2$ are nonnegative integers. 
But then $p_j$ divides either $p_i$ or $p_k$, a contradiction.

Therefore $\Phi_{p_j^2}|A$. Suppose now that there actually exists an element $a_0\in\call$ such that (\ref{KM/p_k^2ifiber}) holds. Then $A\subset \Pi(a_0,p_j)$. Since $\calj$ is nonempty, it follows that $\calj$ must intersect $\Pi(a_0,p_j^2)$, contradicting (\ref{Ksubset}). This proves the fibering conclusion of the lemma.
\end{proof}

\begin{lemma}\label{splitting}
Assume (F3).
The following holds true:
\begin{enumerate}
\item [(i)] If $\Phi_{p_i^2}\Phi_{M_k}|A$, then $A\subset\Pi(a,p_i)$ for any $a\in A$ (hence $A$ is contained in a coset of $p_i\ZZ_M$).
\item [(ii)] If $\Phi_{p_i}|A$, then $|A\cap\Pi(a,p_i)|=p_jp_k$ for all $a\in A$.
Moreover, for every $a\in A$ we have either $A\cap \Pi(a,p_i)\subset \calj$ or $A\cap \Pi(a,p_i)\subset \calk$.  \end{enumerate}
\end{lemma}
\begin{proof}
For part (i), by Lemma \ref{Kp_i^2plane} every $a_k\in\calk$ satisfies (\ref{KM/p_k^2jfiber}). By Lemma \ref{planebound}, 
\begin{equation}\label{maxp_iplane}
|A\cap\Pi(a_k,p_i^{2})|=p_jp_k \text{, hence } A\cap\Pi(a_k,p_i^2)\subset \calk.
\end{equation}
If $\Phi_{p_i^{2}}|A$, then $A\subset\Pi(a_k,p_i)$. This proves the first part of the lemma. 

For part (ii), assume that $\Phi_{p_i}|A$. Then $|A\cap\Pi(a,p_i)|=p_jp_k$ for all $a\in A$. 
The second part follows from Lemma \ref{planesubsetcalk} (i) with $\alpha_i=1$. 
\end{proof}

\begin{lemma}\label{Phi_N_j|B}
Assume (F3) and 
\begin{equation}\label{Phi_{p_i}|Aass}
\Phi_{p_i}\Phi_{M_k}|A.
\end{equation}
Then $\Phi_{N_j}\nmid A$. 
\end{lemma}
\begin{proof}
Assume, by contradiction, that $\Phi_{N_j}|A$. By (\ref{symdif}) and Corollary \ref{subset}, we must have $p_j=\min_\nu p_\nu$. We first claim that
\begin{equation}\label{fiberatti}
\hbox{if }\Phi_{N_j}|A \hbox{, then }\calj \hbox{ must be }N_j\hbox{-fibered on }D(N_j) \hbox{-grids.}
\end{equation}
Indeed, suppose that (\ref{fiberatti}) fails. By Lemma \ref{Kunfibered}, we have 
\begin{equation}\label{Jfiberdiv}
\{D(M)|m|M\}\subset \Div(A).
\end{equation}
Applying Corollary \ref{subset} and (\ref{symdif}) again in the $p_k$ direction, we get $\Phi_{N_k}|B$. 
By (\ref{Jfiberdiv}) and Lemma \ref{fiberfibering} (ii), $B$ is $N_k$-fibered in the $p_k$ direction, implying a (1,2)-cofibered structure for $(A,B)$ with all fibers in $\calk$ as cofibers. 

Fix $a_k\in \calk$, and recall that it must satisfy (\ref{KM/p_k^2jfiber}). This produces a family of $M$-fibers in the $p_k$ direction, all contained in $A\cap\Pi(a_k,p_i^2)$. 
Using Lemma \ref{fibershift} to shift and align these fibers if necessary, we get a T2-equivalent set $A'$ such that
$A'\oplus B=\ZZ_M$ and $a_k*F_j*F_k\subset A'$. 
By T2-equivalence, $\Phi_{p_i}|A'$, and by Lemma \ref{planebound},
\begin{equation}\label{fiberatti2}
a_k*F_j*F_k = A'\cap \Pi(a_k,p_i^{2}).
\end{equation} 

Since $\cali=\emptyset$, $a_k$ cannot belong to an $M$-fiber in the $p_i$ direction, and in particular there exists an $x\in \ZZ_M\setminus A$ with $(a_k-x,M)=M/p_i$. We have $A'_{x}\subset \Pi(a_k,p_i^{2})\cup \Pi(x,p_i^{2})$. However, if $A'_{x}\cap\Pi(x,p_i^{2})\neq\emptyset$, then $|A'\cap\Pi(x,p_i^{2})|>0$. By Corollary \ref{planegrid} and
(\ref{maxp_iplane}),  we have $\Phi_{p_i^{2}}|A'$, which contradicts the fact that $\Phi_{p_i}|A'$. 

Thus $A'_x\subset \Pi(a_k,p_i^{2})$, and by (\ref{fiberatti2}), for any $b\in B$ we have
$$
1=\sum_{m\in \{1,p_j,p_k,p_jp_k\}}\frac{1}{\phi(mp_i)}\bbA'_{M/mp_i}[x|\Pi(a_k,p_i^{2})]\bbB_{M/mp_i}[b].
$$
Hence $\{D(M)|m|M\}\cap \Div(B)$ is nontrivial, contradicting (\ref{Jfiberdiv}). This proves (\ref{fiberatti}).

By (\ref{symdif}), we may find $a_j\in \calj\setminus\calk$. It follows from Lemma \ref{fiberfibering} (i) that $A\cap\Lambda(a_j,D(N_j))$ can only be 
$N_j$-fibered in the $p_j$ direction. Recall from Lemma \ref{splitting} that $|A\cap\Pi(a_j,p_i)|=p_jp_k$ and $A\cap\Pi(a_j,p_i)\subset\calj\setminus\calk$. But the fibering in $\calj\setminus\calk$ implies that
$$
p_jp_k=c\cdot p_j^2,
$$
thus $p_j$ must divide $p_k$, which is clearly false. The lemma follows.
\end{proof}

\begin{lemma}\label{Phi_M/p_ip_k^2}
Assume (F3) and (\ref{Phi_{p_i}|Aass}).
Then $\Phi_{M_k/p_i}|A$.
\end{lemma} 
\begin{proof}
Denote $N=M_k/p_i$, and write 
$$
A(X)=\calj(X)-(\calj\cap\calk)(X)+\calk(X).
$$
By Lemma \ref{fiberedstructure} (i) and (\ref{fibercyc}), $\Phi_N$ divides both $\calj$ and $\calj\cap\calk$. 
Hence it suffices to prove that $\Phi_N|\calk$.

Let $a_k\in\calk$. By Lemma \ref{Kp_i^2plane}, we have $\bbK^{M_k}_{M_k}[a_k]=\bbK^{M_k}_{M_k}[x_j]=p_k$ for all $x_j\in \ZZ_M$ with $(a_k-x_j,M)=M/p_j$. In particular, $|\calk \cap \Pi(a_k,p_i^2)|=p_jp_k$, and, since $\Phi_{p_i}|A$, there are no other elements of $A$ in $\Pi(a_k,p_i)$. 
It follows that, for any $x_j$ as above,
\begin{equation}\label{balancedNvertices}
\bbK^N_N[a_k]=\bbK^N_N[x_j]=p_k
\end{equation}

We need to prove that $\bbK^N_N [\Delta]=0$ for every $N$-cuboid $\Delta$. It suffices to check this under the assumption that at least one vertex of $\Delta$ belongs to $\calk$, so that two of its vertices are at points $a_k$ and $x_j$ as above. The other two vertices are at $x,x'\in \ZZ_M$ with 
$$
(a_k-x,N)=(x_j-x',N)=N/p_i \text{ and } (x-x',N)=N/p_j .
$$ 
By (\ref{balancedNvertices}), the cuboid face containing $a_k$ and $x_j$ is balanced. Consider now the face containing $x$ and $x'$. By Lemma \ref{splitting} (ii) we need to consider two cases. 
If $A\cap\Pi(x,p_i)\subset\calj$, then this face must be balanced on the scale $N$, since $\Phi_N|\calj$. Otherwise, we must have $A\cap\Pi(x,p_i)\subset\calk$, and then by the same argument as above, either $\mathbb{K}^N_N[x]=\mathbb{K}^N_N[x']=p_k$ or $\mathbb{K}^N_N[x]=\mathbb{K}^N_N[x']=0$. In both cases, the cuboid is balanced, which proves the lemma.
\end{proof}

\begin{lemma}\label{M_kBdivisors}
Assume (F3) and (\ref{Phi_{p_i}|Aass}).
Then 
$$\Phi_d|B \ \ \forall d\in\{N_j,N_k,M/p_jp_k, M/p_j^2p_k, M/p_jp_k^2\}.$$
\end{lemma}

\begin{proof}
For $d=N_j$, this follows from Lemma \ref{Phi_N_j|B}. 

Assume, by contradiction, that $\Phi_{N_k}|A$. By Lemma \ref{fibercyc3}, we must have $\Phi_{N_k}|\calk$. As in the proof of Lemma \ref{Phi_M/p_ip_k^2}, we have $|\calk \cap \Pi(a_k,p_i^2)|=p_jp_k$ for all $a_k\in\calk$, with each line $\ell_k(x)$ for $x\in a_k*F_j$ containing an $M$-fiber in the $p_k$ direction. In particular,
\begin{equation}\label{fancykline}
|\calk\cap\ell_k(a_k)|=p_k \ \  \forall a_k\in\calk.
\end{equation}
On the other hand, let $a'_k\in\calk\setminus\calj$. Then $\Phi_{N_k}|\calk$ implies one of the following:
\begin{itemize}
\item $\calk$ is not fibered on $\Lambda(a'_k, D(N_k))$. By Lemma \ref{Kunfibered}, 
there exists an $x\in \ZZ_M$ such that $\mathbb{K}^{N_k}_{N_k/p_k}[x]=\phi(p_k^2)$.
\item $\calk$ is $N_k$-fibered on $\Lambda(a'_k, D(N_k))$. By Lemma \ref{fiberfibering} (i), it can only be fibered in the $p_k$ direction.
\end{itemize}
Both of these are clearly incompatible with (\ref{fancykline}). This proves that $\Phi_{N_k}|B$.

To prove that $\Phi_{M/p_jp_k}\nmid A$, let $a_k\in \calk\setminus\calj$, and consider $M/p_jp_k$-cuboids with  vertices at $a_k$ and $x\in\ell_k(a_k)$ with $(x-a_k,M)=M/p_k^2$.  
By (\ref{Kp_i^2p_jplane}) and Lemma \ref{splitting} (ii), we have $\bbA^{M/p_jp_k}_{M/p_jp_k}[v]=0$ for any cuboid vertex $v$ other than $a_k$ and $x$. Varying $x$ as above, we see that in order for all such cuboids to be balanced we must have
$$|A\cap\Pi(a_k,p_i^2)|=p_k \bbA^{M/p_jp_k}_{M/p_jp_k}[a_k].$$
But this is not possible, since the left side is equal to $p_jp_k$ and the right side is divisible by $p_k^2$.
The same argument, with the cuboids collapsed further to scale $M/p_j^2p_k$, proves that 
$\Phi_{M/p_j^2p_k}\nmid A$.

Finally, we prove that $\Phi_{M_k/p_j}\nmid A$. Consider any $M_k/p_j$-cuboid with one vertex at $a_k\in \calk\setminus\calj$. By the same argument as above, we have $\bbA^{M_k/p_j}_{M_k/p_j}[v]=0$ for all vertices $v\neq a_k$, hence the cuboid cannot be balanced.
\end{proof}

\begin{lemma}\label{M_kp_isubtile}
Assume (F3) and (\ref{Phi_{p_i}|Aass}). Then the conditions of Theorem \ref{subtile} are satisfied in the $p_i$ direction,
after interchanging $A$ and $B$.
\end{lemma}

\begin{proof}
By assumption, $\Phi_{p_i^{2}}|B$. We need to verify that 
for every $d$ such that $p_i^{n_i}|d|M$ and $\Phi_d\nmid B$, we have
\begin{equation}\label{laundrylist}
\Phi_{d/p_i^\alpha}|A,\ \ \alpha\in \{1,2\}.
\end{equation}
By Lemma \ref{M_kBdivisors}, it remains to check (\ref{laundrylist}) for $d\in\{M,M_j, M_k\}$.

\begin{itemize}
\item For $d=M$, (\ref{laundrylist}) follows from Lemma \ref{fibercyc3} since $\cali=\emptyset$.

\item For $d=M_k$, we have $\Phi_{M_k/p_i}\Phi_{p_j^2}|A$ by Lemmas \ref{Kp_i^2plane} 
and \ref{Phi_M/p_ip_k^2}.

\item For $d=M_j$, if $\Phi_{M_j}|B$, there is nothing to prove. If on the other hand $\Phi_{M_j}|A$, then
Lemma \ref{Kp_i^2plane} and Lemma \ref{Phi_M/p_ip_k^2} hold with $j$ and $k$ interchanged, and so (\ref{laundrylist}) also holds in this case.
\end{itemize}
\end{proof}

This resolves the case $\Phi_{M_k}|A$. The case $\Phi_{M_j}|A$ is similar, with $j$ and $k$ interchanged. It remains to prove Proposition \ref{F3structure} under the assumption that
\begin{equation}\label{cycassumption}
\Phi_{M_\nu}\nmid A \text{ for } \nu\in \{j,k\}. 
\end{equation}
Without loss of generality, we may also assume that 
\begin{equation}\label{porzadekalfabetyczny}
p_k>p_j.
\end{equation}
By Corollary \ref{subset}, this implies that 
$\Phi_{N_k}\Phi_{M_k}|B$. It follows that $B$ is $\calt$-null with respect to the cuboid type $\calt=(N_k,\vec{\delta},1)$, where $\vec{\delta}=(1,1,0)$. Since cuboids of this type are 2-dimensional, it follows by Lemma \ref{2d-cyclo} that for every $b\in B$ at least one of the following holds:
\begin{equation}\label{BTp_jfiber}
\bbB_M[y]+ \bbB_{M/p_k}[y]=1 \text{ for every } y\in \ZZ_M \text{ with } (b-y,M)=M/p_j,
\end{equation}
\begin{equation}\label{BTp_ifiber}
\bbB_{M}[y]+\bbB_{M/p_k}[y]=1 \text{ for every } y\in \ZZ_M \text{ with } (b-y,M)=M/p_i.
\end{equation}
In particular, this implies that
\begin{equation}\label{Btopdivisors}
\{D(M)|m|M\}\cap \Div(B)\neq \emptyset.
\end{equation}

\begin{lemma}\label{BPhi_{N_j}}
Assume (F3), (\ref{cycassumption}), and (\ref{porzadekalfabetyczny}). 
Then $\Phi_{p_i}|A$ and $\Phi_{N_j}\nmid A$.
\end{lemma}
\begin{proof}
We start with the second part. Assume for contradiction that
$\Phi_{N_j}|A$. By Lemma \ref{Kunfibered} applied to $p_j$, (\ref{Btopdivisors}), and Lemma \ref{fiberfibering} (i), 
the set $\calj\setminus\calk$ must be $N_j$-fibered in the $p_j$ direction. Let $a_j\in \calj\setminus\calk$.
We now consider two cases.
 
\begin{itemize} 
\item Suppose that $\Phi_{p_i}|A$.
By Lemma \ref{splitting} (ii), we have $A\cap\Pi(a_j,p_i)\subset\calj\setminus\calk$ and $|A\cap\Pi(a_j,p_i)|=p_jp_k$.
But then the fibering of $\calj\setminus\calk$ implies that $p_jp_k$ is divisible by $p_j^2$, a contradiction.

\item Assume now that $\Phi_{p_i^{2}}|A$. Let $A'$ be a translate of $A$ such that $a_j\in A'_{p_i}$. By the cyclotomic divisibility assumption, we have $|A'_{p_i}|=p_jp_k$. On the other hand, by the fibering properties of $A$, 
$$
p_jp_k=|A'_{p_i}|=c_jp_j^2+c_kp_k,\ \ c_j>0.
$$
Thus $c_j=p_kc'_j$ and $p_j=c'_jp_j^2+c_k$ with $c'_j>0$, a contradiction.
\end{itemize}

Therefore $\Phi_{N_j}\nmid A$. By (\ref{cycassumption}), we have $\Phi_{N_j}\Phi_{M_j}|B$. Applying the same argument as in (\ref{BTp_jfiber}), (\ref{BTp_ifiber}) to $p_j$ instead of $p_k$, we get that every $b\in B$ must satisfy at least one of
$$
\bbB_M[y]+ \bbB_{M/p_j}[y]=1 \text{ for every } y\in \ZZ_M \text{ with } (b-y,M)=M/p_k,
$$ 
\begin{equation}\label{BTjp_ifibered}
\bbB_{M}[y]+\bbB_{M/p_j}[y]=1 \text{ for every } y\in \ZZ_M \text{ with } (b-y,M)=M/p_i.
\end{equation}
But since $p_k>p_j$, if the former holds for some $b\in B$, we must have $M/p_k\in \Div(B)$, which is not allowed. Thus (\ref{BTjp_ifibered}) holds for all $b\in B$. Hence the assumptions of Lemma \ref{cyclo-grid} hold for $B$, with $m=M/p_ip_j$ and $s=p_i^2$. It follows that $\Phi_{p_i^{2}}|B$, and therefore $\Phi_{p_i}|A$.
\end{proof}

\begin{lemma}\label{noM_nusubtile}
Assume (F3), (\ref{cycassumption}), and (\ref{porzadekalfabetyczny}). 
Then the conditions of Theorem \ref{subtile} are satisfied in the $p_i$ direction,
after interchanging $A$ and $B$. 
\end{lemma}
\begin{proof}
We verify the conditions of Theorem \ref{subtile}. By (\ref{cycassumption}) and Lemma \ref{BPhi_{N_j}}, we have
$\Phi_d|B$ for $d\in\{p_i^2,M_j,M_k,N_j\}$. Next, we claim that
\begin{equation}\label{almostdone}
\Phi_d|B\hbox{ for }d\in\{M/p_jp_k,M/p_j^2p_k,M/p_jp_k^2\}.
\end{equation}
Since (\ref{BTjp_ifibered}) holds for all $b\in B$, we may write $B$ as 
$$
B(X)=B^i(X)+Q(X)(X^{M/p_j}-1)
$$
for some polynomial $Q(X)$, where $B^i$ is $M$-fibered in the $p_i$ direction. By (\ref{fibercyc}), we have $\Phi_d|B^i$ for all $p_i^2|d$. Using also that $\Phi_d|(X^{M/p_j}-1)$ for all $d|M/p_j$, we get (\ref{almostdone}).  

Finally, since $\Phi_M|A$, we need to prove that $\Phi_{M/p_i^\alpha}|A$ for $\alpha\in \{1,2\}$. Indeed, since $\cali$ is empty, this follows from Lemma \ref{fibercyc3}.
\end{proof}

This concludes the proof of Proposition \ref{F3structure}.


\section{Acknowledgement} 

Both authors were supported by NSERC Discovery Grants.
We are grateful to the anonymous referee for many helpful comments and suggestions.



\bibliographystyle{amsplain}

\begin{thebibliography}{99}


\bibitem{Bh}
S. Bhattacharya, {\it Periodicity and Decidability of Tilings of $\ZZ^2$}, Amer. J. Math. 142 (2020),
255–266.

\bibitem{CM} E. Coven, A. Meyerowitz, {\it Tiling the integers
with translates of one finite set}, J. Algebra 212 (1999),
161--174.

\bibitem{deB} N.G. de Bruijn, {\it On the factorization of cyclic groups}, Indag. Math. 15 (1953), 370--377.


\bibitem{dutkay-kraus}
D. E. Dutkay, I. Kraus, {\it On spectral sets of integers}, in: {\it Frames and Harmonic Analysis}, Contemporary Mathematics, vol. 706, edited by Y. Kim, A.K. Narayan, G. Picioroaga, and  E.S. Weber, American Mathematical Society, 2018,
pp. 215--234.

\bibitem{dutkay-lai}
D. E. Dutkay, C-K Lai, {\it Some reductions of the spectral set conjecture to integers}, Math. Proc. Cambridge Phil. Soc. 156 (2014), 123--135.


\bibitem{FKS} T. Fallon, G. Kiss, G. Somlai, {\it Spectral sets and tiles in $\ZZ^2_p\times \ZZ^2_q$},  J. Funct. Anal., 282(12) (2022), Article 109472.


\bibitem{FMV}
T. Fallon, A. Mayeli, D. Villano, {\it The Fuglede Conjecture holds in $\mathbb{F}^3_p$ for $p=5,7$}, Proc. Amer. Math. Soc., to appear, DOI: https://doi.org/10.1090/proc/14750




\bibitem{FMM}
B. Farkas, M. Matolcsi, P. M\'ora, {\it On Fuglede's conjecture and the existence of universal spectra},
J. Fourier Anal. Appl. 12(5) (2006), 483--494.

\bibitem{FR}
B. Farkas, S. G. R\'ev\'esz, {\it Tiles with no spectra in dimension 4}, Math. Scandinavica, 98 (2006), 44--52.




\bibitem{Fug} B. Fuglede, {\it Commuting self-adjoint partial differential operators
and a group-theoretic problem}, J. Funct. Anal. 16 (1974), 101--121.


\bibitem{GLW} A. Granville, I. {\L}aba, Y. Wang, {\it A characterization
of finite sets that tile the integers}, unpublished preprint, 2001, arXiv:math/0109127.

\bibitem{GL} R. Greenfeld, N. Lev {\it Fuglede's spectral set conjecture for convex polytopes}, Analysis \& PDE 10(6) (2017), 1497--1538.


\bibitem{GT} R. Greenfeld, T. Tao. {\it The structure of translational tilings in $\mathbb {Z}^ d$},  Discrete Analysis, 2021:16, 28 pp.


\bibitem{IKT}
A. Iosevich, N. Katz, T. Tao, {\it The Fuglede spectral conjecture holds for convex planar domains}, Math. Res. Lett. 10 (2003), 559--569.

\bibitem{IMP}
A. Iosevich, A. Mayeli, J. Pakianathan, {\it The Fuglede conjecture holds in $\ZZ_p \times \ZZ_p$}, Analysis \& PDE 10(4) (2017), 757--764.


\bibitem{KMSV} G. Kiss, R. D. Malikiosis, G. Somlai, M. Vizer, 
{\it On the discrete Fuglede and Pompeiu problems}, Analysis \& PDE 13 (3), (2020), 765-788.

\bibitem{KMSV2} G. Kiss, R. D. Malikiosis, G. Somlai, M. Vizer, {\it Fuglede's conjecture holds for cyclic groups of order $ pqrs$}, Discrete Anal. 2021:12, 24 pp.

\bibitem{KS} G. Kiss, S. Somlai, {\it Fuglede’s conjecture holds on $\ZZ^2_p\times\ZZ_q$}, Proc. Amer.
Math. Soc.,  149(10) (2021), 4181-4188.

\bibitem{KoLev} M.N. Kolountzakis, N. Lev, {\it Tiling by translates of a function: results and open problems}, Discrete Analysis, 2021:12, 24 pp.


\bibitem{KM}
M.N. Kolountzakis, M. Matolcsi, {\it Complex Hadamard matrices and the spectral set conjecture}, 
Collect. Math., Vol. Extra (2006), 281--291.

\bibitem{KM2}
M. N. Kolountzakis, M. Matolcsi, {\it Tiles with no spectra}, Forum Math. 18(3) (2006), 519--528.

\bibitem{KL}
S. Konyagin, I. {\L}aba, {\it Spectra of certain types of polynomials and
tiling the integers with translates of finite sets}, J. Number Theory
103 (2003), 267--280.


\bibitem{L} I. {\L}aba, {\it The spectral set conjecture and multiplicative
properties of roots of polynomials}, J. London Math. Soc. 65 (2002), 661--671.



\bibitem{LaLo1} I. {\L}aba, I. Londner, {\it Combinatorial and harmonic-analytic methods for integer tilings},
Forum of Mathematics - Pi (2022), Vol. 10:e8, 46 pp.


\bibitem{LaLo3} I. {\L}aba, I. Londner, {\it Splitting for integer tilings and the Coven-Meyerowitz tiling conditions},
preprint, 2022, arXiv:2207.11809




\bibitem{LS} J.C. Lagarias, S. Szab\'o, {\it Universal spectra and Tijdeman's
conjecture on factorization of cyclic groups}, J. Fourier Anal. Appl. 1(7) (2001), 63--70. 

\bibitem{LW1} J.C. Lagarias, Y. Wang, {\it Tiling the line with translates of
one tile}, Invent. Math. 124 (1996), 341--365.

\bibitem{LW2} J.C. Lagarias, Y. Wang, {\it Spectral sets and factorization of 
finite abelian groups}, J. Funct. Anal. 145 (1997), 73--98.


\bibitem{LL} T.Y. Lam and K.H. Leung, {\it On vanishing sums of roots of unity},
J. Algebra 224 (2000), 91--109.

\bibitem{LM} N. Lev, M. Matolcsi, {\it 	
	The Fuglede conjecture for convex domains is true in all dimensions},  Acta Math. 228 (2022), no. 2, 385-420.

\bibitem{M} R. D. Malikiosis, {\it On the structure of spectral and tiling subsets of cyclic groups},  Forum of Mathematics - Sigma (2022), Vol. 10:e23, 1-42.


\bibitem{MK} R. D. Malikiosis, M. N. Kolountzakis, {\it Fuglede's conjecture on cyclic groups of order 
$p^nq$}, Discrete Analysis, 2017:12, 16pp.

\bibitem{Mann} H. B. Mann, {\it On Linear Relations Between Roots of Unity}, Mathematika 12, Issue 2 (1965), 107--117.

\bibitem{matolcsi}
M. Matolcsi, {\it Fuglede's conjecture fails in dimension 4}, Proc. Amer. Math. Soc. 133(10) (2005), 3021--3026.

\bibitem{New} D.J. Newman, {\it Tesselation of integers}, J. Number Theory 9 (1977), 107--111.


\bibitem{Re1} L. R\'edei, {\it \"Uber das Kreisteilungspolynom}, Acta Math. Hungar. 5 (1954), 27--28.

\bibitem{Re2} L. R\'edei, {\it Nat\"urliche Basen des Kreisteilungsk\"orpers}, Abh. Math. Sem. Univ. Hamburg 23 (1959), 180--200.


\bibitem{Sands} A. Sands, {\it On Keller's conjecture for certain cyclic
groups}, Proc. Edinburgh Math. Soc. 2 (1979), 17--21.


\bibitem{schoen}
I. J. Schoenberg, {\it A note on the cyclotomic polynomial}, Mathematika 11 (1964), 131-136.

\bibitem{shi}
R. Shi, {\it Fuglede's conjecture holds on cyclic groups $\ZZ_{p^2qr}$
}, Discrete Analysis 2019:14, 14pp.

\bibitem{shi2}
R. Shi, {\it Equi-distribution on planes and spectral set conjecture on $\ZZ_{p^2}\times\ZZ_p$}, J. London Math. Soc. 102(2), (2020), 1030–1046.


\bibitem{somlai} G. Somlai, {\it Spectral sets in $\ZZ_{p^2qr}$ tile}, preprint, arXiv:1907.04398.

\bibitem{Steinberger} J. P. Steinberger, {\it Minimal vanishing sums of roots of unity with large coefficients},
Proc. London Math. Soc. 97 (3) (2008), 689--717.


\bibitem{Sz} S. Szab\'o, {\it A type of factorization of finite abelian groups},
Discrete Math. 54 (1985), 121--124.

\bibitem{Szabo-book} S. Szab\'o, {\it Topics in Factorization of Abelian Groups}, Hindustan Book Agency, 2004.


\bibitem{tao-fuglede}
T. Tao, {\it Fuglede's conjecture is false in 5 and higher dimensions}, Math. Res. Letters, 11 (2004), 251--258.


\bibitem{Tao-blog} T. Tao, {\it Some notes on the Coven-Meyerowitz conjecture} (blog post and discussion in comments), https://terrytao.wordpress.com/2011/11/19/some-notes-on-the-coven-meyerowitz-conjecture/


\bibitem{Tij} R. Tijdeman, {\it Decomposition of the integers as a direct sum of
two subsets}, in {\em Number Theory (Paris 1992--1993)}, London Math. Soc.
Lecture Note Ser., vol. 215, 261--276, Cambridge Univ. Press, Cambridge, 1995.

\bibitem{zhang} T. Zhang, {\it Fuglede's conjecture holds in $\ZZ_p\times \ZZ_{p^n}$}, preprint, arXiv:2109:08400.



\end{thebibliography}

\bigskip

\noindent{{\sc Laba:} Department of Mathematics, UBC, Vancouver,
B.C. V6T 1Z2, Canada}

\noindent{\it ilaba@math.ubc.ca}

\smallskip

\noindent{{\sc Londner:} Department of Mathematics, UBC, Vancouver,
B.C. V6T 1Z2, Canada. 
\\
{\it Current address:} Department of Mathematics, Faculty of Mathematics and Computer Science, Weizmann Institute of Science, Rehovot 7610001, Israel}

\noindent{\it itay.londner@weizmann.ac.il}

\medskip

\noindent{ORCID ID: {\L}aba 0000-0002-3810-6121, Londner 0000-0003-3337-9427}

\end{document}